\documentclass[final,1p,times]{elsarticle}

\graphicspath{{./Figures/}}

\usepackage{latexsym,amsfonts,amsmath,theorem,amssymb}
\usepackage{stmaryrd,array,tabularx}
\usepackage{pstricks,graphicx}
\usepackage{a4wide}
\usepackage{color}
\usepackage{caption}
\usepackage{theorem}
\usepackage{framed}
\usepackage{blindtext}
\usepackage{float}
\usepackage[caption=false]{subfig}
\usepackage[algoruled]{algorithm2e}

\usepackage[noend]{algpseudocode}
\usepackage{tikz}

\algrenewcommand\algorithmicrequire{\textbf{Input:}}
\algrenewcommand\algorithmicensure{\textbf{Output:}}

\newtheorem{theorem}{Theorem}[section]
\newtheorem{lemma}[theorem]{Lemma}
\newtheorem{proposition}[theorem]{Proposition}
\newtheorem{assumptions}[theorem]{Assumption}
\newtheorem{definition}[theorem]{Definition}

\newtheorem{remark}[theorem]{Remark}
\newtheorem{example}[theorem]{Example}
\newenvironment{proof}{\begin{trivlist}
		\item[\hskip\labelsep{\bf Proof.}]}{$\hfill\Box$\end{trivlist}}

\numberwithin{equation}{section}
\numberwithin{figure}{section}

\newcommand{\bsx}{{\boldsymbol{x}}}

\newcommand{\parame}{{\tau}}
\newcommand{\paramedet}{{\tau_0}}
\newcommand{\vxi}{\boldsymbol{\xi}}

\newcommand{\Pot}{{\Phi}}

\newcommand{\dHel}{\mathrm{d}_{\mathrm{Hel}}}
\newcommand{\smpl}{\mathrm{smp}}
\newcommand{\RB}{\mathrm{RB}}
\newcommand{\lin}{\mathrm{lin}}
\newcommand{\Nsto}{N_{\mathrm{sto}}}
\newcommand{\snap}{\mathrm{snap}}

\renewcommand{\hat}{\widehat}
\journal{Elsevier}
\renewcommand*{\appendixname}{}
\begin{document}

\begin{frontmatter}



\title{Fast sampling of parameterised Gaussian random fields}

\author[Ma]{Jonas Latz}
\ead{jonas.latz@ma.tum.de}
\author[AI]{Marvin Eisenberger}
\ead{marvin.eisenberger@in.tum.de}
\author[Ma]{Elisabeth Ullmann}
\ead{elisabeth.ullmann@ma.tum.de}

\address[Ma]{Chair of Numerical Analysis, Department of Mathematics, Technical University of Munich, Boltzmannstr. 3, 85748 Garching b.M., Germany}
\address[AI]{Chair for Computer Vision and Artificial Intelligence, Department of Computer Science, Technical University of Munich, Boltzmannstr. 3, 85747 Garching b.M., Germany}

\begin{abstract}
Gaussian random fields are popular models for spatially varying uncertainties, arising for instance in geotechnical engineering, hydrology or image processing.
A Gaussian random field is fully characterised by its mean function and covariance operator.
In more complex models these can also be partially unknown.
In this case we need to handle a family of Gaussian random fields indexed with hyperparameters.
Sampling for a fixed configuration of hyperparameters is already very expensive due to the nonlocal nature of many classical covariance operators. 
Sampling from multiple configurations increases the total computational cost severely.
In this report we employ parameterised Karhunen-Lo\`eve expansions for sampling.
To reduce the cost we construct a reduced basis surrogate built from snapshots of Karhunen-Lo\`eve eigenvectors.
In particular, we consider Mat\'ern-type covariance operators with unknown correlation length and standard deviation.
We suggest a linearisation of the covariance function and describe the associated online-offline decomposition.
In numerical experiments we investigate the approximation error of the reduced eigenpairs.
As an application we consider forward uncertainty propagation and Bayesian inversion with an elliptic partial differential equation where the logarithm of the diffusion coefficient is a parameterised Gaussian random field.
In the Bayesian inverse problem we employ Markov chain Monte Carlo on the reduced space to generate samples from the posterior measure.
All numerical experiments are conducted in 2D physical space, with non-separable covariance operators, and finite element grids with $\sim 10^4$ degrees of freedom.
\end{abstract}


\end{frontmatter}

\section{Introduction} \label{Sec:Introduction}

Many mathematical models in science and engineering require input parameters which are often not fully known or are perturbed by observational noise.
In recent years it has become standard to incorporate the noise or lack of knowledge in a model by using uncertain inputs.
In this work we are interested in models based on partial differential equations (PDEs) where the inputs are spatially varying random functions.
These so called random fields are characterised by a probability measure on certain function spaces.

We consider two typical tasks in uncertainty quantification (UQ): $(i)$ the forward propagation of uncertainty (forward problem) \cite{Ghanem2017, Smith2014}, and $(ii)$ the (Bayesian) inverse problem \cite{Dashti2017,Stuart2010}.
In $(i)$ we study the impact of uncertain model inputs on the model outputs and quantities of interest.
The mathematical task is to derive the pushforward measure of the model input under the PDE solution operator.
In $(ii)$ we update a prior distribution of the random inputs using observations; this gives the posterior measure.
Mathematically, this amounts to deriving the conditional measure of the inputs given the observations using a suitable version of Bayes' formula.
Unfortunately, in most practical cases there are no analytical expressions for either the pushforward or the posterior measure.
We focus on sampling based measure approximations, specifically Monte Carlo (MC) for the forward problem, and Markov chain Monte Carlo (MCMC) for the Bayesian inverse problem.
Importantly, MC and MCMC require efficient sampling procedures for the random inputs.

In this work we consider Gaussian random fields which are popular models e.g. in hydrology.
We recall the following sampling approaches for Gaussian random fields.
Factorisation methods construct either a spectral or Cholesky decomposition of the covariance matrix.
The major computational bottleneck is the fact that the covariance operator is often nonlocal, and a discretisation will almost always result in a dense covariance matrix which is expensive to handle.
Circulant embedding \cite{ChanWood:1997, Dietrich1997, Graham2018} relies on structured spatial grids and stationary covariance operators.
In this case, the covariance matrix can be factorised using the Fast Fourier Transform.
Alternatively, we can approximate the covariance matrix by hierarchical matrices and low rank techniques, see e.g. \cite{Benner2018, ChenStein:2017, DElia2013, Feischl2018, Khoromskij2009}.
A pivoted Cholesky decomposition is studied in \cite{Harbrechtetal:2015}.
More recently, the so called SPDE-based sampling has been developed in the works \cite{BolinKirchner:2017, BKK, Osbornetal:2017a, Osb2018, Roininenetal:2018}.
The major idea is to generate samples of Gaussian fields with Mat\'ern covariance operators by solving a certain discretised fractional PDE with white noise forcing. 
The Karhunen-Lo\`eve (KL) expansion \cite{Karhunen1947, Loeve1978} is another popular approach for sampling Gaussian random fields, however, it also suffers from the nonlocal nature of many covariance operators.
See e.g. the works \cite{Betzetal:2014, Rizzi2018, Khoromskij2009, PraneshGhosh:2015, Saibabaetal:2016, SchwabTodor:2006, ZhengDai:2017} for the efficient computation of the KL expansion.

Gaussian random fields are completely characterised by the mean function and covariance operator, and are thus simple models of spatially varying functions.
They are also flexible; depending on the regularity of the covariance operator it is possible to generate realisations with different degrees of smoothness.
However, in some practical situations the full information on the covariance operator might not be available, e.g. the correlation length, smoothness, and point-wise variance of the random field are not known.
Of course, these model parameters can be fixed a priori.
However, the posterior measure of a Bayesian inverse problem is often very sensitive to prior information.
We illustrate this in the following simple example.

\begin{example} \label{Example_introductory}
We consider a Gaussian random field with exponential covariance operator on the unit square $D=(0,1)^2$ and correlation length $\ell = 0.5$.
The goal is to estimate the statistics of this field given 9 noisy point observations within the framework of Bayesian inversion (cf. \S\ref{Subsec_BIP}).
The noise is centered Gaussian with a noise level about 1\%.
In Figure~\ref{Fig_Intro_Bayes} we plot a realisation of the true field together with the posterior mean and variance associated with four prior fields with a different (fixed) correlation length each.
The posterior mean and variance have been computed analytically.
Note that this is possible since $\ell$ is fixed, the prior and noise are Gaussian, and the forward response operator is linear.

\end{example}
\begin{figure}
\includegraphics[scale=0.65]{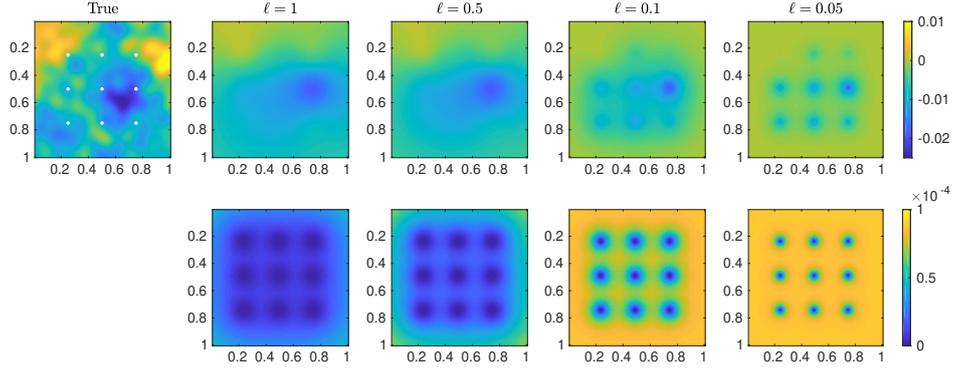}
\caption{Estimation of a Gaussian random field. 
The top-left figure shows a realisation of the true random field. 
The task is to estimate this field given 9 noisy point evaluations (white dots). 
The four top-right (bottom-right) figures show the posterior mean (pointwise posterior variance) for mean-zero Gaussian prior random fields with exponential covariance operator, standard deviation $\sigma=1$, and correlation lengths $\ell= 1, 0.5, 0.1, 0.05$.}
\label{Fig_Intro_Bayes}
\end{figure}

We make two observations in Figure~\ref{Fig_Intro_Bayes}.
First, we see that it is not possible to identify the true random field perfectly in this experiment.
This is due to the sparsity of the data; it is not a defect of the Bayesian inversion.
The posterior measure is well-defined, and has been computed analytically without a sample error in this experiment.

The second observation is the main motivation for our work.
We clearly see in Figure~\ref{Fig_Intro_Bayes} that the posterior measure depends crucially on the underlying prior measure and associated correlation length.
If the assumed correlation length is too small compared to the truth, then the posterior mean estimate is only accurate close to the observation locations.
If, on the other hand, the assumed correlation length is too large, we obtain an overconfident posterior mean estimate.
Inaccurate, fixed prior random field parameters can substantially deteriorate the estimation result in Bayesian inverse problems.
We treat this problem by modelling unknown \textit{hyperparameters} as random variables.

In statistics, a model with multiple levels of randomness is called a \textit{hierarchical model}.
In Bayesian statistics, this means that we work with parameterised prior measures (\textit{hyperpriors}).
Hierarchical models in forward and inverse UQ are discussed in \cite{Wikle2017}.
We also refer to \cite[Ch.10]{Robert2007} for general hierarchical Bayesian analyses.
Hierarchical Bayesian inverse problems with spatially varying random fields are discussed in \cite{BolinKirchner:2017, Dunlop2017, JiangOu:2017, Roininen2016, Sraj2016, TagadeChoi:2014}.
Hierachical models are also considered in the frequentist approach to inference, see e.g.  \cite{Haskard2007,Minden2017} for random field models and spatial statistics.

In our work we consider \textit{parameterised} Gaussian random fields where the covariance operator depends on random variables.
Notably, the resulting random field is not necessarily Gaussian, and we can thus model a larger class of spatial variations.
However, the greater flexibility of parameterised Gaussian fields brings new computational challenges as we explain next.

Assume that we discretise a Gaussian random field by a KL expansion.
The basis functions in this expansion are the eigenfunctions of the covariance operator.
For fixed, deterministic hyperparameters it is sufficient to compute the KL eigenpairs only once since the covariance operator is fixed.
However, changing the hyperparameters changes the covariance and often requires to re-compute the KL eigenpairs.
The associated computational cost and memory requirement scales at least quadratically in the number of spatial unknowns.
Hence it is often practically impossible to use uncertain hyperparameters in a (Gaussian) random field model in 2D or 3D physical space.
To overcome this limitation we suggest and study a reduced basis (RB) surrogate for the efficient computation of parameter dependent KL expansions.
In \cite{Sraj2016} the authors use a polynomial chaos surrogate for this task.
In contrast, our reduced basis surrogate approximates the KL eigenpairs by a linear combination of snapshot eigenpairs associated with certain fixed values of the hyperparameters.
Since this requires the solution of eigenproblems in a small number of unknowns (which is equal to the number of reduced basis functions) the computational cost can be reduced significantly.

Reduced basis methods were introduced in \cite{Noor1980}, and are typically used to solve PDEs for a large number of parameter configurations, see e.g. \cite{Book2, Book1}.
In contrast, parameterised eigenproblems have not been treated extensively with reduced basis ideas.
We refer to \cite{Fumagalli2016,Horger2016, Horger2017,Machiels2000,Sirkovic2016,Vallaghe2015} for applications and reviews of reduced basis surrogates for parameterised eigenproblems with PDEs.
For non-parametric KL eigenproblems we mention that reduced basis methods have been combined with domain decomposition ideas \cite{Rizzi2018}.
In this situation we need to solve several low-dimensional eigenproblems on the subdomains in the physical space.
It is then possible to construct an efficient reduced basis by combining the subdomain solutions, see \cite{Rizzi2018} for details.

The forward and inverse uncertainty quantification with PDE-based models typically requires many expensive model evaluations.
These evaluations are the basis but also the bottleneck in virtually any UQ method.
While we construct a surrogate to replace the expensive computation of KL eigenpairs, it is much more common in the UQ literature to construct surrogates for the underlying computational model.
See e.g. \cite{Farcas2018,Peherstorfer2016,Peherstorfer2018} 
for the forward problem and
\cite{Chen2015,Drohmann2015,Lieberman2010,Manzoni2016,
Mattis2018,Rubio2018,Stuart2018} for the Bayesian inverse problem.

We point out that many differential operators are local, and hence the associated discretised operator is often sparse.
Linear equations or eigenvalue problems with sparse coefficient matrices can often be solved with a cost that scales linearly in the number of unknowns.
In contrast, the reduced linear system matrix is often dense and the solution cost is at least quadratic in the number of unknowns.
Hence, for reduced basis methods to be cheaper compared to solves with the full discretised PDE operator it is necessary that the size of the reduced basis is not too large compared to the number of unknowns of the full system.
Notably, in most cases the discretised KL eigenproblem results in a dense matrix since the covariance integral operator is nonlocal.
Hence we expect a significant reduction of the total computational cost even if the size of the reduced basis is only slightly smaller than the number of unknowns in the unreduced eigenspace. 

The remainder of this report is organised as follows. 
In \S\ref{Sec:Background} we establish briefly the mathematical theory and computational framework for working with parameterised Gaussian random fields.
In \S\ref{Sec_RB_Param_EP} we propose and study a reduced basis surrogate for the parametric KL eigenpairs.
Specifically, we consider Mat\'ern-type covariance operators, and suggest a linearisation to enable an efficient offline-online decomposition within the reduced basis solver.
In \S\ref{Sec_RB_Sampling} we introduce \textit{reduced basis sampling}, and analyse its computational cost and memory requirements.
Finally, we present results of numerical experiments in \S\ref{Sec:Numerical_Experiments}.
We study the approximation accuracy and speed-up of the reduced basis surrogate, and illustrate the use of the reduced basis sampling for hierarchical forward and Bayesian inverse problems.

\section{Mathematical and computational framework}  \label{Sec:Background}
	To begin, we introduce the notation and some key elements, in particular, Gaussian and parameterised Gaussian measures.
Moreover, we discuss sampling strategies and their associated costs for Gaussian and parameterised Gaussian measures. 
We focus on samplers which make use of truncated KL expansions.

\subsection{Gaussian measures}
Let $(\Omega, \mathcal{A}, \mathbb{P})$ denote a probability space.
We recall the notion of a real-valued Gaussian random variable which induces a Gaussian measure on $\mathbb{R}$.
\begin{definition}
	The random variable $\vxi: \Omega \rightarrow \mathbb{R}$ follows a \textit{non-degenerate Gaussian measure}, if 
	\begin{equation*}
	\mathbb{P}(\vxi \leq \bsx) := \mathrm{N}(a, b^2)((-\infty, \bsx]) := \int_{-\infty}^{\bsx}\frac{1}{(2\pi)^{1/2}b} \exp\left(- \frac{(\bsx'-a)^2}{b^2}\right) \mathrm{d}\bsx', \ \ \ \bsx \in \mathbb{R},
	\end{equation*}
	for some $a \in \mathbb{R}$ and $b > 0$. 
	The Gaussian measure is \textit{degenerate}, if $b = 0$. 
	In this case, we define $\mathrm{N}(a, 0) := \delta_a$, the Dirac measure concentrated in $a$. 
	\end{definition}
Let $X$ denote a separable $\mathbb{R}$-Hilbert space with Borel-$\sigma$-algebra $\mathcal{B}X$. 
We now introduce Gaussian measures on $X$.
\begin{definition}\label{def:GRV}
The $X$-valued random variable $\theta: \Omega \rightarrow X$ has a Gaussian measure, if $T \theta$ follows a Gaussian measure for any $T \in X^*$ in the dual space of $X$. 
We write $\theta \sim \mathrm{N}(m, \mathcal{C})$, where
\begin{align*}
m &= \mathbb{E}[\theta] := \int \theta \mathrm{d}\mathbb{P}, \\
\mathcal{C} &= \mathrm{Cov}(\theta) := \mathbb{E}[(\theta-m) \otimes (\theta-m)].
\end{align*}
\end{definition}
In Definition~\ref{def:GRV} we distinguish two cases.
If $X$ is finite-dimensional, then we call $\theta$ a \textit{(multivariate) Gaussian random variable} with mean vector $m$ and covariance matrix $\mathcal{C}$.
If $X$ is infinite-dimensional, then $\theta$ is called \textit{Gaussian random field} with mean function $m$ and covariance operator $\mathcal{C}$. 

We now discuss two options to construct a Gaussian measure with given mean function $m \in X$ and covariance operator $\mathcal{C}: X \rightarrow X$.
While any $m \in X$ can be used as a mean function, we assume that $\mathcal{C}$ is a linear, continuous, trace-class, positive semidefinite and self-adjoint operator.
This ensures that $\mathcal{C}$ is a valid covariance operator, see \cite{Adler1990,Lifshits1995}.
We denote the set of valid covariance operators on $X$ by 
$\mathrm{CO}(X)$.

If $\dim X < \infty$, we can identify a Gaussian measure on $X$ in terms of a probability density function w.r.t. the Lebesgue measure.
\begin{proposition} \label{Propo_Gaussian_PDF}
Let $X := \mathbb{R}^N$, $m \in X$ and $\mathcal{C} \in \mathrm{CO}(X)$ with full rank. 
Then, the Gaussian measure can be written as
\begin{equation}
\mathrm{N}(m,\mathcal{C})(B) = \int_B \mathrm{n}(\theta; m, \mathcal{C}) \mathrm{d}\theta, \ \ \ B \in \mathcal{B}X
\end{equation}
where
\begin{equation}
\mathrm{n}(\theta; m,\mathcal{C})(B) = \det(2 \pi \mathcal{C})^{-1/2} \exp\left(-(1/2)\langle\theta-m, \mathcal{C}^{-1}(\theta-m)\rangle\right)
\end{equation}
is the associated probability density function.
\end{proposition}

If $\dim X = \infty$, we can construct a Gaussian measure on $X$ using the so-called \textit{Karhunen-Lo\`{e}ve (KL) expansion}.

\begin{definition} \label{Def_KL_exp} 
	Let $\dim X = \infty$ and let $(\lambda_i, \psi_i)_{i=1}^\infty$ denote the eigenpairs of $\mathcal{C}$, where $(\psi_i)_{i=1}^\infty$ form an orthonormal basis of $X$.
Let $\vxi: \Omega \rightarrow \mathbb{R}^{\infty}$ be a measurable function. 
Furthermore, let the components of $\vxi$ form a sequence  $(\xi_i)_{i=1}^\infty$  of independent and identically distributed (i.i.d.) random variables, where $\xi_1 \sim \mathrm{N}(0,1)$. 
Then, the expansion
\begin{equation*}
\theta_{\mathrm{KL}} :=
m + \sum_{i=1}^{\infty} \sqrt{\lambda_i} \xi_i \psi_i
\end{equation*}
is called KL expansion.  
\end{definition} 
One can easily verify the following proposition, see \cite{Karhunen1947,Loeve1978}.
\begin{proposition}
$\theta_{\mathrm{KL}}$ is distributed according to   $\mathrm{N}(m,\mathcal{C})$.
\end{proposition}
In the remainder of this paper we assume that all eigenpairs are ordered descendantly with respect to the absolute value of the associated eigenvalues.
For illustration purposes we give a example for a Gaussian measure on an infinite-dimensional, separable  Hilbert space.
\begin{example} \label{Exponential Covariance}
Let $X := \mathcal{L}^2(D; \mathbb{R})$, 
where $D \subseteq \mathbb{R}^d$ ($d = 1,2,3$) is open, bounded and connected. 
We define the \textit{exponential covariance operator} with correlation length $\ell > 0$ and standard deviation $\sigma > 0$ as follows,
\begin{align} \label{EQ_exponential_covariance}
\mathcal{C}_{\exp}^{(\ell, \sigma)}&: X \rightarrow X, \  \varphi \mapsto \int_D c_{\exp}^{(\ell, \sigma)}(\bsx, \cdot) \varphi(\bsx) \mathrm{d}\bsx, \\
c_{\exp}^{(\ell, \sigma)}&: D \times D \rightarrow \mathbb{R}, \ (\bsx,\bsx') \mapsto \sigma^2 \exp\left( -\ell^{-1}\mathrm{dist}(\bsx,\bsx') \right), \notag
\end{align}
where $\mathrm{dist}$ is the metric induced by the 2-norm.
\end{example}
We now briefly discuss the implications of $\ell$ and $\sigma$ on a (mean-zero) Gaussian measure with exponential covariance operator  $\mathrm{N}(0, \mathcal{C}_{\exp}^{(\ell, \sigma)})$ as in \eqref{EQ_exponential_covariance}.
The correlation length $\ell$ models the strength of the correlation between the random variables $\theta(\bsx)$ and $\theta(\bsx')$ in two points $\bsx, \bsx' \in D$ in the random field $\theta \sim \mathrm{N}(0, \mathcal{C}_{\exp}^{(\ell, \sigma)})$.
If $\ell$ is small, the marginals $\theta(\bsx)$ and $\theta(\bsx')$ are only strongly correlated if $\mathrm{dist}(\bsx,\bsx')$ is small. 
Otherwise, if $\ell$ is large, then $\theta(\bsx)$ and $\theta(\bsx')$ are  also strongly correlated, if $\mathrm{dist}(\bsx,\bsx')$ is large.
In Figure \ref{Fig_Various_Corr_len}, we plot samples of mean-zero Gaussian measures with exponential covariance and different correlation lengths.
The pointwise standard deviation $\sigma$ is constant for all $\bsx \in D$. 
One can show that the realisations of $\theta$ are a.s. continuous, independently of $\ell$ and $\sigma$.
\begin{figure}
\centering
\includegraphics[scale=0.9]{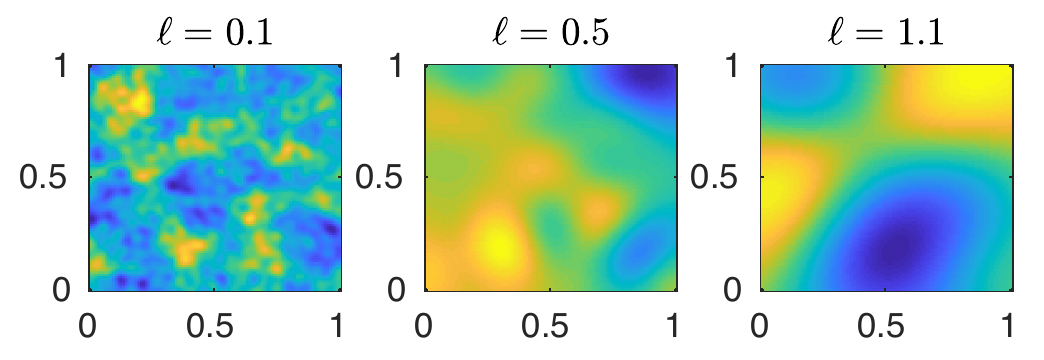}
\caption{Samples of mean-zero Gaussian random fields with exponential covariance, $\sigma = 1$ and $\ell \in \{0.1,0.5,1.1\}$.} \label{Fig_Various_Corr_len}
\end{figure}

\subsection{Parameterised Gaussian measures} \label{subs:ParamGaussMe}
We now construct parameterised measures.
We denote the space of parameters by $R \subseteq \mathbb{R}^{\mathrm{N}_R}$ and assume that $R$ is non-empty and finite-dimensional.
Moreover, we assume that $(R,\mathcal{R})$ is a measurable space.
The associated $\sigma$-algebra $\mathcal{R}$ is the trace-$\sigma$-algebra $\mathcal{R} := R \cap \mathcal{B}\mathbb{R}^{\mathrm{N}_R} := \{R \cap B : B \in\mathcal{B}\mathbb{R}^{\mathrm{N}_R}\}$.
In the following we denote an element of $R$ by $\paramedet$ and a random variable taking values in $R$ by $\parame: \Omega \rightarrow R$.
We recall the notion of a \textit{Markov kernel}; this is a measure-theoretically consistent object for the representation a parameterised probability measure.
\begin{definition}
A Markov kernel from $(R,\mathcal{R})$ to $(X, \mathcal{B}X)$ is a function $K: \mathcal{B}X \times R \rightarrow [0,1]$, where 
\begin{enumerate}
\item $K(B | \cdot ): R \rightarrow [0,1]$ is measurable for all $B \in \mathcal{B}X$,
\item $K(\cdot | \paramedet): \mathcal{B}X \rightarrow [0,1]$ is a probability measure for all $\paramedet \in R$.
\end{enumerate}
\end{definition}
We consider Markov kernels where $K(\cdot | \paramedet)$ is a  Gaussian measure for all $\paramedet \in R$. 
Hence, we define a parameterised Gaussian measure in terms of a Markov kernel from $(R, \mathcal{R})$ to $(X, \mathcal{B}X)$.  
Particularly, we define
\begin{equation} \label{Eq_param_Gaussian}
R \ni \paramedet \mapsto K(\cdot | {\paramedet}):=\mathrm{N}(m(\paramedet),\mathcal{C}(\paramedet)),
\end{equation}
where
\begin{align*}
m : R \rightarrow X, \ \  \mathcal{C} : R \rightarrow \mathrm{CO}(X)
\end{align*}
are measurable functions. 
The family of Gaussian measures in \eqref{Eq_param_Gaussian} does not define a valid parameterised measure yet, it remains to identify a measure on the parameter space $R$.
To this end let $\parame: \Omega \rightarrow R$ be an $R$-valued  random variable.
We assume that $\parame$ is distributed according to some probability measure $\mu'$.
Now, let $\theta: \Omega \rightarrow X$ be an $X$-valued random variable. 
We assume that  $\theta \sim K(\cdot | \parame).$
Then, the joint measure of $(\parame, \theta):  \Omega \rightarrow R \times X$ is given by
\begin{equation}\label{mu}
\mu:=  \mathcal{R} \otimes \mathcal{B}X \ni B \mapsto \iint_B K(\mathrm{d}\theta|\parame) \mu'(\mathrm{d}\parame) \in [0,1].
\end{equation}
The measure $\mu$ in \eqref{mu} models a two-step sampling mechanism.
Indeed, to generate a sample $(\parame, \theta) \in R \times X$ with distribution $\mu$ we proceed as follows:
\begin{enumerate}
\item Sample $\parame \sim \mu'$,
\item Sample $\theta \sim K(\cdot | \parame)$.
\end{enumerate}
Finally, let $\mu''$ be the marginal of $\mu$ w.r.t. to $\theta$, i.e., 
\begin{equation}\label{mu2}
\mu''(B_2) := \mu(R \times B_2), \qquad B_2 \in \mathcal{B}X.
\end{equation}
Alternatively, $\mu''$ can be defined in terms of the composition of $\mu'$ and $K$, $\mu'' := \mu'K$.
{Note that $\mu''$ is a measure on $X$, and that the Markov kernel $K$ is the parameterised Gaussian measure which we wanted to construct.}
\begin{remark}
We point out that even if $K(\cdot | \paramedet)$ is a Gaussian measure for any $\paramedet \in R$, the marginal $\mu''$ is not necessarily a Gaussian measure.
We give two examples.
\begin{itemize}
\item[(a)] Let $R := X := \mathbb{R}$, let $\mu'$ be a Gaussian measure and $K( \cdot | \tau_0) := \mathrm{N}(\tau_0, \sigma^2)$. 
Then $\mu''$ is a Gaussian measure.
{Indeed, this construction models a family of Gaussian random variables where the mean value is Gaussian.}
\item[(b)] Let $R$ be a finite set. 
Then $\mu''$ is called \textit{Gaussian mixture} and is typically not Gaussian. See \S 1.1 and \S 2.1 in \cite{Everitt1981}.
\end{itemize}
\end{remark}
Now, we return to Example \ref{Exponential Covariance} and construct an associated Markov kernel.
\begin{example} \label{Exponential Covariance_2}
We consider again the expo\-nen\-tial covariance operator in \eqref{EQ_exponential_covariance}.
Let $\underline{\ell} > 0$ and $\overline{\sigma} > \underline{\sigma}> 0$. 
For any $\ell \in [\underline{\ell}, \mathrm{diam}(D)]$ and $\sigma \in [\underline{\sigma}, \overline{\sigma}]$, one can show that $\mathcal{C}_{\exp}^{(\ell, \sigma)} \in \mathrm{CO}(X)$ is a valid covariance operator. 
The parameters $\tau = (\ell, \sigma)$ are random variables on a non-empty set $R := [\underline{\ell}, \mathrm{diam}(D)] \times [\underline{\sigma}, \overline{\sigma}]$. 
The associated probability measure $\mu'$ is given by 
\begin{equation*}
\mu' := \mu_\ell' \otimes \mu_\sigma'.
\end{equation*} 
Here, $\mu_\ell'$ is given such that $\ell^{-1} \sim \mathrm{Unif}[\mathrm{diam}(D)^{-1},\underline{\ell}^{-1}]$ and $\mu_\sigma' $ is a Gaussian measure that is truncated outside of $[\underline{\sigma}, \overline{\sigma}]$.
$\sigma$ models the standard deviation of $\theta(\bsx)$, for any $\bsx \in D$. 
The measure $\mu'$ and the Markov kernel $K(\cdot | \ell, \sigma) = \mathrm{N}(0, \mathcal{C}_{\exp}^{(\ell, \sigma)})$ induce a joint measure $\mu$. This can now be understood as follows:
\begin{enumerate}
\item Sample $\tau$ from $\mu'$:
\begin{itemize}
\item[(a)] Sample the correlation length $\ell \sim \mu'_\ell $,
\item[(b)] Sample the standard deviation $\sigma \sim \mu'_\sigma$.
\end{itemize}
\item Sample the random field $\theta \sim K(\cdot | \ell, \sigma)$ with exponential covariance operator, standard deviation $\sigma$ and correlation length $\ell$.
\end{enumerate}
Hence, we modelled a Gaussian random field with exponential covariance, where the correlation length and standard deviation are unknown.
\end{example}
In the following sections we study parameterised Gaussian measures in the setting of forward uncertainty propagation and Bayesian inversion.


\subsection{Sampling of Gaussian random fields}
\label{Subsec_Sampl_stra}
Consider a Gaussian random field $\mathrm{N}(m, \mathcal{C})$ on $(X, \mathcal{B}X)$.
For practical computations the infinite-dimensional parameter space $X$ must be discretised.
Here, we use finite elements.
Let $\mathrm{B}_h := (\varphi_i : i=1,\dots,N) \in X^{N}$ denote an $N$-dimensional basis of a finite element space.
We approximate $X$ by $X_h := \mathrm{span}(\mathrm{B}_h)$.
Note that we can identify $X_h \cong \mathbb{R}^N$.

Let $\langle \cdot, \cdot \rangle$ denote the Euclidean inner product on $\mathbb{R}^N$.
Note that $\mathbb{R}^N$ is a separable Hilbert space with inner product $\langle \cdot, \cdot \rangle_M= \langle \cdot,M \cdot \rangle$, where $M = \mathrm{B}_h^*\mathrm{B}_h$ is the Gramian matrix associated with the finite element basis $\mathrm{B}_h$.
The Gaussian measure $\mathrm{N}(m, \mathcal{C})$ on $(X, \mathcal{B}X)$ can then be represented on $\mathbb{R}^N$ by the measure $\mathrm{N}(\mathrm{B}_h^*m,\mathrm{B}_h^*\mathcal{C}\mathrm{B}_h)$ with mean vector $\mathrm{B}_h^*m$ and covariance matrix $\mathrm{B}_h^*\mathcal{C}\mathrm{B}_h$.

From now on assume that $X := \mathbb{R}^N$ is a finite dimensional space with inner product $\langle \cdot, \cdot \rangle_X$.
Moreover, we assume that the Gaussian measure $\mathrm{N}(m, \mathcal{C})$ has a mean equal to zero, that is $m\equiv 0$.
Note, that the covariance operator $\mathcal{C} \in \mathbb{R}^{N \times N}$ is now a matrix.
A simple sampling strategy uses the Cholesky decomposition $LL^*$ of $\mathcal{C}$.
Let $\vxi \sim \mathrm{N}(0,\mathrm{Id}_N)$. 
Then, it is easy to see that $L\vxi \sim \mathrm{N}(0,\mathcal{C})$.
The computational cost of a Cholesky decomposition is $O(N^3;N\rightarrow \infty)$.

Alternatively, we can use the KL expansion (see Definition \ref{Def_KL_exp})
\begin{equation*}
\sum_{i=1}^N \sqrt{\lambda_i} \xi_i \psi_i \sim \mathrm{N}(0, \mathcal{C}),
\end{equation*}
where $(\lambda_i, \psi_i)_{i=1}^N$ are the eigenpairs of $\mathcal{C}$.
Recall that the eigenvectors form an orthonormal basis of $X$ and that $\vxi \sim \mathrm{N}(0,\mathrm{Id}_N)$.
Computing the spectral decomposition of a symmetric matrix is typically more expensive compared to the Cholesky decomposition.
However, for some special cases the spectral decomposition can be computed cheaply.
One example is circulant embedding \cite{Dietrich1997} which requires a structured mesh, and a stationary covariance function.
Another example can be found in \cite{Dunlop2017} where the eigenpairs of the covariance operator are known analytically.

The KL expansion offers a natural way to reduce the stochastic dimension of an $X$-valued Gaussian random variable by simply truncating the expansion.
This can be interpreted as dimension reduction from a high-dimensional to a low-dimensional stochastic space.
A reduction from an infinite-dimensional to a finite-dimensional stochastic space is also possible.
Let $N_{\mathrm{sto}} \in \mathbb{N}, N_{\mathrm{sto}} < N$ and consider the truncated KL expansion
\begin{equation*}
\theta_{\mathrm{KL}}^{N_{\mathrm{sto}} }:= \sum_{i=1}^{N_{\mathrm{sto}}} \sqrt{\lambda_i} \xi_i \psi_i.
\end{equation*}
The sum of the remaining eigenvalues in the truncated KL expansion give the following $\mathcal{L}^2$ error bound:
\begin{equation}\label{L2errorKLE}
\mathbb{E}\left[\|\theta_{\mathrm{KL}} -  \theta_{\mathrm{KL}}^{N_{\mathrm{sto}} }\|_X^2 \right] = \sum_{i=N_{\mathrm{sto}} +1}^{N} \lambda_i.
\end{equation}
Furthermore, the truncated KL expansion solves the minimisation problem
\begin{equation*}
\min_{\hat{\theta} \in \mathcal{L}^2(\mathbb{R}^{N_{\mathrm{sto}}};X)}\mathbb{E}\left[\|\theta_{\mathrm{KL}} -  \hat{\theta}  \|_X^2 \right],
\end{equation*}
for any given $N_{\mathrm{sto}} \in \mathbb{N}$. 
Hence, the truncated KL expansion $\theta_{\mathrm{KL}}^{N_{\mathrm{sto}}}$ is the optimal $N_{\mathrm{sto}}$-dimensional function which approximates $\theta_{\mathrm{KL}}$.
In the statistics literature this method is called  \textit{principal component analysis}.

Observe that $\theta_{\mathrm{KL}}^{N_{\mathrm{sto}}}$ is a Gaussian random field on $X$ with covariance operator
\begin{equation} \label{EQ_Finite_Rank_Cov}
\mathcal{C}^{N_{\mathrm{sto}}} := \sum_{i=1}^{N_{\mathrm{sto}}} \lambda_i (\psi_i \otimes \psi_i).
\end{equation}
This covariance operator has the rank $\leq \Nsto < N$.
Sampling with the truncated KL expansion requires only the leading $N_{\mathrm{sto}}$ eigenpairs. 
We assume that this reduces the computational cost asymptotically to $O(N^2 \Nsto;N\rightarrow \infty)$.
We discuss this in more detail in \S\ref{Subsec_Comp_Cost}.
In the remainder of this report, we generate samples with the truncated KL expansion.

\subsection{Sampling of parameterised Gaussian random fields} \label{Subsec_Comp_Cost}

We use sample-based techniques to approximate the pushforward and posterior measure in forward propagation of uncertainty problems and Bayesian inverse problem, respectively.
To this end we require samples $(\parame_1, \theta_1),\dots,(\parame_{N_{\smpl}},\theta_{N_{\smpl}}) \sim \mu$.
We assume that sampling from $\mu'$ is accessible and inexpensive.
However, for each sample $\parame_n \sim \mu'$ we also need to sample $\theta_n \sim \mu_0(\cdot | \parame_n)$ using the truncated KL expansion.
This requires the assembly of the (dense) covariance matrix $\mathcal{C}(\parame_n)$, and the computation of its leading $\Nsto$ eigenpairs.
We abbreviate this process by the function $\mathrm{eigs}(\mathcal{C}(\parame_n), \Nsto)$ which returns $\Psi := (\lambda_i^{1/2}(\parame_n)\psi_i(\parame_n))_{i = 1}^{\Nsto}$.
The complete sampling procedure is given in Algorithm~\ref{Alg_SMK}.

The cost for the assembly of the covariance matrix is of order $O(N^2; N\rightarrow \infty)$.
We assume that the cost of a single function call $\mathrm{eigs}(\cdot, \Nsto)$ is of order $O(N^2 \cdot \Nsto;N \rightarrow \infty)$.
This corresponds to an \textit{Implictly Restarted Lanczos Method}, where $p = O(\Nsto)$.
See \cite{Calvetti1994,GuRuhe2000} for details.
Also note that this method is implemented in \textsc{ARPACK} (and thus for instance in \textsc{Matlab}) as \texttt{eigs}. 
Thus, the total computational cost of Algorithm~\ref{Alg_SMK} is $O(N_{\mathrm{smp}}\cdot (N^2 \cdot (\Nsto + 1));N \rightarrow \infty)$. 
The largest contribution to the computational cost is the repeated computation of the leading eigenpairs of $\mathcal{C}(\parame_n)$.
As mentioned before, we can avoid this cost in certain special cases, e.g. when considering a covariance operator that allows for circulant embedding or that has known eigenpairs.
However, in this paper we focus on parameterised covariance operators where the full eigenproblems have to be solved numerically for all parameter values.
\begin{algorithm}[htb]

 \For{$n \in \{1,...,N_{\mathrm{smp}}\}$}{
 Sample $\parame_n \sim \mu'$
 
$\Psi_n \leftarrow  \mathrm{eigs}(\mathcal{C}(\parame_n), \Nsto$)
 
 Sample $\vxi \sim \mathrm{N}(0, \mathrm{Id}_{\Nsto})$
 
 $\theta_n \leftarrow m_0(\parame_n) + \Psi_n\vxi$
 
 }
 \caption{Sampling from parameterised Gaussian measure $\mu$}
 \label{Alg_SMK}
\end{algorithm}

We conclude this section by discussing the choice of $\Nsto$ for parameterised Gaussian random fields. 
In \S \ref{Subsec_Sampl_stra} we study the truncation error of the Karhunen-Lo\`eve expansion of a Gaussian random field.
We now extend this study to parameterised Gaussian random fields.
Let $$\theta_{\mathrm{KL}}^{N_{\mathrm{sto}} }:= \sum_{i=1}^{N_{\mathrm{sto}}} \sqrt{\lambda_i(\tau)} \xi_i \psi_i(\tau),$$
where $\tau \sim \mu'$ and $\vxi \sim \mathrm{N}(0, \mathrm{Id}_{\Nsto})$.
Note that $\theta_{\mathrm{KL}}^{N_{\mathrm{sto}} }$ is an approximation to the parameterised Gaussian random field $\theta \sim \mu''$.
The mean square error of $\theta$ and $\theta_{\mathrm{KL}}^{N_{\mathrm{sto}} }$ can be computed as follows:
\begin{align*}
\mathbb{E}\left[\|\theta -  \theta_{\mathrm{KL}}^{N_{\mathrm{sto}} }\|_X^2 \right] &= \iint_{R \times \mathbb{R}^{\Nsto}} \left(\sum_{i=\Nsto +1}^N \sqrt{\lambda_i(\tau)} \xi_i \psi_i(\tau) \right)^2 \mathrm{N}(0,\mathrm{Id}_{\Nsto})(\mathrm{d}\vxi) \mu'(\mathrm{d}\tau) \\ &= \int_R \sum_{i=N_{\mathrm{sto}} +1}^{N} \lambda_i(\tau) \mu'(\mathrm{d}\tau).
\end{align*}
For Gaussian random fields $\Nsto$ is typically chosen such that the root mean square error fulfills a certain threshold. 
For example,
\begin{align*}
\Nsto := \min \left\lbrace N'=1,\dots,N : \sum_{i=1}^{N'} \lambda_i \geq A \cdot \sum_{i=1}^{N} \lambda_i \right\rbrace
\end{align*} 
where $A$ is a fixed factor.
Looking at the error bound in \eqref{L2errorKLE} we see that $A$ determines which amount of the total variance of the exact (Gaussian) random field is captured by the truncated KL expansion.
The same strategy can be applied for parameterised Gaussian random fields. 
Let \begin{align*}
\Nsto^\mathrm{all} := \min \left\lbrace N'=1,\dots,N : \sum_{i=1}^{N'} \lambda_i(\tau_0) \geq A \cdot \sum_{i=1}^{N} \lambda_i(\tau_0) \quad (\tau_0 \in R)\right\rbrace
\end{align*} 
be the number of terms that fulfils the threshold $A$ for all parameters $\tau_0 \in R$.
Then, the mean square error is bounded by
\begin{equation}\label{mse}
\mathbb{E}\left[\|\theta -  \theta_{\mathrm{KL}}^{N_{\mathrm{sto}}^\mathrm{all}}\|_X^2 \right] \leq (1-A) \cdot \mathbb{E}\left[\|\theta\|_X^2\right].
\end{equation}
Alternatively, it is possible to choose $\Nsto$ individually for each $\tau_0 \in R$,
\begin{align*}
\Nsto^{\tau_0} := \min \left\lbrace N'=1,\dots,N : \sum_{i=1}^{N'} \lambda_i(\tau_0) \geq A \cdot \sum_{i=1}^{N} \lambda_i(\tau_0) \quad\right\rbrace.
\end{align*} 
This gives the truncated representation 
$$\theta_{\mathrm{KL}}^{N_{\mathrm{sto}}^{\tau} }:= \sum_{i=1}^{N_{\mathrm{sto}}^{\tau}} \sqrt{\lambda_i(\tau)} \xi_i \psi_i(\tau) \quad (\vxi \sim \mathrm{N}(0, \mathrm{Id}_{\Nsto}), \tau \sim \mu').$$
Clearly, the mean square error of this expansion fulfils the exact same error bound as in \eqref{mse}.
However, the total number of terms in the expansion for a fixed parameter value $\tau$ might be smaller.
Recall that the cost of the sampling depends (linearly) on the number of KL terms. 
Observe that $\Nsto^{\mathrm{all}}$ is a sharp upper bound for $\Nsto^{\tau_0}$, $\tau_0 \in R$.
Hence, using $\Nsto^{\tau_0}$ is overall not more expensive than using $\Nsto^\mathrm{all}$, and the truncated expansion satisfies the same error bound.
Moreover, the numbers $(\Nsto^{\tau_0})_{\tau_0 \in R}$ are a priori unknown and have to be computed. 
To avoid this additional cost and to simplify the following discussion, we use $\Nsto := \Nsto^{\mathrm{all}}$ independently of $\tau_0 \in R$.

\section{Reduced basis approach to parameterised eigenproblems} \label{Sec_RB_Param_EP}
In \S\ref{Subsec_Comp_Cost} we discuss the sampling of parameterised Gaussian random fields.
The largest contribution to the computational cost is the repeated computation of eigenpairs of the associated parameterised covariance matrix $\mathcal{C}(\paramedet)$ for multiple parameter values $\paramedet \in R$.
Reduced basis (RB) methods construct a low-dimensional trial space for a family of parameterised eigenproblems.
This is the cornerstone in our fast sampling algorithm.
To begin we explain the basic idea behind reduced basis (RB) approaches for eigenproblems.
There are many options to construct a reduced basis, such as the proper orthogonal decomposition (POD), as well as single- and multi-choice greedy approaches.
POD and greedy approaches for parameterised eigenproblems are discussed and compared in \cite{Horger2016}.
In this paper we focus on the POD.

RB algorithms have two parts, an offline or preprocessing phase and an online phase.
The offline phase consists of the construction of the reduced basis.
In the online phase the reduced basis is used to solve a large number of low-dimensional eigenproblems.
Finally, in \S\ref{Subsec:Matern}, we discuss the offline-online decomposition for Mat\'ern-type covariance operators.
To be able to implement the RB approach efficiently we approximate the full covariance operators by linearly separable operators.
We explain how this can be done, and analyse the proposed class of approximate covariance operators.

\subsection{Basic idea} \label{Subsec_Basic_Idea_RB}
Let $\mathcal{C}: R \rightarrow \mathrm{CO}(X)$ be a measurable map, where $(X, \langle \cdot, \cdot \rangle_X) := (\mathbb{R}^N, \langle \cdot, \cdot \rangle_M)$ is a finite-dimensional space arising from the discretisation of an infinite-dimensional function space (see  \S \ref{Subsec_Sampl_stra}). 
Recall that in Algorithm \ref{Alg_SMK} we need to solve the generalised eigenproblem associated with $\mathcal{C}(\paramedet)$ for multiple parameter values $\paramedet \in R$. 
That is, we want to find $(\lambda_i(\paramedet), \psi_i(\paramedet))_{i = 1}^{N_{\mathrm{sto}}} \in (\mathbb{R} \times X)^{N_{\mathrm{sto}}}$, such that
\begin{equation}\label{fullEVP}
\mathcal{C}(\paramedet) \psi_i(\paramedet) = \lambda_i(\paramedet) M\psi_i(\paramedet).
\end{equation}
$X$ is in general high-dimensional, which results in a large computational cost for solving the eigenproblems.
However, it is often not necessary to consider the full space $X$.
If we assume that the eigenpairs corresponding to different parameter values are closely related, then the space
\begin{equation*}
\mathrm{span}\{\psi_i(\paramedet) : i=1,\dots,N_{\mathrm{sto}}, \paramedet \in R\} \subseteq X
\end{equation*}
can be approximated by a low dimensional subspace $X_{\RB}$, where $N_{\RB} := \dim X_{\RB} \ll N$.
We point out that the truncated KL expansion requires $\Nsto$ eigenpairs by assumption.
However, the reduced operators are ${N_{\RB} \times N_{\RB}}$ matrices with $N_{\RB}$ eigenpairs. 
Hence, $N_{\RB} \geq \Nsto$ is required.

Now, let $W \in X_{\RB}^{N_{\RB}} $ be an orthonormal basis of $X_{\RB}$ with respect to the inner product $\langle \cdot , \cdot \rangle_{X_{\RB}} := \langle \cdot , \cdot \rangle_{X} := \langle \cdot,  \cdot \rangle_M$.
$W$ is called \textit{reduced basis} and $X_{\RB}$ is called \textit{reduced space}.
We can represent any function $\psi(\paramedet) \in X_{\RB}$ by a coefficient vector $w(\paramedet) \in \mathbb{R}^{N_{\RB}}$, such that $\psi(\paramedet) = Ww(\paramedet)$.
The \textit{reduced eigenproblem} is obtained by a Galerkin projection of the full eigenproblem in \eqref{fullEVP}, and is again a generalised eigenproblem.
The task is to find $(\lambda_i^\RB(\paramedet), w_i^\RB(\paramedet))_{i = 1}^{\Nsto} \in (\mathbb{R} \times\mathbb{R}^{N_{\RB}})^{\Nsto}$, such that
\begin{equation}  \label{EQ_rhoRBEP}
\mathcal{C}^{\RB}(\paramedet) w_i(\paramedet) = \lambda^{\RB}_i(\paramedet) M^{\RB} w_i(\paramedet),\ \ \  i = 1,\dots,\Nsto.
\end{equation}
In \eqref{EQ_rhoRBEP} we have the \textit{reduced operator} $\mathcal{C}^{\RB}(\paramedet)  := W^*\mathcal{C}(\paramedet)W$, and the \textit{reduced Gramian matrix} $M^{\RB}  := W^* M W$, that are both  ${N_{\RB}\times N_{\RB}}$ matrices.
The eigenvector approximation in $X_{\RB}$ can then be obtained by $$\psi_i^\RB(\paramedet) = Ww_i(\paramedet), \qquad i = 1,\dots, \Nsto.$$

\subsection{Offline-online decomposition} \label{Subsec:Offline_Online}
A reduced basis method typically has two phases.
In the offline phase the reduced basis $W$ is constructed. 
In the online phase the reduced operator $\mathcal{C}^{\RB}(\paramedet)$  is assembled, and the reduced eigenproblem \eqref{EQ_rhoRBEP} is solved for selected parameter values $\paramedet \in R$.
To be able to shift a large part of the computational cost from the online to the offline phase we assume that the following offline-online decomposition is available for the family of parameterised covariance operators.
\begin{assumptions} \label{Ass_Linearly_Separable}
Let $N_{\mathrm{lin}} \in \mathbb{N}$. 
We assume that there are functions $F_k: R \rightarrow \mathbb{R}$ and linear operators $\mathcal{C}_k$,  $k = 1,\dots,N_{\mathrm{lin}}$, such that
\begin{equation*}
\mathcal{C}(\paramedet) = \sum_{k=1}^{N_{\mathrm{lin}}}F_k(\paramedet)\mathcal{C}_k, \ \ \  \paramedet \in R.
\end{equation*}
In this case, $\mathcal{C}(\paramedet)$ is called a \textit{linearly separable} operator.
\end{assumptions}

\subsubsection{Offline phase} \label{Subsubsec_Offline_Phase}
We use snapshots of the full eigenvectors to construct the reduced basis.
Meaning that we choose a vector $\parame^{\snap} \in R^{N_\snap}$ and solve the full eigenproblem \eqref{fullEVP} for all elements of $\parame^{\snap}$. 
We then have
\begin{equation*}
W_\snap = (\psi_i(\parame^\snap_s) : s = 1,\dots,N_\snap, i = 1,\dots,\Nsto),
\end{equation*}
where all of the computed eigenfunctions are included and here represented as column vectors. Hence, we obtain a matrix $W_\snap \in \mathbb{R}^{N \times \Nsto N_\snap}$.
Moreover, we define the reduced space $X_\RB :=\mathrm{span}(W_\snap)$.
Next, we construct an orthonormal basis for this vector space. 
One option to do this is the \textit{proper orthogonal decomposition}.
As result of the POD we obtain a spectral decomposition of $W_\snap^{(2)}:=W_\snap W_\snap^*$ of the form
\begin{equation*}
W_\snap^{(2)} = Q \Lambda Q^{*},
\end{equation*}
where $\Lambda := \mathrm{diag}(\lambda_1^{\snap},\dots,\lambda_N^{\snap})$ is a diagonal matrix containing the eigenvalues of $W_\snap^{(2)}$ and each  column of $Q$ contains the associated orthonormal eigenvectors. 
We use the eigenvectors with non-zero eigenvalues as basis vectors of $X_\RB$, that is,
\begin{equation*}
W := (Q_{\cdot,i} : \lambda_i^{\snap} > 0, i=1,\dots,N).
\end{equation*}
The magnitude of the eigenvalues of $W_\snap^{(2)}$ is an indicator for the error when the corresponding eigenvectors are not included in $W$.
See the discussion in \S \ref{Subsec_Sampl_stra}.
Neglecting reduced basis vectors however is beneficial due to the smaller dimension of the reduced basis.
Depending on the pay-off of the dimension reduction compared to the approximation accuracy of the reduced basis one can choose a threshold $\underline{\lambda} > 0$ and work with the basis 
\begin{equation*}
W := (Q_{\cdot,i} : \lambda_i^{\snap} > \underline{\lambda}, i=1,\dots,N).
\end{equation*}
In this case, we redefine  $X_\RB :=\mathrm{span}(W)$.
\begin{remark}
In this paper we compute the singular value decomposition (SVD) of $W_\snap$ instead of the spectral decomposition of $W_\snap^{(2)}$.
This is possible since the squares of the singular values of $W_\snap$ are identical to the eigenvalues of $W_\snap^{(2)}$.
\end{remark}

There are many options to choose the snapshot parameter values $\parame^{\snap}$. 
In our applications $\parame$ is a random variable with probability measure $\mu'$.
Hence, a straightforward method is to sample independently $\parame^{\snap}_s \sim \mu'$ $(s = 1,\dots,N_\snap)$.
Alternatively, one can select deterministic points in $R$, e.g. quadrature nodes.
We will come back to this question in \S \ref{Sec:Numerical_Experiments} where we discuss the numerical experiments.
Furthermore, note that it is generally possible to use different reduced bases $W_1, W_2,\dots$ for different subsets $R_1, R_2,\dots \subseteq R$ of hyperparameters and/or index sets $I_1,I_2,\dots \subseteq \{1,\dots,\Nsto\}$  of eigenpairs.

\subsubsection{Online phase}
In the online phase we iterate over various hyperparameter values $\paramedet \in R$. 
In every step, we assemble the operator $\mathcal{C}^\RB(\paramedet)$, and then we solve the associated eigenproblem \eqref{EQ_rhoRBEP}.
By Assumption \ref{Ass_Linearly_Separable} it holds $$\mathcal{C}(\paramedet)=\sum_{k=1}^{N_\lin}F_k(\paramedet)\mathcal{C}_k.$$
Hence, the reduced operator can be assembled efficiently as follows,
\begin{align*}
\mathcal{C}^\RB(\parame) = W^*\sum_{k=1}^{N_\lin}F_m(\parame)\mathcal{C}_kW = \sum_{k=1}^{N_\lin}F_k(\parame)W^*\mathcal{C}_kW = \sum_{k=1}^{N_\lin}F_k(\parame)\mathcal{C}_k^\RB.
\end{align*}
The reduced operators $\mathcal{C}_k^\RB$ $(k = 1,\dots,N_\lin)$ can be computed in the offline phase and stored in the memory.
In the online phase, we then only need to compute a certain linear combination of $(\mathcal{C}_k^\RB)_{k = 1}^{N_\lin}$.
This reduces the computational cost of the assembly of the reduced operator significantly.
After the assembly step we solve the reduced eigenproblem \eqref{EQ_rhoRBEP} to obtain the eigenfunctions $\psi_i^\RB(\paramedet):=Ww_i(\paramedet) \in X$ and eigenvalues $\lambda_i^\RB(\parame), i = 1,\dots,\Nsto$.

\subsection{Mat\'ern-type covariance operators} \label{Subsec:Matern}
Mat\'ern-type covariance operators are widely used in spatial statistics and uncertainty quantification.
They are particularly popular for modelling spatially variable uncertainties in porous media.
We are interested in solving the KL eigenproblem with Mat\'ern covariance operators with hyperparameters e.g. the correlation length.
Unfortunately, the Mat\'ern-type covariance operators are not linearly separable with respect to the hyperparameters of interest (see Assumption~\ref{Ass_Linearly_Separable}).
For this reason we introduce and analyse a class of linearly separable covariance operators which can approximate Mat\'ern-type covariance operators with arbitrary accuracy.
\begin{definition} \label{Def_Matern_int}
Let $D \subseteq \mathbb{R}^d, d=1,2,3$ be an open, bounded and connected domain, and let $X := \mathcal{L}^2(D; \mathbb{R})$. Furthermore, let $\ell \in (0, \mathrm{diam}(D)), \nu \in (0, \infty], \sigma \in (0, \infty)$.
Define the covariance kernel $c{(\nu, \ell, \sigma)}: [0, \infty) \rightarrow [0, \infty)$ as
\begin{equation*}
z \mapsto c{(\nu, \ell, \sigma)}(z) =  \frac{\sigma^2 \cdot 2^{1-\nu}}{\Gamma(\nu)}\left(\sqrt{2\nu}\frac{z}{\ell}\right)^\nu K_\nu\left( \sqrt{2\nu}\frac{z}{\ell}\right),
\end{equation*}
where $K_\nu$ is the modified Bessel function of the second kind. 
Then, the \textit{Mat\'ern-type covariance operator with smoothness $\nu$, standard deviation $\sigma$ and correlation length $\ell$} is given by
\begin{equation*}
\mathcal{C}{(\nu, \ell, \sigma)}: X \rightarrow X, \varphi \mapsto \int_D \varphi(\bsx) c{(\nu, \ell, \sigma)}(\mathrm{dist}(\bsx,\cdot)) \mathrm{d}\bsx,
\end{equation*}
where $\mathrm{dist}:D\times D\rightarrow [0, \infty)$ is the Euclidean distance in $D$.
\end{definition}

\begin{remark}
The exponential covariance operator in Examples \ref{Exponential Covariance} and \ref{Exponential Covariance_2} is a Mat\'ern-type covariance operator. 
Indeed, $\mathcal{C}{(1/2, \ell, \sigma)}\equiv \mathcal{C}^{(\ell, \sigma)}_{\exp}$.
\end{remark}

In Example~\ref{Exponential Covariance_2} we discussed the possibility of using the standard deviation $\sigma$ and the correlation length $\ell$ as hyperparameters in a Mat\'ern-type covariance operator.
What are the computational implications for the KL expansion?
Changing $\sigma$ only rescales the KL eigenvalues, and does not require a re-computation of the KL.
However, changing the correlation length clearly changes the KL eigenfunctions.
We can see this in Figure \ref{Figure_Comparison_Eigenfunctions}.
However, the good news is that the KL eigenfunctions for different correlation lengths are very similar, for example, the number and location of extrema is preserved.
This suggests that we might be able to construct a useful reduced basis from selected snapshots of KL eigenpairs corresponding to different correlation lengths.
\begin{figure}[htb]
\centering
\includegraphics[scale=0.75]{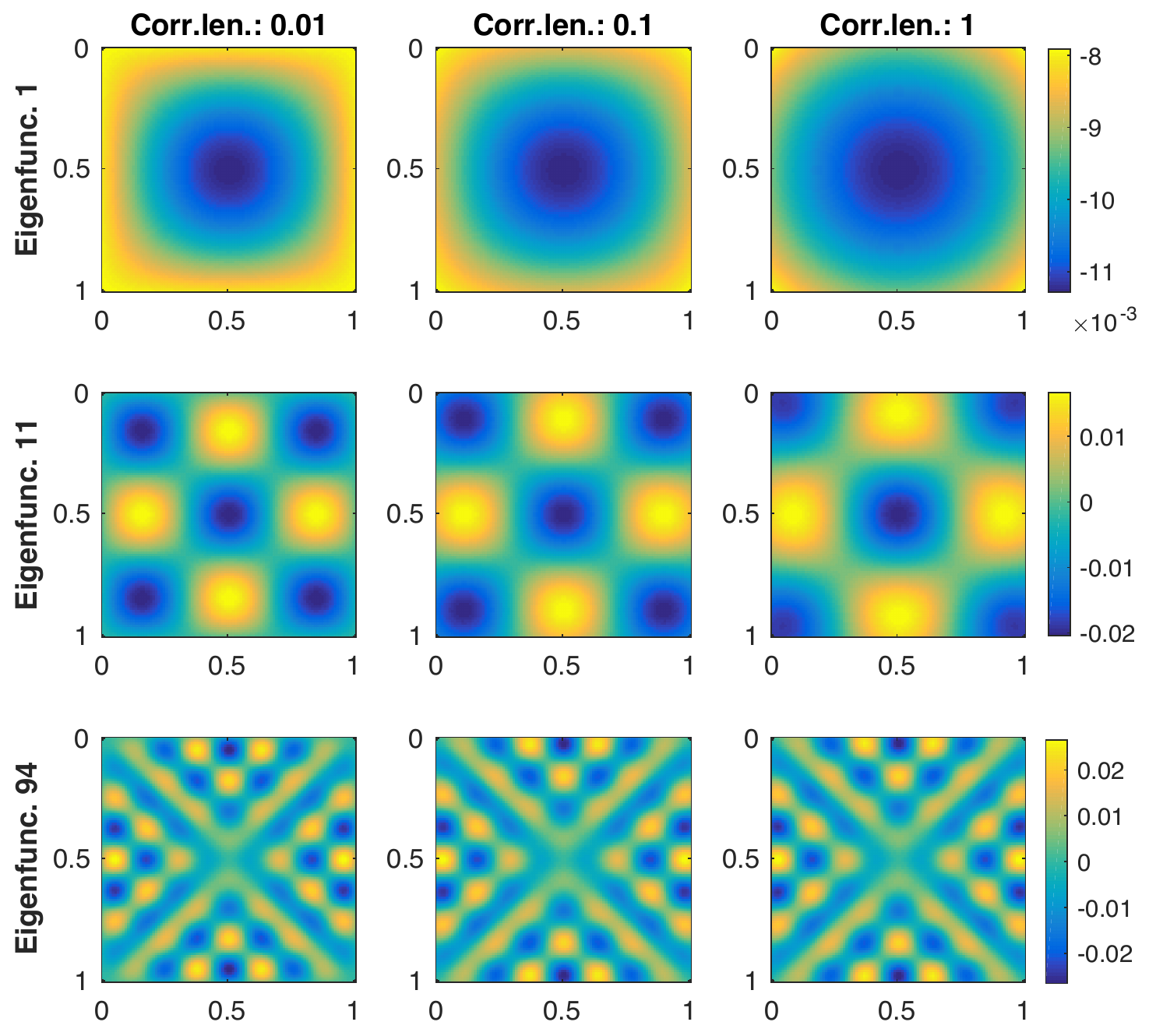}
\caption{Eigenfunctions 1, 11 and 94 of the Mat\'ern-type covariance operator with correlation lengths $\ell = 0.01, 0.1, 1$ and $\nu = 1/2$.}
\label{Figure_Comparison_Eigenfunctions}
\end{figure}

To be able to construct and use the reduced basis efficiently requires the linear separability of the covariance operator, see Assumption~\ref{Ass_Linearly_Separable}.
The Mat\'ern operator is linearly separable with respect to $\sigma$.
Unfortunately, it is not linearly separable with respect to $\ell$, since the covariance function 
$c{(\nu, \ell, \sigma)}$ is not linearly separable.
However, it is possible to approximate $\mathcal{C}(\nu, \ell, \sigma)$ with any precision by a linearly separable operator.
Using the approximate, linearly separable operator allows us to construct an offline-online decomposition for the exact Mat\'ern covariance operator without the need to use advanced linearisation techniques, such as the discrete empirical interpolation method ((D)EIM).
We show this in the remainder of this section for $\nu \in (0, \infty) \backslash \mathbb{N}$.
Similar approximations for $\nu \in \mathbb{N} \cup \{\infty\}$ follow from the analyticity of $K_\nu$.

\begin{assumptions}\label{assumptions:ell}
	The correlation length $\ell$ satisfies $0<\underline{\ell} \leq \ell$ with fixed $\underline{\ell}$.
\end{assumptions}

\begin{definition} \label{Def_LinearSeparable_matern}
Let $\nu \in (0, \infty) \backslash \mathbb{N}$ and $N_\lin \in 2\mathbb{N}$. 
Moreover, let Assumption~\ref{assumptions:ell} hold.
We define the \emph{$N_\lin$-term approximate} of $c(\nu, \ell, \sigma)$ by
\begin{align*}
c&(\nu, \ell, \sigma, N_\lin)(z)\\
&:= \frac{\sigma^2\pi \mathrm{csc}(\pi\nu)}{\Gamma(\nu)}   \sum_{k=1}^{N_\lin/2}\frac{(\sqrt{\nu}z/\ell)^{2k-2}}{2^{k-1} \Gamma(k-\nu)(k-1)!} -  \frac{(\sqrt{\nu}z/\ell)^{2\nu+2k-2} }{2^{k+\nu-1}\Gamma(k+\nu)(k-1)!}.
\end{align*}
The associated covariance operator is then defined as
$$\widetilde{\mathcal{C}}{(\nu, \ell, \sigma, N_\lin)}:X \rightarrow X, \varphi \mapsto \int_D \varphi(\bsx) c{(\nu, \ell, \sigma, N_\lin)}(\mathrm{dist}(\bsx,\cdot)) \mathrm{d}\bsx.$$
\end{definition}
Note that the operator $\widetilde{\mathcal{C}}{(\nu, \ell, \sigma, N_\lin)}$ is linearly separable w.r.t. $\ell$. 
In particular, we have
\begin{align*}
\widetilde{\mathcal{C}}{(\nu, \ell, \sigma, N_\lin)}&:=\sum_{k=1}^{N_\lin}F_k(\nu,\ell,\sigma)\mathcal{C}_k(\nu) \\
F_k(\nu,\ell,\sigma) &:= \frac{\pi \mathrm{csc}(\pi\nu)}{\Gamma(\nu)}  \cdot \begin{cases}  \frac{\sigma^2}{\ell^{k-2}}, &\text{if } k \in 2\mathbb{N},\\
\frac{\sigma^2}{\ell^{2\nu + k-1}}, &\text{if } k \in 2\mathbb{N}-1,
\end{cases} \\
\mathcal{C}_k\varphi &:= \begin{cases} \int_D\varphi(\bsx) \frac{(\sqrt{\nu}\mathrm{dist},\cdot))^{k-2}}{2^{k/2-1} \Gamma(k/2-\nu)(k/2-1)!} \mathrm{d}\bsx, &\text{if } k \in 2\mathbb{N},\\
\int_D\varphi(\bsx)  \frac{(-1) \cdot (\sqrt{\nu}z/\ell)^{2\nu+k-1} }{2^{k/2+\nu-1/2}\Gamma(k/2+1/2+\nu)(k/2-1/2)!} \mathrm{d}\bsx, &\text{if } k \in 2\mathbb{N}-1,
\end{cases}
\end{align*}
for any $\varphi \in X$.

The operator $\widetilde{\mathcal{C}}{(\nu, \ell, \sigma, N_\lin)}$ arises from a truncation of a series expansion of $K_\nu$.
This is detailed in the proof of the following lemma, where we derive an error bound between the exact Mat\'ern covariance operator $\widetilde{\mathcal{C}}{(\nu, \ell, \sigma, N_\lin)}$ and the linearly separable approximation ${\mathcal{C}}{(\nu, \ell, \sigma)}$.

\begin{lemma} \label{Theorem_Matern_approx_linearly_sepa}
Let $\nu \in (0,\infty) \backslash \mathbb{N}$, let $N_\lin \in 2\mathbb{N}$ and let Assumption~\ref{assumptions:ell} hold. Then,
\begin{align*}
\| \widetilde{\mathcal{C}}{(\nu, \ell, \sigma,N_\lin)}&- {\mathcal{C}}{(\nu, \ell, \sigma)} \|_X \\&\leq \mathrm{diam}(D)^{2d} \frac{\pi|\mathrm{csc}(\pi\nu)|}{2^{1-\nu}}(1+\zeta_{\max}^{2\nu})  \exp\left(\frac{\zeta_{\max}^2}{4}\right)    \frac{\zeta_{\max}^{2N_\lin}}{(N_\lin)!},
\end{align*}
where $\zeta_{\max} := {\mathrm{diam}(D)}/{\underline{\ell}}$.
\end{lemma}
\begin{proof}
See Appendix~\ref{Section_Proofs_of_Lemmata}.
\end{proof}
The covariance operator approximation brings new issues. 
The Mat\'ern-type covariance operators $\mathcal{C}{(\nu, \ell, \sigma)}$ are valid covariance operators in $\mathrm{CO}(X)$.
However, this is not necessarily the case for $\widetilde{\mathcal{C}}{(\nu, \ell, \sigma, N_\lin)}$.
One can easily verify the following.
\begin{lemma} \label{Lemma_C_tilde}
The operator $\widetilde{\mathcal{C}}{(\nu, \ell, \sigma, N_\lin)}$ is self-adjoint, trace-class and continuous.
\end{lemma}
\begin{proof}
See Appendix~\ref{Section_Proofs_of_Lemmata}.
\end{proof}
However, $\widetilde{\mathcal{C}}{(\nu, \ell, \sigma, N_\lin)}$ is not necessarily positive definite.
Under weak assumptions we can cure this by replacing $\widetilde{\mathcal{C}}{(\nu, \ell, \sigma, N_\lin)}$ by an operator $\widetilde{\mathcal{C}}_{0}{(\nu, \ell, \sigma, N_\lin)}$ which has the exact same eigenfunctions and positive eigenvalues as $\widetilde{\mathcal{C}}{(\nu, \ell, \sigma, N_\lin)}$, however, all negative eigenvalues are set to zero. 
Formally, we define  $\widetilde{\mathcal{C}}_{0}{(\nu, \ell, \sigma, N_\lin)}$ by
\begin{align} \label{C0tilde}
\widetilde{\mathcal{C}}_{0}{(\nu, \ell, \sigma, N_\lin)} = \sum_{i=1; \widetilde{\lambda}_i > 0}^{\infty} \widetilde{\lambda}_i (\widetilde{\psi}_i \otimes \widetilde{\psi}_i),
\end{align}
where $(\widetilde{\lambda}_i,\widetilde{\psi}_i)_{i=1}^\infty$ are eigenpairs of $\widetilde{\mathcal{C}}{(\nu, \ell, \sigma, N_\lin)}$ and the eigenfunctions are orthonormal.
Note that the same technique has been applied in \cite{Chernov2016} to remove the degeneracy of multilevel sample covariance estimators.
Fortunately, we can show that the approximation error of $\widetilde{\mathcal{C}}_{0}{(\nu, \ell, \sigma, N_\lin)}$ is of the same order as the error of $\widetilde{\mathcal{C}}{(\nu, \ell, \sigma, N_\lin)}$.

\begin{lemma} \label{Lemma_Operator_bounds}
	The Mat\'ern-type covariance operator $\mathcal{C}(\nu, \ell, \sigma)$ and the approximate operator $\widetilde{\mathcal{C}}_{0}(\nu, \ell, \sigma,N_\lin)$ in \eqref{C0tilde} satisfy 
		\[\| \widetilde{\mathcal{C}}_{0}{(\nu, \ell, \sigma,N_\lin)}- {\mathcal{C}}{(\nu, \ell, \sigma)} \|_X \leq 2 \| \widetilde{\mathcal{C}}{(\nu, \ell, \sigma,N_\lin)}- {\mathcal{C}}{(\nu, \ell, \sigma)} \|_X.\]
\end{lemma}
\begin{proof}
See Appendix~\ref{Section_Proofs_of_Lemmata}.
\end{proof}
We summarise the results in Lemma~\ref{Theorem_Matern_approx_linearly_sepa}--\ref{Lemma_Operator_bounds} as follows.
\begin{proposition}\label{Propo_Matern_approx_compl}
Let $\nu \in (0,\infty) \backslash \mathbb{N}$.
Under Assumption~\ref{assumptions:ell} there is a linearly separable, valid covariance operator
$\widetilde{\mathcal{C}}_0{(\nu, \ell, \sigma, N_\lin)} \in \mathrm{CO}(X)$
consisting of $N_\lin \in 2\mathbb{N}$ terms, such that
\begin{equation*}
\| \widetilde{\mathcal{C}}_0{(\nu, \ell, \sigma,N_\lin)}- {\mathcal{C}}{(\nu, \ell, \sigma)} \|_X \leq \mathrm{const}''(\nu) (1+\zeta_{\max}^{2\nu})  \exp\left(\frac{\zeta_{\max}^2}{4}\right)    \frac{\zeta_{\max}^{2N_\lin}}{(N_\lin)!},
\end{equation*}
where $\zeta_{\max} := {\mathrm{diam}(D)}/{\underline{\ell}}$ and $\mathrm{const}''(\nu)>0$ is a constant that depends only on $\nu$.
\end{proposition}
The expansion in \eqref{C0tilde} has infinitely many terms.
We truncate this expansion and retain only the leading $\Nsto$ terms, denoting the resulting covariance operator by $\widetilde{\mathcal{C}}_0{(\nu, \ell, \sigma,N_\lin, \Nsto)}$.

Finally, we discuss the sample path continuity.
Samples of Gaussian random fields with Mat\'ern-type covariance operators are almost surely continuous functions.
In the following proposition we show that this also holds for the realisations of the Gaussian random fields with measure $\mathrm{N}(0, \widetilde{\mathcal{C}}_0{(\nu, \ell, \sigma,N_\lin,  \Nsto)})$.
\begin{proposition}
Let $\theta \sim \mathrm{N}\left(0, \widetilde{\mathcal{C}}_0{(\nu, \ell, \sigma,N_\lin, \Nsto)}\right)$, where $\nu \in (0, \infty)\backslash \mathbb{N}$. 
Then it holds $\theta \in C^0(D)$. 
\end{proposition}
\begin{proof}
We consider the random field in terms of its (finite) KL expansion,
\begin{equation*}
\theta := \theta_{\mathrm{KL}}^{N_{\mathrm{sto}} }:= \sum_{i=1}^{N_{\mathrm{sto}}} \sqrt{\widetilde{\lambda}_i} \widetilde{\psi}_i \xi_i,
\end{equation*}
where $(\widetilde{\lambda}_i, \widetilde{\psi}_i)_{i=1}^{\Nsto}$ are the eigenpairs of $\widetilde{\mathcal{C}}_0{(\nu, \ell, \sigma,N_\lin, \Nsto)}$ with positive eigenvalues. 
Then, $\theta$ is a continuous function, if $(\widetilde{\psi}_i)_{i=1}^{\Nsto}$ is a family of continuous functions.
This is clear by the definition of $(\widetilde{\psi}_i)_{i=1}^{\Nsto}$.
Indeed, let $i = 1,\dots,\Nsto$. 
Then it holds 
\begin{equation*}
\widetilde{\psi}_i(\bsx')= \frac{1}{\lambda_i} \int_D \widetilde{\psi}_i(\bsx) \widetilde{c}^{(\nu, \ell, \sigma,N_\lin)}(\mathrm{dist}(\bsx,\bsx')) \mathrm{d}\bsx, \ \ \ y \in D.
\end{equation*}
By definition, $\widetilde{c}^{(\nu, \ell, \sigma,N_\lin)}(\mathrm{dist}(\cdot, \cdot))$  is a continuous function. 
Hence, one can easily verify that the right-hand side is a continuous function in $\bsx'$.
\end{proof}

We now comment on the error bound given in Lemma \ref{Theorem_Matern_approx_linearly_sepa} and Proposition \ref{Propo_Matern_approx_compl}.
As $N_\lin$ increases, the error bound goes to zero, asymptotically like $O(1/(N_\lin!); N_\lin \rightarrow \infty)$. 
However, when the lower bound of the correlation length $\underline{\ell}$ is small, the prefactor of the error bound explodes like $O(\exp(\underline{\ell}^{-2});\underline{\ell} \downarrow 0)$.
Hence, for small $\underline{\ell}$, a very large number of terms $N_\lin$ is required to obtain a small error.
In addition, for large $N_\lin$, numerical cancellations occur and reduce the accuracy of the approximation.
We show this for the exponential covariance in Figure \ref{Fig_truncation_error_kernel} where we plot the truncation error
\begin{equation} \label{Eq_truncationerror_covar_kernel}
\sup_{z\in [0,\sqrt{2}], \ell \in [\underline{\ell}, \sqrt{2}]}|c(\nu, \ell, \sigma, N_\lin)(z)-c(\nu, \ell, \sigma)(z)|,
\end{equation}
where $\nu = 1/2$ refers to the exponential covariance and $\sigma = 1$, for different choices of $N_\lin$ and $\underline{\ell}$.
We clearly see that the linearisation technique in Definition \ref{Def_LinearSeparable_matern} is not suitable for very small correlation lengths. 
In such a case, one could use alternative linearisation techniques, e.g. a polynomial chaos expansions in $\ell$, a Taylor expansion of the Fourier representation of the Mat\'ern kernel, or (D)EIM.
\begin{figure}
\centering
\includegraphics[scale=0.7]{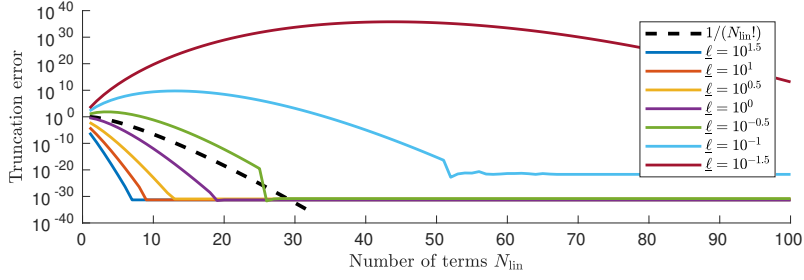}
\caption{Error in the exponential covariance kernel when using the truncated linearisation in  Definition \ref{Def_LinearSeparable_matern}, for different numbers of terms $N_\lin = 1,\dots,100$ and different minimal correlation lengths $\underline{\ell} = 10^{-1.5},\dots,10^{1.5}$.}
\label{Fig_truncation_error_kernel}
\end{figure}

Until now we considered the estimation of $\ell$ and $\sigma$, but not the estimation of $\nu$. 
We comment on this in the following remark.
Note further that the estimation of the smoothness $\nu$ of a Gaussian random field is studied for instance in the PhD thesis 
\cite[Chp. 4]{Haskard2007} where a maximum-likelihood type estimation of the smoothness is performed.
\begin{remark}
For the Mat\'ern kernel the  map $\nu \mapsto \zeta^\nu K_\nu(\zeta)$ is analytic for fixed $\zeta \neq 0$. 
Hence a linearisation as discussed in this section is generally possible. 
However, we expect that a reduced basis approach is not an efficient way to estimate the smoothness of a random field.
Note that the smoothness parameter determines the smoothness of the random field realisations and shape of the eigenfunctions.
To accurately represent functions with variable smoothnesses one would need separate reduced bases for each value of the smoothness parameter since otherwise we cannot guarantee mathematically that the random field sample is a.s. in the correct function space.
We expect that over a small range of smoothness parameters a reduced basis could be constructed, however, the potential computational savings are limited.

\end{remark}
\section{Reduced basis sampling} \label{Sec_RB_Sampling}
In \S\ref{Subsec_Comp_Cost} we discuss sampling from a Markov kernel, or, equivalently, a parameterised Gaussian measure
\begin{equation*}
K(\cdot | \parame) := \mathrm{N}(m(\parame), \mathcal{C}(\parame)),
\end{equation*}
where $\parame \sim \mu'$. 
Now, to reduce the computational cost, we combine Algorithm~\ref{Alg_SMK} and reduced bases in \S\ref{Subsec_RB_sampling_algorithm}.
We discuss the computational cost of the offline and the online phase of the suggested reduced basis sampling in \S\ref{Subsec:RB_Comp_Cost}.
Finally, in \S \ref{Subsec:RB_RF_expansion} we explain how the reduced basis induces an alternative expansion for the parameterised Gaussian random field. 

\subsection{Algorithm} \label{Subsec_RB_sampling_algorithm}
First we describe the offline phase of the reduced basis sampling.
We use a POD approach to compute a reduced basis $W$ for $\mathcal{C}(\cdot)$.
Here, it is important that the dimension of $X_\RB = \mathrm{span}(W)$ is larger than $\Nsto$.
Furthermore, we assume that $\mathcal{C}(\cdot)$ fulfils Assumption~\ref{Ass_Linearly_Separable}, i.e., it has the linearly separable form $$\mathcal{C}(\cdot) := \sum_{k=1}^{N_\lin} F_k(\cdot)\mathcal{C}_k.$$
Having constructed the reduced basis $W$, we compute $\mathcal{C}_k^\RB = W^*\mathcal{C}_kW, k =1,\dots,N_\lin$.
\begin{algorithm}[htb]

 \For{$n \in \{1,\dots,N_{\smpl}\}$}{
 Sample $\parame_n \sim \mu'$
 
 $\mathcal{C}^\RB(\parame_n) \leftarrow \sum_{k=1}^{N_\lin} F_k(\parame_n)\mathcal{C}_k^\RB$
 
$\Psi_n^\RB(\parame_n) \leftarrow  \mathrm{eigs}(\mathcal{C}^\RB(\parame_n), \Nsto$)

 $\Psi_n(\parame_n) \leftarrow W\Psi_n^\RB(\parame_n)$
 
 Sample $\vxi \sim \mathrm{N}(0, \mathrm{Id}_{\Nsto})$
 
 $\theta_n \leftarrow m(\parame_n) + \Psi_n(\parame_n) \vxi$
 }
 \caption{Reduced basis sampling from the parameterised Gaussian measure}
 \label{Alg_RBSMK}
\end{algorithm}
Then, we proceed with the online phase (see Algorithm \ref{Alg_RBSMK}).
We iterate over $n = 1,\dots,N_{\smpl}$.
In each step we first sample $\parame_n \sim \mu'$.
Then, we evaluate the reduced covariance operator $\mathcal{C}^\RB(\parame_n)$, compute the eigenpairs $(\lambda^\RB_i(\parame_n),w^\RB_i(\parame_n))_{i=1}^{\Nsto}$ of $\mathcal{C}^\RB(\parame_n)$, and return $\Psi_n^\RB(\parame_n) := \left(\sqrt{\lambda^\RB_i(\parame_n)}w^\RB_i(\parame_n) : i = 1,\dots,\Nsto \right)$.
Next, we compute the representation of  $\Psi_n^\RB(\parame_n)$ on the full space $X$, that is, $\Psi_n(\parame_n) := W\Psi_n^\RB(\parame_n)$.
Finally, we proceed as in Algorithm~\ref{Alg_SMK}:
We sample a multivariate standard Gaussian random variable with $\Nsto$ components and return $m(\parame_n) + \Psi_n(\parame_n) \vxi \sim \mathrm{N}(m(\parame_n), \mathcal{C}^{\RB,  \Nsto}(\parame_n))$.
The covariance operator is given by  $\mathcal{C}^{\RB,  \Nsto}(\parame_n) = \Psi_n(\parame_n)\Psi_n(\parame_n)^*.$

\subsection{Computational cost} \label{Subsec:RB_Comp_Cost}
We assume again that $X  = \mathbb{R}^N$.
The number of solves of the full eigenproblem in the offline phase is $N_{\snap}$.
We consider the following assumptions throughout the rest of the article.
We will see that these crucial to obtain a speed-up with reduced basis sampling.
\begin{assumptions}
	Let $N_{\snap} \ll N_{\smpl}$, $\Nsto \leq N_\RB \ll N$. 
	Moreover, assume that $\mathcal{C}(\paramedet)$ and $ \mathcal{C}^\RB(\paramedet)$ are dense matrices for $\paramedet \in R$.
\end{assumptions}

The computational cost of the tasks in the offline phase is given in Table \ref{Table_Cost_offline_phase}. 
The total offline cost is $O({N_{\snap}N^2}+ {N_{\snap}N^2 \Nsto} + {{N_{\snap}N^2 \Nsto}};N\rightarrow \infty)$.
Since $N_{\snap} \ll N_{\smpl}$ the offline cost is asymptotically much cheaper than the cost of Algorithm~\ref{Alg_SMK} where we solve the full eigenproblem for each sample. 
\begin{table}[htb]
\centering
\begin{tabular}{l|l}\hline
Task & Computational Cost \\ \hline
Construct the full operator & $O(N_{\snap}N^2)$ \\
Solve the full eigenproblem & $O(N_{\snap}N^2 \Nsto)$ \\
POD & $O(N_{\snap}N^2 \Nsto)$ \\ \hline
\end{tabular}
\caption{Computational cost of the offline phase.}
\label{Table_Cost_offline_phase}
\end{table}

The computational cost of the tasks in the online phase is given in Table \ref{Table_Cost_online_phase}.
The total online cost is $O({N_{\smpl}N_\RB^2 N_\lin} + {N_{\smpl}N_\RB^2 \Nsto}+ {N_{\smpl}N_\RB N}; N\rightarrow \infty)$.
\begin{table}[htb]
\centering
\begin{tabular}{l|l}\hline
Task & Computational Cost \\ \hline
Construct the reduced operator& $O(N_{\smpl}N_\RB^2 N_\lin)$ \\
Solve the reduced eigenproblem & $O(N_{\smpl}N_\RB^2 \Nsto)$ \\
Map the reduced solution to the full space & $O(N_{\smpl}N_\RB N)$ \\ \hline
\end{tabular}
\caption{Computational cost of the online phase.}
\label{Table_Cost_online_phase}
\end{table}
In the online phase we solve the covariance eigenproblems in the reduced space.
The high-dimensional full space $X$ is only required when we map the reduced sample to $X$.
The cost of these steps is linear in the dimension $N$ of $X$, and quadratic in the dimension of the reduced basis $N_\RB$ for every sample.
In constrast, the cost of Algorithm \ref{Alg_SMK} is at least quadratic in $N$, for every sample.
Hence, for every sample, we need to solve an $O(N;N\rightarrow \infty)$ problem using RB, but an $O(N^2;N\rightarrow \infty)$ problem in the full space.
This clearly demonstrates the advantages of RB sampling.

\subsection{Reduced basis random field expansion}
\label{Subsec:RB_RF_expansion}
In Algorithm \ref{Alg_RBSMK} we describe sampling from the full random field $\theta$. 
However, this can be inefficient for two reasons:
\begin{enumerate}
\item In some applications, e.g. the Bayesian inverse problem, the random field samples need to be stored and kept in memory. 
This is often impossible due to memory limitations.
\item In some methods, e.g. the MCMC method in Algorithm \ref{Alg_PCNMCMCGibbs}, we need to assess random fields with respect to some covariance operators. 
Hence, not only the random fields have to be kept in memory but also at least one covariance operator.
This requires one order of magnitude more memory than a single random field, that is, $O(N^2; N \rightarrow \infty)$ compared to $O(N; N \rightarrow \infty)$.
\end{enumerate}
Observe that the reduced basis enables a natural compression of the full random field $\theta$.
Let $\theta \sim K(\cdot|\paramedet)$ for some $\paramedet \in R$. 
The reduced basis implies the representation
\begin{equation*}
\theta = m(\paramedet) + \Psi_0(\paramedet) \vxi = m(\paramedet) + W\Psi_0^\RB(\paramedet) \vxi
\end{equation*}
for some $\vxi \sim \mathrm{N}(0, \mathrm{Id}_{\Nsto})$.
We can represent $\theta$ in terms of $\theta_{\RB} := (\Psi_0^\RB(\paramedet) \vxi)\in \mathbb{R}^{{N_\RB}}$
which gives
\begin{equation}\label{RBRF}
\theta = m(\paramedet) + W \theta_\RB,
\end{equation}
the \textit{reduced basis expansion} of the random field $\theta$. 
Note that $\theta_\RB \sim \mathrm{N}(0, \mathcal{C}^{\RB,  \Nsto}(\paramedet))$ is a Gaussian random field on the reduced space.
In contrast to a KL expansion, the random field expansion in \eqref{RBRF} is not necessarily optimal when we compare the number of terms $N_\RB$ required to achieve a certain approximation accuracy in the mean-square error. 
%
There are many alternative random field expansions depending on the choice of the basis, such as the KL expansion of the non-Gaussian random field with covariance operator
 $$\overline{\mathcal{C}} := \int_R \mathcal{C}(\parame) \mu'(\mathrm{d}\parame).$$
or the eigenvectors $(\psi_i(\paramedet))_{i=1}^\infty$ for a fixed $\paramedet \in R$.
Various approaches for random field expansions have been proposed and compared in \cite{Sraj2016}.
We advocate the reduced basis random field expansion, since the offline-online decomposition has been computed for the basis $W$. 
A basis change requires the computation of new reduced operators $\mathcal{C}^{\RB,  \Nsto}(\paramedet)$ and $(\mathcal{C}_k^{\RB})_{k=1}^{N_{\lin}}$, respectively.

Observe that the covariance of $\theta$ for a fixed $\paramedet$ can be fully described by  $\mathcal{C}^{\RB,  \Nsto}(\paramedet) \in\mathbb{R}^{N_\RB \times N_\RB}$.
Hence, we can sample $\theta$ cheaply, and we can also represent the covariance of $\theta$ by a small matrix.
This is particularly useful in Bayesian inverse problems as we explain next.

Consider a Bayesian inverse problem with parameterised prior random field and assume that $m(\paramedet) \in X_\RB$, $\paramedet \in R$.
If this is not the case, one can either add $m(\paramedet)$ for the required $\paramedet \in R$ to $X_\RB$ or project $m(\paramedet)$ orthogonally onto $X_\RB$.
In Algorithm \ref{Alg_PCNMCMCGibbs} we need the full random field $\theta$ as input for the potential $\Phi$ to compute the acceptance probability of the pCN-MCMC proposal.
Moreover, to compute the acceptance probability of the $\tau$-proposal, again we need to construct the full precision operator of the proposed $\tau^* \in R$.
The computational cost of this construction is of order $O(N;N\rightarrow \infty)$.
However, since $\theta_\RB$ already contains the full covariance information, we can replace $\mathcal{C}(\tau^*)^{-1}$ by $\mathcal{C}^{\RB,  \Nsto}(\parame^*)^{-1}$.
Then, the computational cost of the Gibbs step in $\tau$-direction is independent of $N$.
We summarise the associated Reduced Basis MCMC method in Algorithm \ref{Alg_RBMCMC}.

\begin{algorithm}[htb]
Let $(\parame_0, \theta_{\RB,0}) \in \mathbb{R}^{N_\RB + 1}$ be the initial state of the Markov chain.

 \For{$n \in \{1,...,N_{\mathrm{smp}}\}$}{
 Sample $\parame^* \sim q_R(\cdot |\parame_{n-1})$ 
 
$\alpha_R(\parame_{n-1};\parame^*) \leftarrow \min\left\lbrace1, \frac{q_R(\parame_{n-1}|\parame^*)}{q_R(\parame^* |\parame_{n-1})}
\frac{\mathrm{n}(\theta_{\RB,n-1};W^*m(\tau^*),\mathcal{C}^{\RB,  \Nsto}(\tau^*))}{\mathrm{n}(\theta_{\RB,n-1};W^*m(\parame_{n-1}),\mathcal{C}^{\RB,  \Nsto}(\parame_{n-1}))}\right\rbrace$
 
 Sample $U_R \sim \mathrm{Unif}[0,1]$
 
 \eIf{$U_R < \alpha_R$}{
 $\parame_{n} \leftarrow \parame^*$
 }{$\parame_{n} \leftarrow \parame_{n-1}$}
 
 Sample $\theta^*_{\RB} \sim \mathrm{N}(\sqrt{1-\beta^2}\theta_{\RB,n-1}, \beta \mathcal{C}^{\RB,  \Nsto}(\parame_n))$
 
 $\alpha_X(\theta_{\RB,n-1};\theta^*_\RB) \leftarrow \min\{1,\exp(-\Pot(W\theta^*_{\RB})+\Pot(W\theta_{\RB,n-1}))\}$
 
 Sample $U_X \sim \mathrm{Unif}[0,1]$
 
 \eIf{$U_X < \alpha_X$}{
 $\theta_{\RB,n} \leftarrow \theta^*_{\RB}$
 }{$\theta_{\RB,n} \leftarrow \theta_{\RB,n-1}$}
 }
 \caption{Reduced Basis Markov Chain Monte Carlo}
 \label{Alg_RBMCMC}
\end{algorithm}

\section{Numerical experiments}\label{Sec:Numerical_Experiments}
In this section we illustrate and verify the reduced basis sampling for use with forward and Bayesian inverse problems.
We start by measuring runtime and accuracy of the reduced basis approximation to the parametric KL eigenproblems.
In Example \ref{Example_Verification}, we consider a forward and a Bayesian inverse problem in a low-dimensional test setting.
This allows us to compare the reduced basis sampling with the samples obtained by using the full, unreduced KL eigenproblems.
We then move on to high-dimensional estimation problems in Examples \ref{Example_flowcell}--\ref{Example_Bayesian_Estimation}.
Note that we are not able to compute reference solutions in the high-dimensional test cases within a reasonable amount of time since these are computationally very expensive.
Nevertheless, these examples are a proof-of-concept and showcase potential applications.

In Examples \ref{Example_flowcell}--\ref{Example_BIP_PDE} we consider the elliptic PDE 
\begin{equation} \label{EQ_PDE_Num_Exp}
- \nabla \cdot (\exp(\theta(\bsx))\nabla p(\bsx)) = f(\bsx) \qquad (\bsx \in D)
\end{equation}
on the unit square domain $D = (0,1)^2$ together with suitable boundary conditions.
The PDE \eqref{EQ_PDE_Num_Exp} is discretised with linear, continuous finite elements on a uniform, triangular mesh.
The coefficient function $\theta$ is a parameterised Gaussian random field with exponential covariance operator, and random correlation length and standard deviation, respectively (see Example \ref{Exponential Covariance_2}).
The spatial discretisation of $\theta$ is done with piecewise constant finite elements on a uniform, rectangular mesh.
{The evaluation of the covariance operator on this finite element space requires to evaluate an integral.
We approximate this integral using a composite midpoint rule, with one quadrature node in each finite element.}

We further discretise $\theta$ by a truncated KL expansion where we retain the leading $\Nsto$ terms. 
The parameter $\Nsto$ is selected such that the truncated KL captures at least 90\% of the total variance.
We list the random field parameters for Examples~\ref{Example_Verification}--\ref{Example_Bayesian_Estimation} in Table \ref{table_simulation_settings}.
We introduce the estimation problems in more detail in the following subsections. 
Note that we solve the test problems in Examples \ref{Example_Verification}--\ref{Example_Bayesian_Estimation} using the reduced basis samplers presented in \S\ref{Sec_RB_Sampling}.

\begin{table}[htb]
\centering
\begin{tabular}{l|rrrr}\hline
                             &   Example \ref{Example_Verification}  & Example \ref{Example_flowcell} & Example \ref{Example_BIP_PDE}  & Example \ref{Example_Bayesian_Estimation} \\ \hline
$\underline{\sigma}  $     &  1    & 0.1      &     0.5      &   0.1        \\
$\overline{\sigma}$        & 1      & 1       &      0.5     &     1      \\
$m_\sigma    $              & 1      & 0.5      &     0.5      &    0.5       \\
$\sigma_\sigma^2$   & 0& $0.1$    &  0         &     0.1      \\
$\underline{\ell}$     &0.3          & 0.3       & 0.3      & 0.1       \\
$\Nsto$               &200            & 100       & 100       & 800     \\ \hline
\end{tabular}
\caption{
	Random field parameters in Examples~\ref{Example_Verification}--\ref{Example_Bayesian_Estimation}.
}
\label{table_simulation_settings}
\end{table}
\subsection{Accuracy and speed up} \label{Subsec:Accuracy}
First we assess the accuracy and time consumption of the reduced basis approximation. 
We measure the quality of the reduced basis surrogate by comparing reduced basis eigenvalues with full eigenvalues in a simplified setting.
 The \textit{full matrix} is the finite element approximation of $\mathcal{C}_{\exp}^{(\ell,1)}$ with $100 \times 100$ piecewise constant finite elements.
The goal is to compute the leading 100 eigenpairs of $\mathcal{C}_{\exp}^{(\ell,1)}$ for selected values $\ell \in [0.1,\sqrt{2}]$.
We compute reference solutions for $\ell = 0.1, 0.5, 1.4$ using the full matrix.
The reduced basis is constructed using $10$ snapshots $$\ell^{\snap} = (2^{1/2}, (2^{-1/2}+1)^{-1}, (2^{-1/2}+2)^{-1},\dots,(2^{-1/2}+9)^{-1}).$$
We compute the leading 100 eigenpairs for all correlation lengths in  $\ell^{\snap}$, and assemble the associated eigenvectors in a single matrix.
Then we apply the POD and retain $N_\RB = 2^1,\dots,2^{13}$ orthonormal basis vectors.
Recall that the offline-online decomposition requires a linearisation of the covariance operator (see \S\ref{Subsec:Matern}).
Throughout this section (\S\ref{Sec:Numerical_Experiments}) we retain $N_\lin = 39$ linearisation terms. 
In this case the truncation error defined in \eqref{Eq_truncationerror_covar_kernel} is $9.09\cdot 10^{-5}$ for the linearisation.

\begin{figure}[htb]
\centering
\includegraphics[scale=0.65]{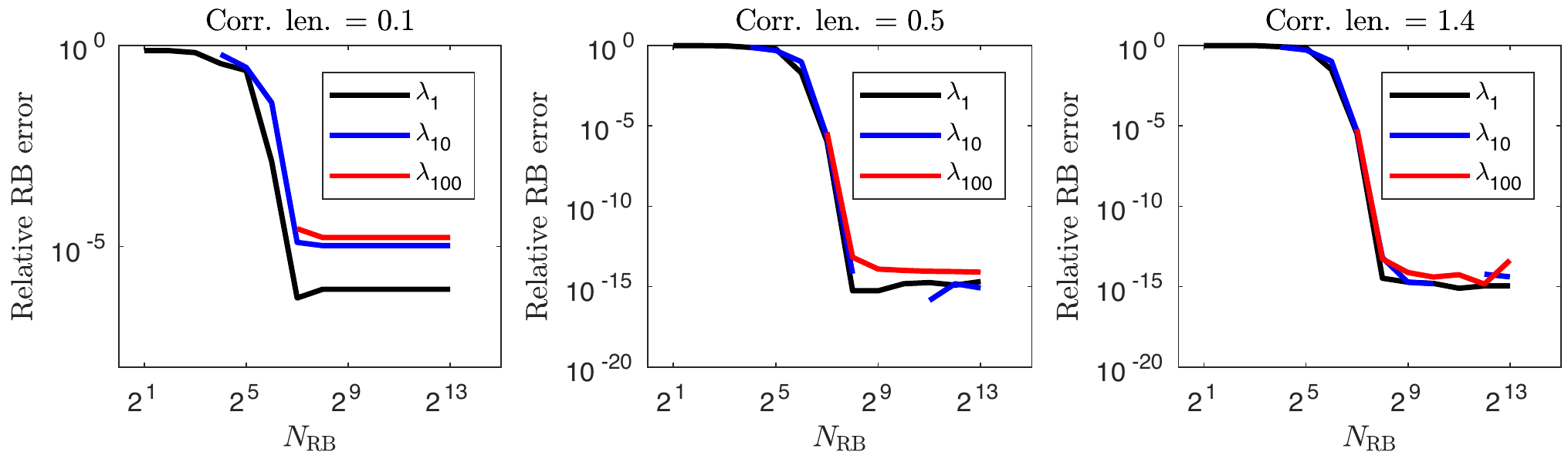}
\caption{Relative reduced basis error of the eigenvalues $\lambda_1(\ell),\lambda_{10}(\ell),\lambda_{100}(\ell)$ for correlation lengths $\ell = 0.1, 0.5, 1.4$ and reduced basis dimensions $N_\RB = 2^1,...,2^{13}$.}
\label{Fig_RB_Err}
\end{figure}

We plot the relative error of the reduced eigenvalue compared to the exact eigenvalue in Figure~\ref{Fig_RB_Err} for various reduced basis dimensions $N_\RB$.
Note that it is not possible to compute eigenvalues with an index larger than $N_\RB$.
For $\ell = 0.1$ the relative RB error stagnates at a level that is not smaller than $ 10^{-6}$. 
In further experiments not reported here we observed that this stagnation is caused by the linearisation error of the covariance kernel (recall that we use $N_\lin = 39$ terms with an error of order $10^{-5}$).
We remark that the root mean square error of the MC and MCMC estimation results in this section is of order $O(10^{-2})$.
Hence, an eigenvalue error of magnitude $ 10^{-6}$ is acceptable.
We point out, however, that a full error analysis of the reduced basis samplers (including the linearisation and RB error) is beyond the scope of this study.
For $\ell = 0.5$ and $\ell=1.4$ we achieve an accuracy of order $O(10^{-6})$ for $N_\RB \approx 128$.
For $N_\RB > 128$ the relative errors are of the size of the machine epsilon.
This error is unnecessarily much smaller than the sampling error mentioned above and introduces a higher computational cost in the online phase.
Hence, in our test problems $N_\RB = 128$ would be a sufficient choice.

To explore the speed-ups that are possible with reduced basis sampling we repeat the experiment.
This time we vary the dimension of the finite element space and use $N = 4^4,\ldots, 4^7$.
The dimension of the reduced basis is fixed with $N_\RB = 256$.
We plot the test results in Figure \ref{Fig_RB_timings}.
The time measurements correspond to serial simulations in \textsc{Matlab} with an {Intel i7 (2.6 Ghz)} CPU and {16 GB RAM} memory.
The dashed lines show the theoretical asymptotic behaviour, that is, $O(N; N \rightarrow \infty)$ for the reduced basis sampling and $O(N^2; N \rightarrow \infty)$ for the full sampling.
\begin{figure}[htb]
\centering
\includegraphics[scale=0.825]{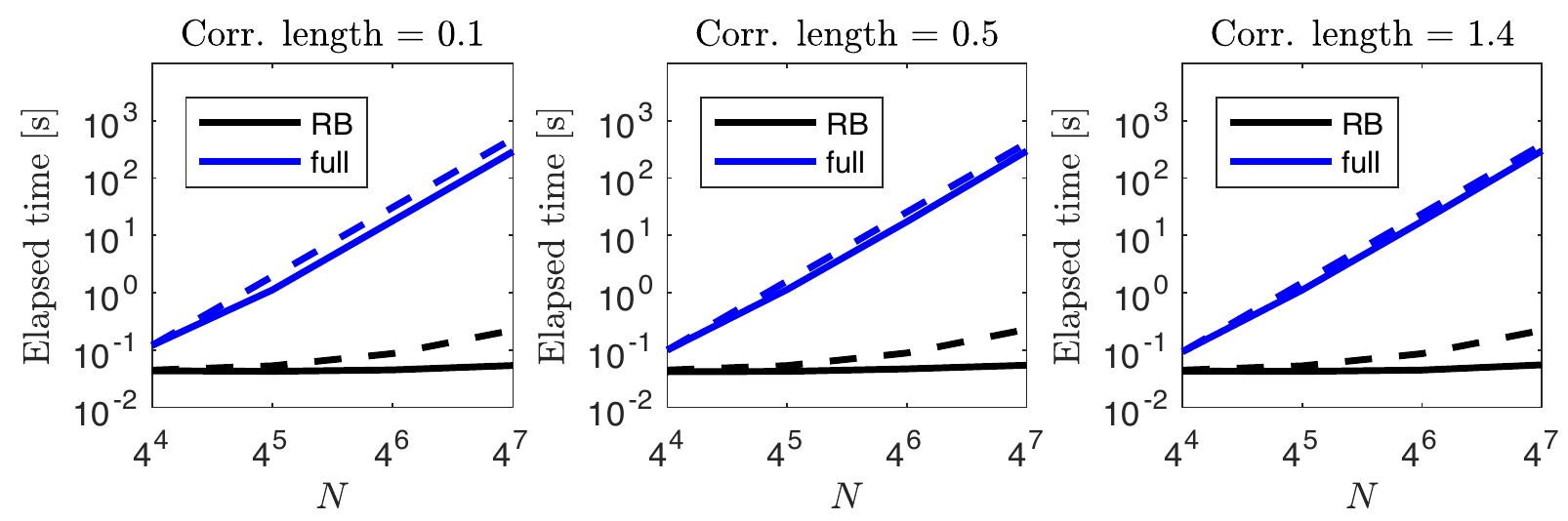}
\caption{Timings for the full and reduced problem with different FE resolutions and correlation lengths. The elapsed time is shown as solid line, and the asymptotic behaviour is shown as dashed line. }
\label{Fig_RB_timings}
\end{figure}
We see that the theoretical and observed timings for the full sampling are almost identical. 
In contrast, the observed timings for RB sampling are smaller than predicted by the theory.
This is caused by the fact that the dimension $N$ of the finite element space is quite small.
As $N$ increases we observe a massive speed-up of the reduced basis sampling compared to the full sampling.
For example, for $N = 4^7$ the reduced basis sampling requires less than $10^{-1}$ seconds, while the full sampling requires several minutes. 
In this case the speed-up in the online phase is of order $\mathcal{O}(10^3)$.

For the estimation problems in Examples~\ref{Example_Verification}--\ref{Example_Bayesian_Estimation} we use Monte Carlo and Markov Chain Monte Carlo with $10^4$ up to $1.5\cdot 10^5$ samples and a random field resolution with $N = 256 \times 256$ finite elements in space.
In these cases, MC and MCMC estimations based on the full KL eigenproblem would take a couple of days up to a couple of years to terminate.
In contrast, the (serial) run-time of the reduced basis sampling is $\sim 15$ minutes in Example \ref{Example_flowcell} and $\sim 18$ hours in Example \ref{Example_Bayesian_Estimation}.
Of course, standard Monte Carlo simulations are trivially parallelisable. 
MCMC is a serial algorithm by design, and parallelisation is not trivial, see e.g. \cite{VanDerwerken2013} for suitable strategies.
Our experiments show that RB sampling can reduce the computational cost without the need for parallelisation.

\subsection{Verification of reduced basis sampling} \label{Subse:Verification}
Next, we test the accuracy of RB sampling using coarse spatial discretisations.
This allows us to obtain reference solutions in a reasonable amount of time.
\begin{example} \label{Example_Verification}
Let $\mu$ be the joint probability measure in Example \ref{Exponential Covariance_2} together with the parameter values given in Table \ref{table_simulation_settings}.
We discretise the random field $\theta \sim \mu''$ using $32^2$ finite elements.
The test problems are as follows.
\begin{itemize}
\item[(a)] Forward uncertainty propagation: We consider a flow cell in 2D with log-permeability $\theta$. See Example \ref{Example_flowcell} for the definition of the flow cell.
Given the random coefficient $\theta$ we want to estimate the probability distribution of the outflow over the boundary $\mathcal{Q}(\theta)$.
We discretise the PDE with $2 \cdot 16^2$ finite elements.
\item[(b)] Bayesian inverse problem: We observe a random field realisation on $D=(0,1)^2$ at nine points in the spatial domain \begin{align*}D_{\mathrm{obs}} := \{(n/4, m/4): n,m = 1,2,3\}.
\end{align*}
In each of the points we observe the value $0.1$, which we assume to be noisy. We want to reconstruct the random field.
The prior measure is $\mu$ as specified above.
The likelihood is given by $$\exp\left(-\frac{1}{2 \cdot 10^{-2}}\sum_{x \in D_{\mathrm{obs}}}(0.1 - \theta(\bsx))^2\right).$$
We want to estimate the posterior mean and variance of the correlation length $\ell$ given the data $y$.
In addition, we compute the model evidence.
Note that this test problem is similar to Example \ref{Example_introductory} in the Introduction, see also the corresponding Figure \ref{Fig_Intro_Bayes}.
\end{itemize}
\end{example}
We solve the test problems in Example~\ref{Example_Verification}(a)--(b) with Monte Carlo.
In part (b) we use importance sampling with samples from the prior as proposal.
We solve (a) and (b) with reduced basis sampling as well as standard sampling based on the full (discretised) eigenvectors of the parameterised covariance operator.
The standard sampling serves as reference solution for the reduced basis sampling.
The reduced basis is constructed using the snapshot correlation lengths $\ell^\snap = (0.322, 0.433, 0.664, 1.414)$; these are simply the inverses of four equidistant points in the interval $[1/\sqrt{2}, 3.11]$ including the boundary points.
This choice clusters snapshots near zero which is desirable due to the singularity of the exponential covariance at $\ell = 0$. 
We apply the POD to construct three reduced bases with different accuracies $\underline{\lambda}:=10^{-1},10^{-5},10^{-9}$. 

For each of the settings we run 61 Monte Carlo simulations with $10^4$ samples each to estimate the mean and the variance of the pushforward measure $\mu''(\mathcal{Q}\in \cdot)$ in part (a), as well as the posterior mean, posterior variance and model evidence in the Bayesian inverse setting in part (b).
We compute a reference solution for all those quantities using $6.15 \times 10^5$ samples.
With respect to the reference solutions we compute the relative error of the 61 estimates in each setting. 
In Table \ref{table_Verification} we give the means and the associated standard deviations (StD) of the relative errors.
We observe that the (mean of the) relative error is of order $O(10^{-2})$ up to $O(10^{-3})$.

Moreover, in Table~\ref{table_Verification} we list the sample mean of the error between the full covariance operators and their representations on the reduced basis, measured in the Frobenius norm.
That is, we list the Monte Carlo estimate using $6.15 \times 10^5$ samples of the expression
\[
\mathbb{E}\left[\Vert \mathcal{C}^{\Nsto}-\mathcal{C}^{\text{RB},\Nsto}\Vert_F\right]
=
\int_{R} \Vert \mathcal{C}^{\Nsto}(\tau)-\mathcal{C}^{\text{RB},\Nsto}(\tau)\Vert_F \ \mathrm{d}\mu'(\tau).
\]
We observe that the error decreases as we include more vectors in the reduced basis, as expected.
\begin{table}[htb]
\centering
\begin{tabular}{l|llllll}\hline
    $\underline{\lambda}$                     & $10^{-1}$ & (StD)   & $10^{-5}$ & (StD) & $10^{-9}$ & (StD)\\ \hline
   
Pushforward mean  & 0.0055 & (0.0041) & 0.0057 &(0.0051) &0.0054 &(0.0040)\\

Pushforward variance  & 0.0451 & (0.0262) & 0.0442 &(0.0352) &0.0321 &(0.0282)\\ \hline

Evidence  & 0.0147 & (0.0262) & 0.0122 &(0.0352) & 0.0154 &(0.0282)\\

Posterior mean  & 0.0087 & (0.0070) & 0.0075 &(0.0062) &0.0082 &(0.0063)\\

Posterior variance  & 0.0856 & (0.0691) & 0.0733 &(0.0608) &0.0807 &(0.0619)\\ \hline
Mean covariance error & $2.609 \times 10^{-4}$&  & $1.240 \times 10^{-7}$ & & $3.606 \times 10^{-11}$ & \\\hline
\end{tabular}
\caption{Relative errors in the Monte Carlo estimation of the mean and variance of the pushforward measure, and posterior mean, posterior variance and model evidence in the Bayesian inverse problem (Example~\ref{Example_Verification}). 
Each error value is the mean taken over 61 simulations with $10^4$ samples each. 
The simulations are performed with reduced basis sampling with POD accuracies $\underline{\lambda}=10^{-1},10^{-5},10^{-9}$. 
The relative errors are computed with respect to a reference solution computed with $6.15 \times 10^{5}$ samples based on the full eigenproblem.
In the last line of the table we list the mean error between the full covariance operator and the operator represented in the reduced basis, measured in the Frobenius norm.}
\label{table_Verification}
\end{table}

\subsection{Forward uncertainty propagation} \label{Subsect_Forward_UQ}
Next we study the forward uncertainty propagation of a hierarchical random field by an elliptic PDE operator.
\begin{example} \label{Example_flowcell}
Consider a flow cell problem on $D=(0,1)^2$ where the flow takes place in the $x_1$-direction.
The boundary conditions are as follows,
\begin{align*}
p(\bsx) &= 0  \ \ \ \ (\bsx \in \{1\} \times [0,1]),\\
p(\bsx) &= 1 \ \ \ \ (\bsx \in \{0\} \times [0,1]),\\
\frac{\partial p}{\partial \vec{n}}(\bsx) &= 0   \ \ \ \ (\bsx \in (0,1) \times \{0,1\}).
\end{align*}
There are no sources within the flow cell ($f \equiv 0$). 
The random field $\theta$ is as in Example \ref{Exponential Covariance_2} with the parameters given in Table \ref{table_simulation_settings}. 
The PDE is discretised with $2 \times 128^2$ finite elements, and the random field with $256^2$ finite elements.
The quantity of interest is the outflow over the (western) boundary $\Gamma_{\mathrm{out}}:=\{0\} \times [0,1]$. 
It can be approximated by
$$\mathcal{Q}(\theta) := - \int_D \kappa(\theta) \nabla p \cdot \nabla \psi \mathrm{d}\bsx,$$
where $\psi|_{D\backslash\Gamma_{\mathrm{out}}} \approx 0$ and $\psi|_{\Gamma_{\mathrm{out}}} \approx 1$ (see e.g. \cite{Douglas1974}).
We discretise the outflow using a piecewise linear, continuous finite element function $\psi$ on the same mesh that we used for the PDE discretisation.
\end{example}

The log-permeability is modelled as a parameterised Gaussian random field.
We employ the reduced basis sampling with $N_\RB=191$. 
We construct the reduced basis analogously to the simple test setting in \S \ref{Subse:Verification}.
However, now we let $\Nsto = 100$, and remove vectors from the POD where the corresponding eigenvalue is smaller than $\underline{\lambda} = 10^{-10}$ (see also \S \ref{Subsubsec_Offline_Phase}).

We estimate the mean and variance of the output quantity of interest. 
We compare 24 estimations by computing the associated coefficient of variation (CoV) for the mean and variance estimator, respectively.
The CoV is defined as the ratio of the standard deviation of the estimator and the absolute value of its mean.
We present the estimation results in Table \ref{table_Sim_Forward}.
\begin{table}[htb]
\centering
\begin{tabular}{l|llll}\hline
                         & Mean & (CoV)    & Variance & (CoV) \\ \hline
MC estimate  & 157.286 & (0.0028) & 3012.2 &(0.0355)  \\ \hline
\end{tabular}
\caption{Mean and variance estimates with $10^4$ samples (Example \ref{Example_flowcell}). 
We compare these estimates to 23 further simulation results by computing the coefficient of variation within the 24 estimates.}
\label{table_Sim_Forward}
\end{table}
The small CoVs tell us that $N_{\smpl} = 10^4$ samples were sufficient to estimate the pushforward measure of the quantity of interest, as well as its mean and variance.
Note that with the reduced basis sampling a single Monte Carlo simulation run took about 18 minutes.

\subsection{Hierarchical Bayesian inverse problem}
We consider two hierarchical Bayesian inverse problems based on random fields. 
Note that we use again $256^2$ finite elements to discretise the random fields in both problems and $2 \times 128^2$ finite elements to discretise the elliptic PDE in Example~\ref{Example_BIP_PDE}.

\begin{example} \label{Example_BIP_PDE}
Consider the Bayesian estimation of a random field and its correlation length. The true underlying random field is propagated through the elliptic PDE \eqref{Eq_param_Gaussian} together with Dirichlet boundary conditions
\begin{equation*}
p(\bsx) = 0 \ \ \ \ (\bsx \in \partial D),\\
\end{equation*}
and 9 Gaussian-type source terms
\begin{equation*}
f(\bsx) = \sum_{n,m=1}^3 \mathrm{n}\left(\bsx_1;0.25n,0.001\right) \cdot \mathrm{n}\left(\bsx_2;0.25m,0.001\right).
\end{equation*}
We observe the solution $p$ at 49 locations.
In particular, the observation operator is given by
\[
\mathcal{O}(p) := (p(n/8,m/8) : n, m = 1,...,7).
\]
 \begin{figure}[htb]
\centering
\includegraphics[scale=0.2]{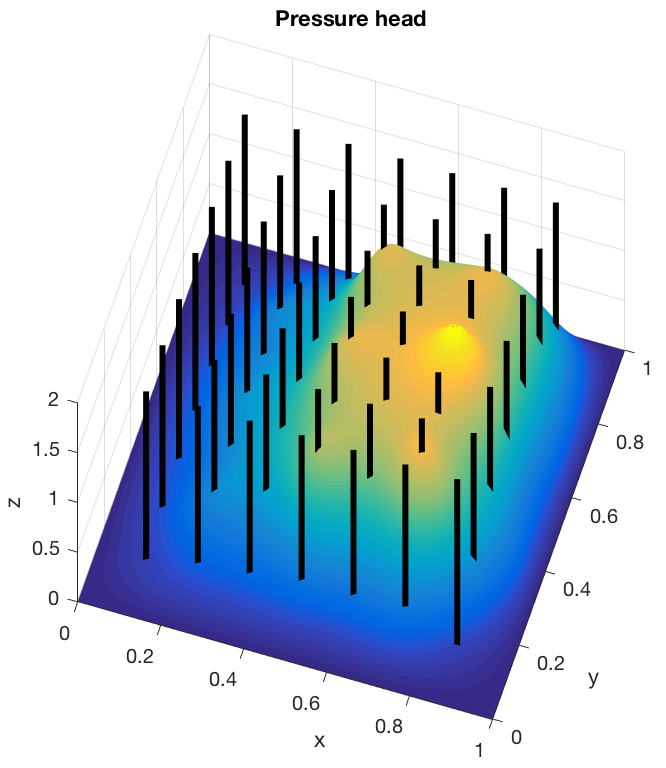} \hspace{0.75cm}
\includegraphics[scale=0.2]{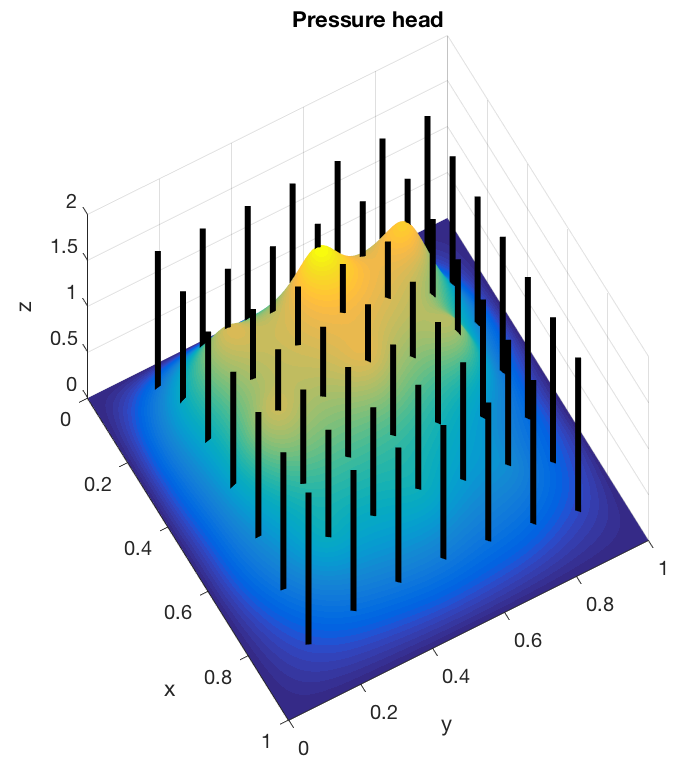}
\caption{Synthetic solution $p$ observed in the Bayesian inverse problem (Example \ref{Example_BIP_PDE}). The black lines indicate the measurement locations. 
	The figures show the model output for a log-permeability with correlation length $\ell = 0.5$ (left) and $\ell = 1.1$ (right), respectively.}
\label{Figure_PDE_pressure_Measurementlocation}
\end{figure}
The (synthetic) observations are generated with log-permeability fields that are samples of a Gaussian random field with exponential covariance operator with $\ell \in \{0.5, 1.1\}$ and $\sigma = \sqrt{1/2}$.
We show the corresponding PDE outputs and the measurement locations in Figure~\ref{Figure_PDE_pressure_Measurementlocation}.
The Gaussian random fields have been sampled with the full (unreduced) $\Nsto = 100$ leading KL terms.
Every observation is perturbed with i.i.d. Gaussian noise $\eta_1 \sim \mathrm{N}(0, 10^{-6})$.
We use the measure $\mu$ in Example \ref{Exponential Covariance_2} with parameter values given in Table \ref{table_simulation_settings} as prior measure. 
\end{example}

\begin{example} \label{Example_Bayesian_Estimation}
Consider the Bayesian estimation of a Gaussian random field together with its standard deviation and correlation length.
We observe the field directly, however, the observations are again noisy.
The estimations are performed with two data sets that have been generated with fixed hyperparameters $\ell \in \{0.2, 1.1\}$ and $\sigma = 1/(\sqrt{2} \cdot 256)$. 
We set $\sigma = 1/\sqrt{2}$ and rescale the KL eigenfunctions by $1/256$. 
The random field discretisation uses an $\Nsto = 800$ dimensional full (unreduced) KL basis.
We observe the random field at $2500$ positions.
Each observation is perturbed by i.i.d. Gaussian noise $\eta_1 \sim \mathrm{N}(0, 10^{-6})$.
The prior measure is the measure $\mu$ in Example \ref{Exponential Covariance_2} with parameter values given Table \ref{table_simulation_settings}. 
\end{example}

Note that in Examples~\ref{Example_BIP_PDE}--\ref{Example_Bayesian_Estimation} we use the same PDE and random field discretisation for the generation of the data and the estimation problem.
The reason is that we are mainly interested in the reduced basis error, and not in the reconstruction error of the inverse problem.
Note further that in Examples~\ref{Example_BIP_PDE}--\ref{Example_Bayesian_Estimation} the standard deviation $\sigma = \sqrt{1/2}$ is fixed a priori, and is not estimated.
The hierarchical Bayesian inverse problems in Example~\ref{Example_BIP_PDE}--\ref{Example_Bayesian_Estimation} are well-posed since the associated Bayesian inverse problem with fixed, deterministic hyperparameters is well-posed (see \cite{Dashti2011}), and since the hyperparameter set $R$ is compact.

\subsubsection{Observations from PDE output} \label{Subsubsect_Observations_from_PDE_Output}
We consider Example~\ref{Example_BIP_PDE} and the settings in Table~\ref{table_simulation_settings}.
The Reduced Basis MCMC method presented in Algorithm~\ref{Alg_RBMCMC} is used to sample from the posterior measure.
The correlation length $\ell \in [0.3, \sqrt{2}]$.
Since this is the same range as in Example \ref{Example_flowcell} we reuse the reduced basis computed in Example \ref{Example_flowcell}.
Recall that the standard deviation $\sigma = 1/\sqrt{2}$ of the random field  $\theta$ is fixed and not estimated.
Moreover, we assume that the observational noise is given by $\eta \sim \mathrm{N}(0,10^{-3} \mathrm{Id})$. 
This corresponds to a noise level of $\sqrt{10^{-3} }/\|y\|_Y \approx 0.6\%$.

We perform experiments for two synthetic data sets with $\ell=0.5$ and $\ell=1.1$, respectively.
For both data sets we compute a Markov chain of length $N_{\smpl}  = 10^5$.
To avoid burn-in effects the initial states are chosen close to the true parameter values for the Markov chains.
In a setting with real world data it is often possible to obtain suitable initial states with Sequential Monte Carlo (see \cite{Beskos2015}).

\begin{figure}[htb]
\centering
\includegraphics[scale=0.52]{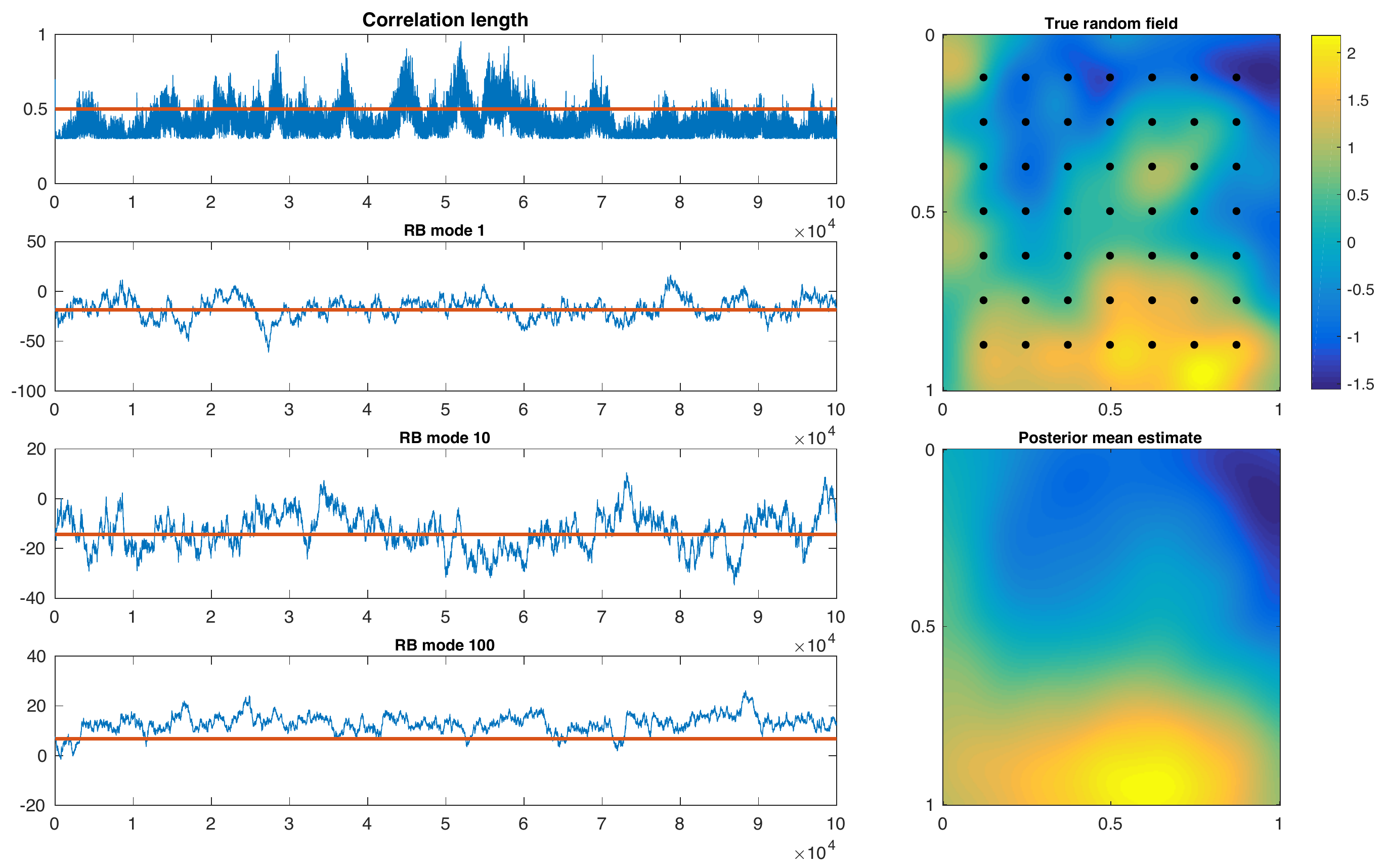}
\caption{Results of the MCMC estimation (Example \ref{Example_BIP_PDE}, $\ell = 0.5$).  
	The top-right plot shows the synthetic truth together with the measurement locations (black dots). Below we plot the posterior mean estimate computed with MCMC. The four path plots on the left side of the figure show the Markov chains for the correlation length $\ell$, and the reduced basis modes $(\theta_\RB)_1$, $(\theta_\RB)_{10}$, and $(\theta_\RB)_{100}$, respectively. The red lines mark the truth.}
\label{FIgure_MCMC_BIP_ell0_5}
\end{figure}

\begin{figure}[htb]
	\centering
	\includegraphics[scale=0.52]{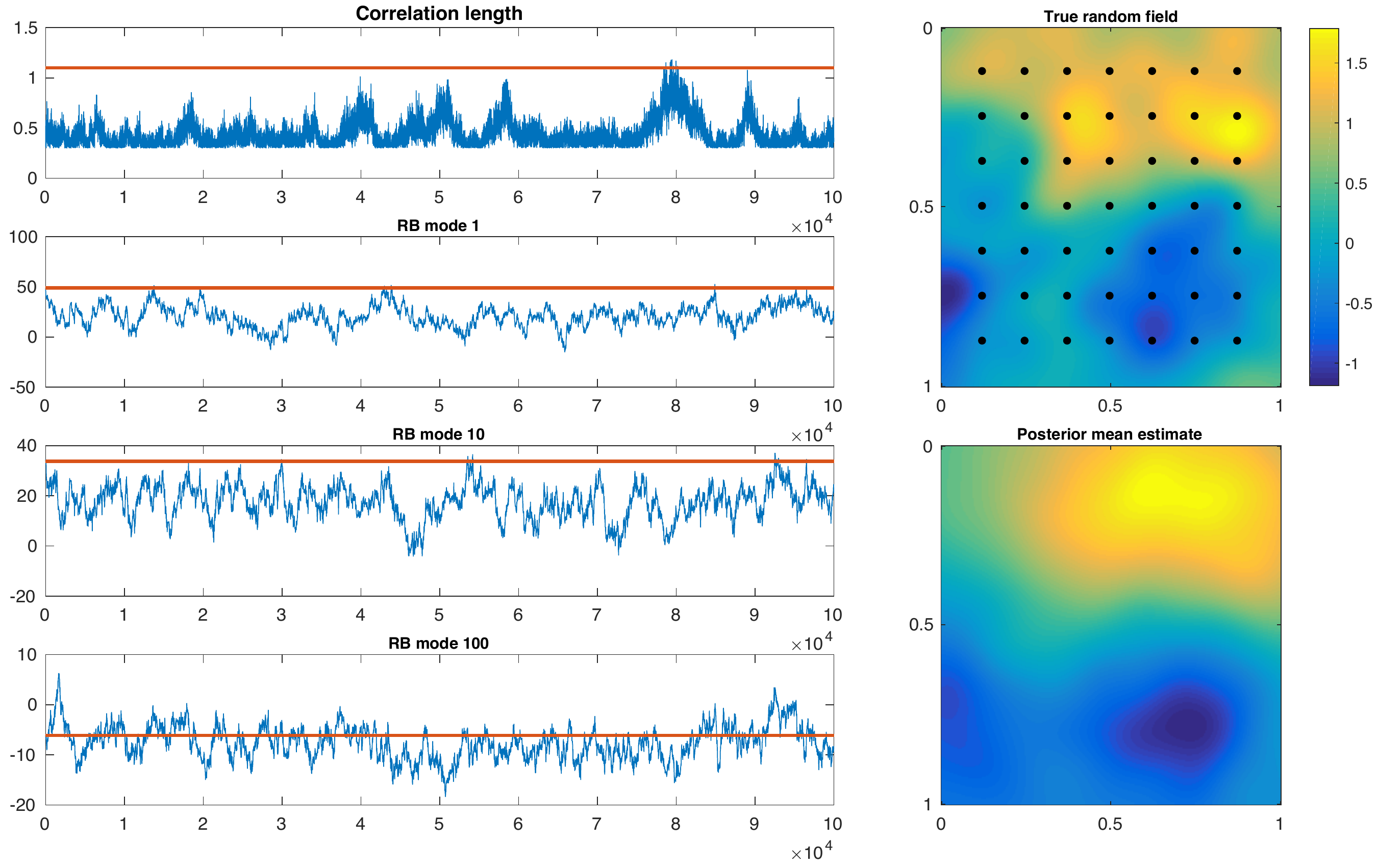}
	\caption{Results of the MCMC estimation (Example \ref{Example_BIP_PDE}, $\ell = 1.1$).  
		The top-right plot shows the synthetic truth together with the measurement locations (black dots). Below we plot the posterior mean estimate computed with MCMC. The four path plots on the left side of the figure show the Markov chains for the correlation length $\ell$, and the reduced basis modes $(\theta_\RB)_1$, $(\theta_\RB)_{10}$, and $(\theta_\RB)_{100}$, respectively. The red lines mark the truth.}
	\label{FIgure_MCMC_BIP_ell1_1}
\end{figure}

The estimation results are depicted in Figure \ref{FIgure_MCMC_BIP_ell0_5} and Figure \ref{FIgure_MCMC_BIP_ell1_1}.
We observe in both figures that the Markov chain for $\ell$ mixes very fast, however, it takes some time for the Markov chains of the reduced basis modes to explore the whole space.
To investigate this further we conduct a heuristic convergence analysis.
To this end we consider multiple Markov chains (see \S 12.1.2 in \cite{Robert2004} for a review of MCMC convergence analysis with multiple Markov chains).
For each of the two test data sets we compute $4$ additional Markov chains starting at different initial states. 
In results not reported here we observed a similar mixing and coverage of the parameter space of the additional chains.
Given these mixing properties, it can reasonably be assumed that the Markov chains have reached the stationary regime.

Moreover, we have computed posterior mean and posterior variance of the correlation length parameter $\ell$ for each of the 5 Markov chains.
The accuracy of these estimates is assessed by computing the coefficient of variation (CoV) within these five estimates.
We tabulate the posterior mean and variance estimates for $\ell$ of a single Markov chain as well as the associated CoVs in Table \ref{Table_Estimation_results_corr_BIP}.
The single Markov chains in this table are the chains shown in Figures \ref{FIgure_MCMC_BIP_ell0_5}-\ref{FIgure_MCMC_BIP_ell1_1}.
The coefficients of variation of the posterior mean and variance estimates are considerably small. 
This tells us that the posterior mean and variance estimates are reasonably accurate.

\begin{table}[htb]
\centering
\begin{tabular}{l|llll}\hline
                         & Mean & (CoV)  & Variance & (CoV) \\ \hline
MCMC $\ell$ given $y$ (Truth: $\ell = 0.5$)   & 0.4105 & (0.0040)& 0.0081 & (0.3235)  \\
MCMC $\ell$ given $y$ (Truth: $\ell = 1.1$)   & 0.4403& (0.0524) & 0.0157 & (0.2346)  \\ \hline
\end{tabular}
\caption{Estimation results of the Bayesian inverse problem with observations from PDE output (Example \ref{Example_BIP_PDE}). We tabulate the posterior mean and variance estimates of the correlation length $\ell$ of one Markov chain each and the CoVs within the estimates of 5 different Markov chains.}
\label{Table_Estimation_results_corr_BIP}
\end{table}
\paragraph{Discussion of the estimation results}
The correlation length is underestimated in both cases. 
In the first case, where the true parameter is given by $\ell = 0.5$, the posterior mean is close to the true parameter. 
The relative distance between truth and posterior mean is about 18\%.
In the second setting, where in truth $\ell = 1.1$, the posterior mean is far away from the true parameter. 
Here, the relative distance between truth and posterior mean is about 60\%.
In both cases, we conclude that the data likelihood was not sufficiently informative to estimate the correlation length more accurately.
However, we note that the succession of the estimates is correct: The posterior mean estimate in the problem with the larger true correlation is larger than the posterior mean estimate in the other case.
Hence we observe a certain consistency with the data in the estimation.

\subsubsection{Observations from a random field} \label{Subsubsect_Observations_from_RF}
Finally, we consider Example \ref{Example_Bayesian_Estimation}.
Here, we allow for much smaller correlation lengths $\ell \in [0.1, \sqrt{2}]$. 
Moreover, we consider an uncertain standard deviation $\sigma$.
This requires more KL terms for an accurate approximation, in particular, we use the leading 800 KL terms.
This also means that we cannot reuse the reduced basis computed in Example \ref{Example_flowcell}.
Instead, we construct a reduced basis as follows.
We solve the KL eigenproblem for $5$ snapshots $$\ell^\snap =  (0.1148, 0.1491, 0.2124, 0.3694, 1.4142)$$ of the correlation length.
The rationale behind this choice is explained in \S \ref{Subsect_Forward_UQ}.
Given the collection of snapshot KL eigenvectors we apply a POD and retain only the basis vectors with $\lambda^{\snap}_i \geq \underline{\lambda} = 10^{-10}$.
Note that we can compute the dependency of $\sigma$ on the eigenpairs analytically and that we do not need to consider them when constructing the reduced basis.

Recall that in this example the observational noise is given by $\eta \sim \mathrm{N}(0,10^{-4} \mathrm{Id})$. This corresponds to a noise level of $\sqrt{10^{-4} }/\|y\|_Y \approx 6.6\%$.
\begin{figure}[htb]
\centering
\includegraphics[scale=0.45]{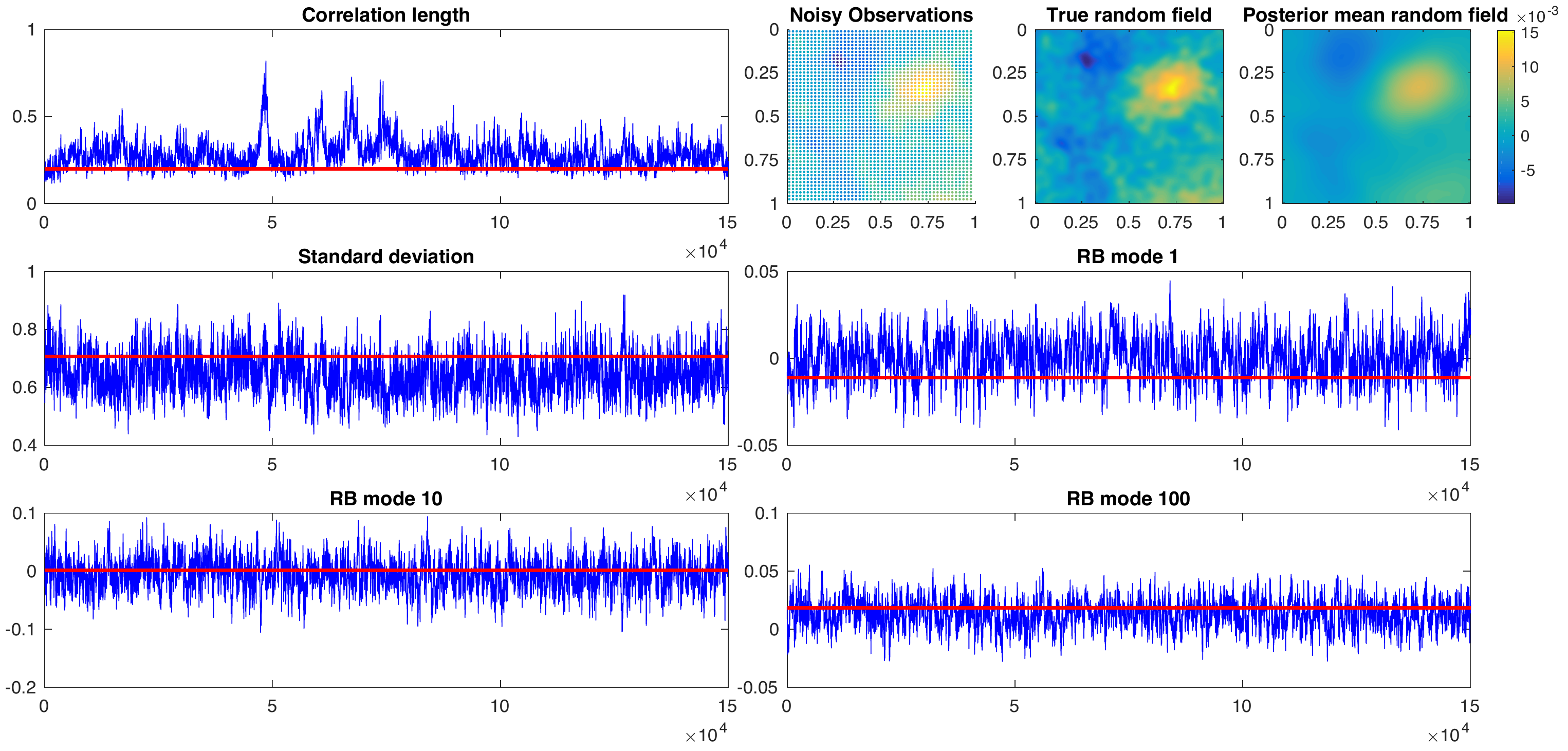}
\caption{Results of the MCMC estimation given random field observations (Example \ref{Example_Bayesian_Estimation}, $\ell=0.2$, $\sigma=1/\sqrt{2}$). 
	In the top-right corner we plot the positions and values of the noisy observations (left), the synthetic truth (middle), and the posterior mean (right). 
	The five path plots show the Markov chains for  $\ell$ and $\sigma$, and the reduced basis modes $(\theta_\RB)_1$, $(\theta_\RB)_{10}$, $(\theta_\RB)_{100}$, respectively. The red lines mark the truth.}
\label{Figure_MCMC_RF_0_2}
\end{figure}
We employ the Reduced Basis MCMC sampler in Algorithm \ref{Alg_RBMCMC} to generate $N_{\smpl}  = 1.5 \times 10^5$ samples of the posterior measure.
We present the Markov chains and estimation results in Figure \ref{Figure_MCMC_RF_0_2} and in Figure \ref{Figure_MCMC_RF_1_1}.
We observe a fast mixing of the Markov chains.
To conduct a heuristic convergence assessment we again compute 4 additional Markov chains with $N_{\smpl}  = 1.5 \times 10^5$ samples each and different initial states.
We found that the additional Markov chains mix similarly compared to the Markov chains shown in Figures \ref{Figure_MCMC_RF_0_2}--\ref{Figure_MCMC_RF_1_1}.
They also cover the same area of the parameter space.
Hence, it can reasonably be concluded that the Markov chains reached a stationary regime.
\begin{table}[htb]
\centering
\begin{tabular}{l|llll}\hline
                         & Mean & (CoV)  & Variance & (CoV)  \\ \hline
MCMC $\ell$  given $y$ (Truth: $\ell = 0.2$)   & 0.2847 & (0.0465) & 0.0064 & (0.1520)    \\ 
MCMC $\sigma$  given $y$ (Truth: $\sigma =1/\sqrt{2}$)   & 0.6438 & (0.0077) & 0.0042& (0.0207)   
\\ \hline
MCMC $\ell$  given $y$ (Truth: $\ell = 1.1$)   & 0.7248 & (0.0161) & 0.0575 & (0.1308)   \\ 
MCMC $\sigma$  given $y$ (Truth: $\sigma =1/\sqrt{2}$)   & 0.5484 & (0.0096) & 0.0052  & (0.0453)  
\\ \hline
\end{tabular}
\caption{Estimation results of the Bayesian inverse problem with observations from a random field (Example \ref{Example_Bayesian_Estimation}). We tabulate the posterior mean and variance of the correlation length $\ell$ and standard deviation $\sigma$.}
\label{Table_Estimation_results_corr_RF_obs}
\end{table}

In addition we present in Table \ref{Table_Estimation_results_corr_RF_obs} the posterior mean and posterior variance estimates of $\ell$ and $\sigma$ associated with the Markov chains given in Figures \ref{Figure_MCMC_RF_0_2}--\ref{Figure_MCMC_RF_1_1}.
To assess the accuracy of these estimates we compare them with the posterior mean and variance estimates of the 4 other Markov chains by computing the coefficients of variations of the estimators.
Again, the coefficients of variation are reasonably small.
\paragraph{Discussion of the estimation results}
While the likelihood was rather uninformative in the PDE-based Bayesian inverse problem, we see overall more consistent estimates in Example~\ref{Example_Bayesian_Estimation}.
For the short correlation length $\ell=0.2$ the relative distance between posterior mean and truth is 42\%.
The long correlation length $\ell=1.1$ is again underestimated.
The relative distance between truth and posterior mean is 34\% in this case.
This result could be explained by the uncorrelated noise that has an influence on the observation of the correlation structure.
In particular, we actually observe a random field $\theta' := \theta + \eta'$, where $\theta \sim \mathrm{N}(0, \mathcal{C}_{\exp}^{({\ell}, {\sigma})})$ and $\eta' \sim \mathrm{N}(0, \sigma^2_* \cdot \mathrm{Id}_X)$, for some $\sigma^2_* > 0$.
In this situation, the random field $\eta'$ can be understood as a random field with correlation length $0$.
This might explain the underestimation of the correlation lengths.
The standard deviations are slightly underestimated and some of the reduced basis modes are overestimated -- this is a consistent result.
The posterior mean random fields appear to be smoother than the true random fields.
This might be due to the high noise level.

\begin{figure}[htb]
\centering
\includegraphics[scale=0.45]{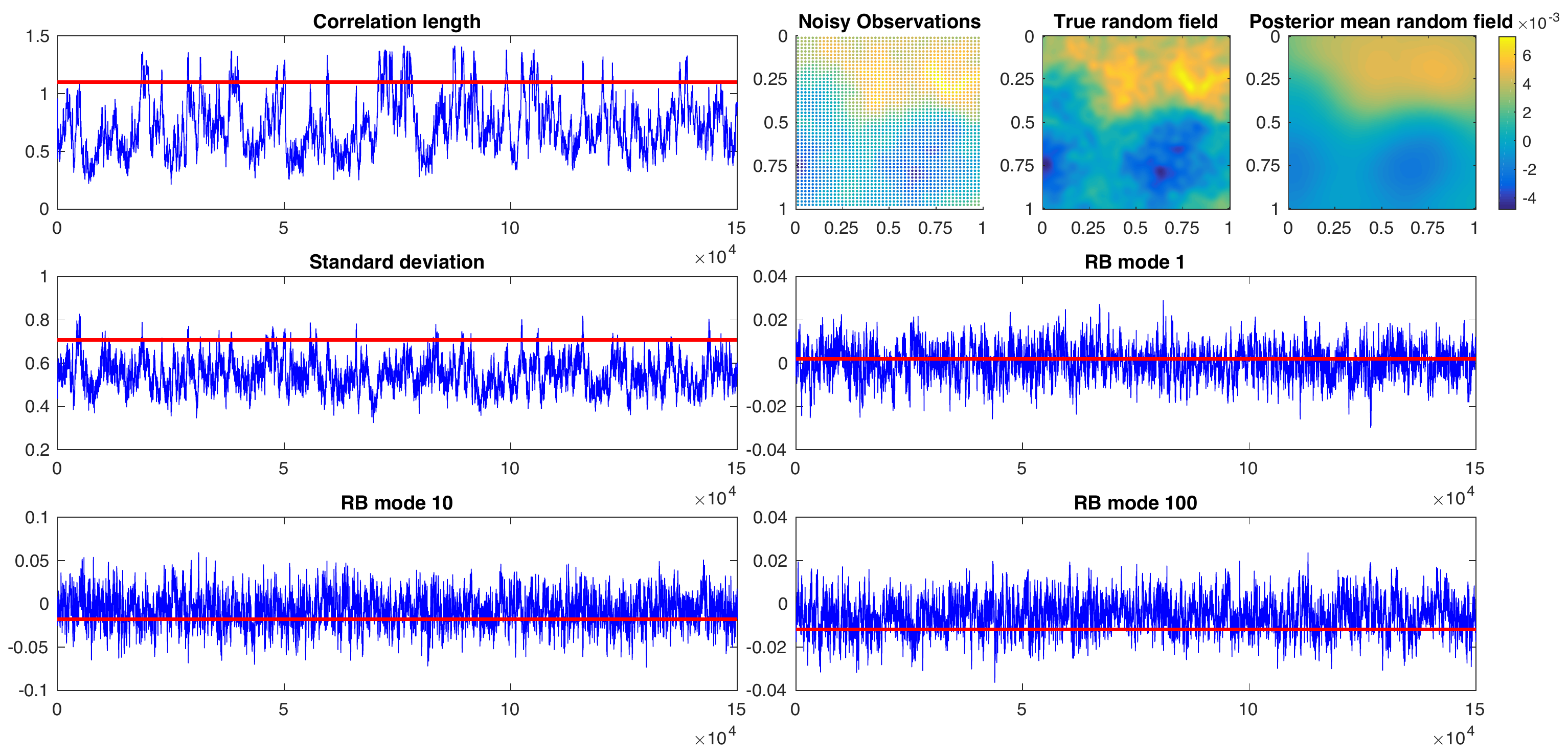}
\caption{Results of the MCMC estimation given random field observations (Example \ref{Example_Bayesian_Estimation}, $\ell=1.1$, $\sigma=1/\sqrt{2}$). 
	In the top-right corner we plot the positions and values of the noisy observations (left), the synthetic truth (middle), and the posterior mean (right). 
	The five path plots show the Markov chains for  $\ell$ and $\sigma$, and the reduced basis modes $(\theta_\RB)_1$, $(\theta_\RB)_{10}$, $(\theta_\RB)_{100}$, respectively. The red lines mark the truth.}
\label{Figure_MCMC_RF_1_1}
\end{figure}

\section{Conclusions}
We developed a mathematical and computational framework for working with parameterised Gaussian random fields arising from hierarchical forward and inverse problems in uncertainty quantification.
Under weak assumptions we proved the well-posed\-ness of the associated hierarchical problems.
We discretised the family of parameterised Gaussian random fields by (parametric) KL expansions.
We showed how the overall discretisation cost can be reduced substantially by a reduced basis surrogate for the parametric KL eigenpairs.
Moreover, we developed a reduced basis sampler for use with Monte Carlo and Markov chain Monte Carlo.
For Mat\'ern-type covariance operators with uncertain correlation length we suggested and analysed a linearisation technique to enable an efficient offline-online decomposition for the reduced basis solver.
We illustrated the accuracy and speed-up of the reduced basis surrogate and RB sampling in simple low-dimensional test problems.
Finally, we applied the reduced basis sampling to more realistic, high-dimensional, forward and Bayesian inverse test problems in 2D physical space.
The test results illustrate that the parametric KL eigenproblem can be approximated with acceptable accuracy by a reduced basis surrogate.
Moreover, the RB sampling gives acceptable accuracies compared to the full, unreduced sampling.
This enables an efficient hierarchical uncertainty quantification with parameterised random fields.
Of course, the size of the acceptable error level is heavily problem dependent, and our numerical experiments are a proof-of-concept.
A rigorous error analysis of the reduced basis samplers including the RB error and the linearisation error is beyond the scope of this study.

\begin{appendix}
\section{Parameterised Gaussian measures in forward and inverse uncertainty quantification} \label{Sec_Appendix_UQ}
In this appendix, we describe a general setting for forward uncertainty propagation and an associated Bayesian inverse problem.
Importantly, we investigate the well-posedness of these problems if the uncertain elements are modeled by \textit{parameterised} Gaussian random fields.
This is a necessary extension of the by now well-established solution theory for Gaussian random inputs. 
We conclude this section by revisiting the forward and inverse problem, respectively, and describe sampling based methods to approximate their solutions.
Both these methods require repeated sampling from a family of Gaussian random fields with different means and covariance operators each.

\subsection{Forward uncertainty propagation}\label{sec:fup}
We consider a classical model problem given by the elliptic partial differential equation (PDE)
\begin{align} \label{EQ_PDE}
- \nabla \cdot \left(\exp(\theta(\bsx)) \nabla p(\bsx)\right)&= f(\bsx) \qquad \forall x \in D,
\end{align}
together with suitable boundary conditions.
The PDE in \eqref{EQ_PDE} models the stationary, single-phase, incompressible flow of fluid in a permeable domain $D \subset \mathbb{R}^d$, combining Darcy's law and a mass balance equation. 
$p$ is the fluid pressure, $f$ is a source term, and $\theta$ is the spatially varying $\log$-permeability of the fluid reservoir. 
We assume that $D$ is connected, open and bounded, and that $\theta$ is an element of a separable Hilbert space $X$.
We define the PDE solution operator
\begin{equation*}
\mathcal{S}: X' \rightarrow H_0^1(D), \qquad \theta \mapsto p=\mathcal{S}(\theta),
\end{equation*}
where $X' \subseteq X$.
If $X = \mathcal{L}^2(D; \mathbb{R})$, a typical choice for $X'$ is the separable Banach space of continuous functions $C^0(D) := \{f \vert f\colon\overline{D} \rightarrow \mathbb{R} \text{ continuous}\}$.

We now consider $\theta$ to be uncertain.
We model the spatially varying uncertainty in $\theta$ by assuming that $\theta$ is a (parameterised) Gaussian random field with continuous realisations almost surely (a.s.).
Forward propagation of uncertainty means in our setting that we want to solve the elliptic PDE \eqref{EQ_PDE} with a random coefficient function $\theta$.
For a well-posed problem, the solution of \eqref{EQ_PDE} is a probability measure on $H_0^1(D)$.

We often consider a scalar-valued quantity of interest $Q: H_0^1(D) \rightarrow \mathbb{R}$  instead of the full space $H_0^1(D)$. 
For convenience we define $\mathcal{Q} = Q \circ \mathcal{S}$ which maps the uncertain parameter directly to the quantity of interest.
The forward propagation of uncertainty is modeled by the push-forward measure
\begin{equation*}
\mathcal{Q}^{\#}\mu'' := \mu''(\mathcal{Q} \in \cdot) := \mathbb{P}(\mathcal{Q}(\theta)\in \cdot),
\end{equation*}
where $\theta \sim \mu''$. 
If $\mu''$ is a Gaussian measure, the well-posedness and properties of $\mathcal{Q}^{\#}\mu''$ have been studied in e.g. in \cite{Charrier2012, Ernst2014}.
In the following, we extend this theory and discuss the existence of $\mathcal{Q}^{\#}\mu''$ with respect to the possibly non-Gaussian measure $\mu''$.
Moreover, we also study the existence of moments of $\mathcal{Q}^{\#}\mu''$.
To this end we make the following assumptions.
\begin{assumptions}\label{assumptions:forward}
\begin{itemize}
\item[(a)] $\mathcal{Q}^{\#}K(\cdot | \paramedet) :=  K(\mathcal{Q} \in \cdot|\paramedet) :=  \mathbb{P}(\mathcal{Q}(\theta)\in \cdot| \parame = \paramedet)$ is well-defined for $\mu'$-almost every $\paramedet \in R$.
\item[(b)] For some $k \in \mathbb{N}$ it holds
\begin{equation*}
m_k(\paramedet) := \int \mathcal{Q}(\theta)^k K(\mathrm{d}\theta|\paramedet)< \infty
\end{equation*}
for $\mu'$-almost every $\paramedet \in R$ and $\int m_k(\parame) \mu'(\mathrm{d}\parame) < \infty$.
\end{itemize}	
\end{assumptions}

\begin{theorem}
Let Assumption~\ref{assumptions:forward} hold.
Then, the measure $\mathcal{Q}^{\#}\mu''$ is well-defined.
Moreover, $\int \mathcal{Q}(\theta)^k \mu''(\mathrm{d}\theta) < \infty$ where $k\in \mathbb{N}$ is as in Assumption~\ref{assumptions:forward}(b).
\end{theorem}
\begin{proof}
By Assumption~\ref{assumptions:forward}(b), $\mathcal{Q}^{\#}K(\cdot | \paramedet)$ is well-defined and a probability measure for $\mu'$-almost every $\paramedet \in R$.
Hence, 
$$\mathcal{Q}^{\#}\mu'' = \int_{R} K(\cdot |\parame) \mu'(\mathrm{d}\parame)$$
is well-defined and a probability measure.
The finiteness of the moments can be shown analogously.
\end{proof}
Typically, $\mathcal{Q}^\#\mu''$ cannot be computed analytically.
We discuss the approximation of this measure by standard Monte Carlo in \S\ref{Subsec_Sampling_UQ}.

\subsection{Bayesian inverse problem} \label{Subsec_BIP}
Again, we consider the PDE \eqref{EQ_PDE} with a random coefficient $\theta \sim \mu_0$.
First, we discuss the standard (``Gaussian'') setting of a Bayesian inverse problem by assuming that $\mu_0$ is a Gaussian measure.

The starting point are noisy observations that we want to use to learn the uncertain parameter $\theta$.
We proceed by the Bayesian approach to inverse problems and compute the conditional probability measure of $\theta$ given the observations.
This measure is called \textit{posterior measure} - it reflects the knowledge about $\theta$ after observing the data.
In contrast, the knowledge about $\theta$ without observations is modelled by $\mu_0$, the \textit{prior measure}.

Let $\mathcal{O}: H_0^1 \rightarrow \mathbb{R}^{N_{\mathrm{obs}}}$ be a linear operator. 
It is called \textit{observation operator} and maps the PDE solution to the observations.
$\mathcal{G} := \mathcal{O} \circ \mathcal{S}$ is called \textit{forward response operator} and maps the uncertain parameter $\theta$ to the observations.
Let $y \in  \mathbb{R}^{N_{\mathrm{obs}}}$ denote the observations.
We assume that the observations $y$ are perturbed by additive Gaussian noise $\eta \sim \mathrm{N}(0, \Gamma)$, where $\Gamma \in \mathbb{R}^{N_{\mathrm{obs}} \times N_{\mathrm{obs}}}$ is a symmetric and positive definite matrix. 
$\eta$ and $\theta$ are statistically independent.
The event that occurs while we observe $y$ is
$
\{\mathcal{G}(\theta)+ \eta = y\} \in \mathcal{B}X.
$
The posterior measure is given by
$
\mathbb{P}(\theta \in \cdot | \mathcal{G}(\theta)+ \eta = y).
$
It can be derived using Bayes' formula which we introduce below.

If  Assumption 2.6 in \cite{Stuart2010} hold, and if the prior measure $\mu_0$ is Gaussian, then one can show that the posterior measure exists, that it is uniquely defined and that it is locally Lipschitz-continuous w.r.t. the data $y \in \mathbb{R}^{N_{\mathrm{obs}}}$. 
These statements refer to the space of probability measures on $(X, \mathcal{B}X)$ that are absolutely continuous w.r.t. $\mu_0$. 
This space, equipped with the Hellinger distance, is a metric space. 
We denote it by $(\mathbb{P}(X, \mathcal{B}X, \mu_0),\mathrm{d}_{\mathrm{Hel}})$,
where
\begin{equation*}
\mathrm{d}_{\mathrm{Hel}}(\nu_1, \nu_2) := \sqrt{\int \left(\sqrt{\frac{\mathrm{d}\nu_1}{\mathrm{d}\mu_0}(\theta)}-\sqrt{\frac{\mathrm{d}\nu_2}{\mathrm{d}\mu_0}(\theta)}\right)^2 \mu_0(\mathrm{d}\theta)}.
\end{equation*}
Hence, the problem of finding a posterior measure on $\mathbb{P}(X, \mathcal{B}X, \mu_0)$ is well-posed in the sense of Hadamard \cite{Hadamard1902}.
We refer to \cite{Stuart2010} for details.

The choice of the prior measure $\mu_0$ has a considerable impact on the posterior.
We discussed this in \S \ref{Sec:Introduction}, see also Figure \ref{Fig_Intro_Bayes}.
However, it is often not possible to determine the prior sufficiently.
Hence it is sensible to assume that the prior measure is (partially) unknown.
$\mu_0$ is then called \textit{hyperprior} and can be modelled by a Markov kernel (see \S\ref{subs:ParamGaussMe}). 
Here, we consider $\mu_0(\cdot | \parame) := K(\cdot | \parame)$ where $\parame$ is a random variable and $\mu'$ is the prior measure of $\parame$.
The posterior measure in this so-called \textit{hierarchical Bayesian inverse problem} is then 
\begin{equation} \label{EQ_Posterior_distribution}
\mu^y := \mathbb{P}((\parame, \theta) \in \cdot | \mathcal{G}(\theta)+ \eta = y).
\end{equation}
It can be determined using Bayes' formula
\begin{equation}
\mu^y(B) = Z_y^{-1} \iint_B \exp(-\Pot(\theta)) \mu_0(\mathrm{d}\theta|\parame)\mu'(\mathrm{d}\parame),
\end{equation}
where
\begin{align*}
B &\in \mathcal{R}\otimes \mathcal{B}X,\\
Z_y &:= \iint_{R \times X} \exp(-\Pot(\theta)) \mu(\mathrm{d}\theta|\parame)\mu'(\mathrm{d}\parame), \\
\Pot(\theta) &:= \frac{1}{2}\|\Gamma^{1/2}(\mathcal{G}(\theta)-y)\|^2_2.
\end{align*}
The function $\Pot$ is called \emph{potential}, $\exp(-\Pot)$ is called \emph{likelihood}, and  $Z_y$ is called \emph{normalising constant} or \emph{model evidence}.
We will now show that the posterior measure based on a hyperprior is well-defined and that the Bayesian inverse problem in this setting is well-posed.
\begin{theorem}
Let $\mu^{y,\paramedet}$ be the posterior measure of the Bayesian inverse problem
\begin{equation*}
\mathcal{G}(\theta)+\eta = y,
\end{equation*}
using $\mu_0(\cdot|\paramedet)$ as a prior measure, where $\paramedet \in R$ is fixed.
In particular, let
\begin{align*}
\mu^{y,\paramedet}(B) &:= Z_{y,\paramedet}^{-1} \int_B \exp(-\Pot(\theta)) \mu_0(\mathrm{d}\theta|\paramedet), \quad B \in \mathcal{B}X,\\
Z_{y,\paramedet} &:= \int_{X} \exp(-\Pot(\theta)) \mu_0(\mathrm{d}\theta|\paramedet).
\end{align*}
We assume that the computation of $\mu^{y,\paramedet}$ is well-posed for $\mu'$-almost every $\paramedet \in R$ 
and that a constant $L(r,\paramedet) \in (0, \infty)$ exists for $\mu'$-almost all $\paramedet \in R$, such that
\begin{equation}
\dHel(\mu^{y_1,\paramedet},\mu^{y_2,\paramedet}) \leq L(r,\paramedet) \|y_1-y_2\|_2,
\end{equation}
for any two datasets $y_1,y_2$, where $\max\{\|y_1\|_2,\|y_2\|_2\}<r$.
Moreover, we assume that $L(r,\cdot) \in \mathcal{L}^2(R,\mathcal{R};\mathbb{R})$ for all $r>0$, and that $Z_{y,(\cdot)} \in \mathcal{L}^1(R,\mathcal{R};\mathbb{R})$.
Then, the computation of $\mu^{y}$ is a well-posed problem.
\end{theorem}
\begin{proof}
By Bayes' Theorem, the posterior measure $\mu^y = \iint_{(\cdot)} \mu^{y,\parame}(\mathrm{d}\theta)\mu'(\mathrm{d}\parame)$ is well-defined and unique if the normalising constant $Z_{y}$ is positive and finite.
By assumption it holds $Z_{y,(\cdot)}\in (0, \infty)$. Hence,
\begin{equation*}
Z_{y} = \int \underbrace{Z_{y,\parame}}_{>0} \mu'(\mathrm{d}\parame) > 0.
\end{equation*}
Furthermore, also by assumption we have
\begin{equation*}
Z_{y} = \int Z_{y,\parame} \mu'(\mathrm{d}\parame) = \|Z_{y,(\cdot)}\|_1 < \infty.
\end{equation*}
Let $y_1, y_2 \in \mathbb{R}^{N_{\mathrm{obs}}}$ be a two datasets and let $\paramedet \in R$. 
By assumption, a constant $L(\parame) \in (0, \infty)$ exists, such that
\begin{equation}
\dHel(\mu^{y_1,\paramedet},\mu^{y_2,\paramedet}) \leq L(\paramedet) \|y_1-y_2\|_2.
\end{equation}
Now, we have
\begin{equation*}
\mathrm{d}_{\mathrm{Hel}}(\mu^{y_1}, \mu^{y_2})^2 = \iint \left(\sqrt{\frac{\mathrm{d}\mu^{y_1}}{\mathrm{d}\mu}(\theta,\parame)}-\sqrt{\frac{\mathrm{d}\mu^{y_2}}{\mathrm{d}\mu}(\theta,\parame)}\right)^2 \mu(\mathrm{d}\parame, \mathrm{d}\theta),
\end{equation*}
where 
\begin{equation*}
\frac{\mathrm{d}\mu^{y}}{\mathrm{d}\mu}(\theta,\parame) := Z_{y}^{-1} \exp(-\Pot(\theta)), \ \ \  y \in \{y_1, y_2\}. 
\end{equation*}
This density is constant in $\paramedet$. 
Let  $y \in \{y_1, y_2\}$.
The density $\frac{\mathrm{d}\mu^{y}}{\mathrm{d}\mu}(\theta,\paramedet)$ is identical to the density in the Bayesian inverse problem using $\mu_0(\cdot | \paramedet)$ as a prior with a fixed $\paramedet \in R$. 
The latter is given by $\frac{\mathrm{d}\mu^{y, \paramedet}}{\mathrm{d}\mu_0(\cdot|\paramedet)}(\theta)$.
Hence, 
\begin{align*}
\mathrm{d}_{\mathrm{Hel}}(\mu^{y_1}, \mu^{y_2})^2 &= \iint \left(\sqrt{\frac{\mathrm{d}\mu^{y_1, \parame}}{\mathrm{d}\mu_0(\cdot|\parame)}(\theta)}-\sqrt{\frac{\mathrm{d}\mu^{y_2, \parame}}{\mathrm{d}\mu_0(\cdot|\parame)}(\theta)}\right)^2 \mu(\mathrm{d}\parame, \mathrm{d}\theta) \\
&\leq \int \dHel(\mu^{y_1,\parame},\mu^{y_2,\parame})^2 \mu'(\mathrm{d}\parame) \\ &\leq \int L(\parame)^2 \|y_1-y_2\|_2^2\mu'(\mathrm{d}\parame) \\
&= \int L(\parame)^2 \mu'(\mathrm{d}\parame)   \|y_1-y_2\|_2^2 < \infty.
\end{align*}
The right-hand side is finite, since $L$ is square-integrable.
This implies that the computation of $\mu^y$ is a well-posed problem.
\end{proof}

\subsection{Monte Carlo for forward uncertainty propagation}\label{Subsec_Sampling_UQ}
We return to the forward uncertainty propagation problem (see \S\ref{sec:fup}).
Classically, the pushforward measure $\mathcal{Q}^{\#}\mu''$ can be approximated by a Monte Carlo method. 
In particular, we draw $N_{\mathrm{smp}} \in \mathbb{N}$ independent samples from the parameterised Gaussian measure $\mu$,
\begin{equation*}
(\parame_1,\theta_1),\dots,(\parame_{N_{\mathrm{smp}}},\theta_{N_{\mathrm{smp}}}) \sim \mu.
\end{equation*}
Then, we evaluate the quantity of interest $\mathcal{Q}(\theta_1),\dots,\mathcal{Q}(\theta_{N_{\mathrm{smp}}})$. 
Finally, we construct the discrete measure 
\begin{equation*}
\widehat{\mathcal{Q}^{\#}\mu''} = \frac{1}{N_{\mathrm{smp}}}\sum_{n=1}^{N_{\mathrm{smp}}}\delta_{\mathcal{Q}(\theta_n)}.
\end{equation*}
The Glivenko-Cantelli Theorem (see \cite[\S 5]{Vapnik1971}) implies that 
$\widehat{\mathcal{Q}^{\#}\mu''} \rightarrow \mathcal{Q}^{\#}\mu''$ weakly, as $N_{\mathrm{smp}} \rightarrow \infty$.

\subsection{Markov Chain Monte Carlo for Bayesian inverse problems}
In contrast to Monte Carlo sampling from the pushforward measure, it is in general not possible to sample independently from the posterior measure.
For this reason we use Markov Chain Monte Carlo (MCMC).
Meaning that we construct an ergodic Markov chain where the stationary measure is the posterior measure.
The samples $$(\parame_1,\theta_1),\dots,(\parame_{N_{\mathrm{smp}}},\theta_{N_{\mathrm{smp}}}) \sim \mu^y$$ can then be used to approximate integrals with respect to the posterior measure. 

In our setting the prior measure is a parameterised Gaussian measure on a (discretised) function space.
Following \cite{Dunlop2017} we suggest to use a Metropolis-within-Gibbs sampler.
This allows us to sample $\parame$  and $\theta$ in an alternating way using two different proposal kernels, one for $\theta$ and one for $\parame$.

In many cases, the hyperparameter space $R$ will be low dimensional. 
Hence, various Metropolis-Hastings proposals can be used to efficiently propose samples of $\tau$.
We denote the conditional density of this Metropolis-Hastings proposal by  $q_{R}: R \times R \rightarrow \mathbb{R}$.
On the other hand, $\theta$ is a $X$-valued random variable where $X$ is high-dimensional or possibly infinite-dimensional.
Therefore, we suggest a preconditioned Crank-Nicholson (pCN) proposal, see \cite{Cotter2013} for details.
Combining the Metropolis-Hastings and the pCN proposal we arrive at the MCMC method in Algorithm \ref{Alg_PCNMCMCGibbs}.
Note that $\beta \in (0, 1]$ is a tuning parameter in the pCN-MCMC method.

Note that in Algorithm \ref{Alg_PCNMCMCGibbs} we require $\dim X < \infty$.
This can be achieved by a discretisation of the space $X$ (see \S\ref{Subsec_Sampl_stra}).
In this case, the Gaussian measure $\mathrm{N}(m(\paramedet),\mathcal{C}(\paramedet))$ has a probability density function (see Proposition \ref{Propo_Gaussian_PDF}), and we use this probability density function to compute the acceptance probability $\alpha_R$ for the Gibbs move of $\parame$.
It is however possible to compute $\alpha_R$ without access to the probability density function.
In this case one can also define an algorithm for an infinite dimensional space $X$. 
We refer to \cite{Dunlop2017} for details on the infinite-dimensional version.

\begin{algorithm}[htb]
Let $(\parame_0, \theta_0)$ be the initial sample of the Markov chain.

 \For{$n \in \{1,\dots,N_{\mathrm{smp}}\}$}{
 Sample $\parame^* \sim q_R(\cdot |\parame_{n-1})$ 

$\alpha_R(\parame_{n-1};\parame^*) \leftarrow \min\left\lbrace1, \frac{q_R(\parame_{n-1}|\parame^*)}{q_R(\parame^* |\parame_{n-1})}
\frac{\mathrm{n}(\theta_{n-1};m(\tau^*),\mathcal{C}(\tau^*))}{\mathrm{n}(\theta_{n-1};m(\parame_{n-1}),\mathcal{C}(\parame_{n-1}))}\right\rbrace$ 
 
 Sample $U_R \sim \mathrm{Unif}[0,1]$
 
 \eIf{$U_R < \alpha_R$}{
 $\parame_{n} \leftarrow \parame^*$
 }{$\parame_{n} \leftarrow \parame_{n-1}$}
 
 Sample $\theta^* \sim \mathrm{N}(\sqrt{1-\beta^2}\theta_{n-1}, \beta \mathcal{C}(\parame_n))$
 
 $\alpha_X(\theta_{n-1};\theta^*) \leftarrow \min\{1,\exp(-\Pot(\theta^*)+\Pot(\theta_{n-1}))\}$ 
 
 Sample $U_X \sim \mathrm{Unif}[0,1]$
 
 \eIf{$U_X < \alpha_X$}{
 $\theta_{n} \leftarrow \theta^*$
 }{$\theta_{n} \leftarrow \theta_{n-1}$}
 }
 \caption{Metropolis-within-Gibbs to sample from the posterior measure of $(\tau, \theta)$}
 \label{Alg_PCNMCMCGibbs}
\end{algorithm}

It can be proved that Algorithm~\ref{Alg_PCNMCMCGibbs} generates samples from $\mu^y$.
We summarise this in the following Proposition and refer to \cite[\S 3]{Dunlop2017} for a similar statement and proof.
\begin{proposition}
Algorithm \ref{Alg_PCNMCMCGibbs} defines a Markov kernel $\mathrm{MK}$ from $(X \times R,\mathcal{B}X \otimes \mathcal{R})$ to $(X \times R,\mathcal{B}X \otimes \mathcal{R})$ that has $\mu^y$ as a stationary measure. In particular,
\begin{equation*}
\mu^y \mathrm{MK} = \mu^y :\Leftrightarrow
\int \mathrm{MK}(B|\parame, \theta)\mu^y(\mathrm{d}\parame, \mathrm{d}\theta) = \mu^y(B) \ (B \in \mathcal{B}X \otimes \mathcal{R}).
\end{equation*}
\end{proposition}

\section{Proofs of auxiliary results} \label{Section_Proofs_of_Lemmata}

\begin{lemma} 
Let $\nu \in (0,\infty) \backslash \mathbb{N}$, let $N_\lin \in 2\mathbb{N}$ and let Assumption~\ref{assumptions:ell} hold. Then,
\begin{equation}
\| \widetilde{\mathcal{C}}{(\nu, \ell, \sigma,N_\lin)}- {\mathcal{C}}{(\nu, \ell, \sigma)} \|_X \leq \mathrm{diam}(D)^{2d} \frac{\pi|\mathrm{csc}(\pi\nu)|}{2^{1-\nu}}(1+\zeta_{\max}^{2\nu})  \exp\left(\frac{\zeta_{\max}^2}{4}\right)    \frac{\zeta_{\max}^{2N_\lin}}{(N_\lin)!},
\end{equation}
where $\zeta_{\max} := \frac{\mathrm{diam}(D)}{\underline{\ell}}$.
\end{lemma}
\begin{proof}
Let $\nu \in (0,\infty) \backslash \mathbb{N}$. 
Consider the function
\begin{equation*}
f_{(\nu)}: [0, \infty) \rightarrow [0, \infty), \quad \zeta \mapsto \zeta^\nu \cdot K_\nu(\zeta).
\end{equation*}
It holds $f_{(\nu)}(\sqrt{2\nu}z/\ell) = \mathrm{const}(\nu, \sigma)c{(\nu, \ell, \sigma)}(z)$, where $\mathrm{const}(\nu, \sigma)>0$ is a constant that does not depend on the correlation length $\ell$.
Moreover, we assume that we work in a bounded computational domain $D$, and that $\ell$ is bounded from below by a fixed positive constant $\underline{\ell} > 0$ (see Assumption~\ref{assumptions:ell}). 
Now, for $\nu \in (0,\infty)\backslash \mathbb{N}$ the function $f_{(\nu)}$ can be written in terms of a series
\begin{equation*}f_{(\nu)}(\zeta) = \frac{\pi \mathrm{csc}(\pi\nu)}{2}  \sum_{k=1}^\infty\left( \frac{1}{2^{2k-2-\nu} \Gamma(k-\nu)(k-1)!} -  \frac{\zeta^{2\nu} }{2^{2k-2+\nu}\Gamma(k+\nu)(k-1)!} \right)\zeta^{2k-2}.
\end{equation*}
This follows from the representations of $K_\nu$ and $I_\nu$ in  \cite[Equations 9.6.2, 9.6.10]{abramowitzstegun1964}.
If we truncate the series after the first $N_\lin/2$ terms we obtain the following function:
\begin{equation*}f_{(\nu,{N_\lin})}(\zeta) = \frac{\pi \mathrm{csc}(\pi\nu)}{2}  \sum_{k=1}^{N_\lin/2}\left( \frac{1}{2^{2k-2-\nu} \Gamma(k-\nu)(k-1)!} -  \frac{\zeta^{2\nu} }{2^{2k-2+\nu}\Gamma(k+\nu)(k-1)!} \right)\zeta^{2k-2}.
\end{equation*}
This (truncated) series expansion is associated with the integral operator $\widetilde{\mathcal{C}}{(\nu, \ell, \sigma, N_\lin)}$ that is given by the kernel  
\begin{equation}
c{(\nu, \ell, \sigma, N_\lin)}(z) := \frac{f_{(\nu,N_\lin)}(\sqrt{2\nu}z/\ell)}{\mathrm{const(\nu,\sigma)}}.
\end{equation}
Now, we bound the asymptotic truncation error. 
Assume w.l.o.g. that $N_\lin > \nu$.
Note that in this case $\Gamma(k-\nu)>1$, and, moreover, $\zeta > 0$. 
Using the triangle inequality we arrive at
\begin{align*}
|f_{(\nu,{N_\lin})}(\zeta)-f_{(\nu)}(\zeta)| &\leq \frac{\pi|\mathrm{csc}(\pi\nu)|}{2}   \Biggl| \sum_{k=N_\lin+1}^{\infty}\left( \frac{1-\zeta^{2\nu}}{2^{2k-2-\nu} \Gamma(k-\nu)(k-1)!}  \right)\zeta^{2k-2}\Biggl| \\  &\leq \frac{\pi|\mathrm{csc}(\pi\nu)|}{2^{1-\nu}}(1+\zeta^{2\nu})   \sum_{k=N_\lin+1}^{\infty}\left( \frac{1}{\Gamma(k-\nu)(k-1)!}  \right)\left(\frac{\zeta}{2}\right)^{2k-2}  \\  &\leq \frac{\pi|\mathrm{csc}(\pi\nu)|}{2^{1-\nu}}(1+\zeta^{2\nu})   \sum_{k=N_\lin+1}^{\infty}\left( \frac{1}{k-1!}  \right)\left(\frac{\zeta}{2}\right)^{2k-2}.
\end{align*}
The infinite sum on the right-hand side above can be bounded by the remainder term of a Taylor series expansion of the exponential function with $N_\lin$ terms and anchor point $\zeta=0$.
This gives
\begin{align*}
|f_{(\nu,{N_\lin})}(\zeta)-f_{(\nu)}(\zeta)| &\leq  \frac{\pi|\mathrm{csc}(\pi\nu)|}{2^{1-\nu}}(1+\zeta^{2\nu}) \exp\left(\frac{\zeta_{\max}^2}{4}\right)   \frac{\zeta^{2N_\lin}}{(N_\lin)!} \\&\leq  \frac{\pi|\mathrm{csc}(\pi\nu)|}{2^{1-\nu}}(1+\zeta_{\max}^{2\nu})  \exp\left(\frac{\zeta_{\max}^2}{4}\right)    \frac{\zeta_{\max}^{2N_\lin}}{(N_\lin)!} =: \mathrm{const}'(N_\lin),
\end{align*}
where $\zeta_{\max} = \frac{\mathrm{diam}(D)}{\underline{\ell}}$.
Finally, let $\varphi \in X$.
By the Cauchy--Schwarz inequality it holds
\begin{align*}
\| &\widetilde{\mathcal{C}}{(\nu, \ell, \sigma,N_\lin)} \varphi - {\mathcal{C}}{(\nu, \ell, \sigma)}\varphi \|_X^2 \\
&=   \int_D \left(\int_D (\widetilde{c}^{(\nu, \ell, \sigma,N_\lin)}(\mathrm{dist}(\bsx,\bsx')) - {c}{(\nu, \ell, \sigma)}(\mathrm{dist}(\bsx,\bsx')))\varphi(\bsx) \mathrm{d}\bsx \right)^2\mathrm{d}\bsx' \\
&\leq   \int_D \left(\int_D (\widetilde{c}^{(\nu, \ell, \sigma,N_\lin)}(\mathrm{dist}(\bsx,\bsx')) - {c}{(\nu, \ell, \sigma)}(\mathrm{dist}(\bsx,\bsx')))^2 \mathrm{d}\bsx \right) \cdot \left( \int_D \varphi(\bsx)^2\mathrm{d}\bsx \right)\mathrm{d}\bsx' \\
&\leq \mathrm{Leb}(d)(D)^2 \cdot \mathrm{const}'(N_\lin)^2 \cdot \|\varphi \|^2_X
\\&\leq \mathrm{diam}(D)^{2d}\cdot \mathrm{const}'(N_\lin)^2 \cdot \|\varphi \|^2_X.
\end{align*}
Taking the square root on both sides and dividing by $\| \varphi \|_X $ gives the desired error bound.
\end{proof}
\begin{lemma}
The operator $\widetilde{\mathcal{C}}{(\nu, \ell, \sigma, N_\lin)}$ is self-adjoint, trace-class and continuous.
\end{lemma}
\begin{proof}
The integral operator $\widetilde{\mathcal{C}}{(\nu, \ell, \sigma, N_\lin)}$ is self-adjoint since the associated kernel function is symmetric.
The operator is trace-class since $D$ is a bounded domain, and $$\int_D \widetilde{c}(\nu, \ell, \sigma, N_\lin)(\mathrm{dist}(\bsx,x)) \mathrm{d}\bsx =\widetilde{c}(\nu, \ell, \sigma, N_\lin)(0) \cdot \mathrm{Leb}(d)(D) < \infty.$$
The boundedness of $D$ also implies the continuity of the operator.
\end{proof}

\begin{lemma}
	The Mat\'ern-type covariance operator $\mathcal{C}(\nu, \ell, \sigma)$ and the approximate operator $\widetilde{\mathcal{C}}_{0}(\nu, \ell, \sigma,N_\lin)$ in \eqref{C0tilde} satisfy 
	\[\| \widetilde{\mathcal{C}}_{0}{(\nu, \ell, \sigma,N_\lin)}- {\mathcal{C}}{(\nu, \ell, \sigma)} \|_X \leq 2 \| \widetilde{\mathcal{C}}{(\nu, \ell, \sigma,N_\lin)}- {\mathcal{C}}{(\nu, \ell, \sigma)} \|_X.\]
\end{lemma}

\begin{proof}
Let $(\widetilde{\lambda_i})_{i=1}^\infty$ denote the eigenvalues of the operator $\widetilde{\mathcal{C}}{(\nu, \ell, \sigma,N_\lin)}$.
Without loss of generality, we assume that the spectrum of $\widetilde{\mathcal{C}}{(\nu, \ell, \sigma,N_\lin)}$ contains a negative eigenvalue.
Since $\widetilde{\mathcal{C}}{(\nu, \ell, \sigma,N_\lin)}$ is trace-class, it holds $|\widetilde{\lambda}_i| \rightarrow 0$ for $i \rightarrow \infty$.
Hence, there is an eigenpair $(\widetilde{\lambda}_{\max}, \widetilde{\psi}_{\max})$ which realises the maximum in the expression
\begin{equation} \label{EQ_maximal_Eigenpair}
\max_{i \in \mathbb{N}\colon \widetilde{\lambda}_i < 0}|\widetilde{\lambda}_i|.
\end{equation}
Thus,
\begin{align*}
\| \widetilde{\mathcal{C}}_{0}{(\nu, \ell, \sigma,N_\lin)}- \widetilde{\mathcal{C}}{(\nu, \ell, \sigma,N_\lin)}\|_X = \| \sum_{i=1; \lambda_i < 0}^\infty \widetilde{\lambda}_i \widetilde{\psi}_i \otimes \widetilde{\psi}_i \|_X = | \widetilde{\lambda}_{\max}|. 
\end{align*}
Moreover, since ${\mathcal{C}}{(\nu, \ell, \sigma)}$ is positive definite, we have $$\widetilde{\psi}_{\max}^*\widetilde{\mathcal{C}}{(\nu, \ell, \sigma,N_\lin)}\widetilde{\psi}_{\max} \geq 0 > \widetilde{\psi}_{\max}^*{\mathcal{C}}{(\nu, \ell, \sigma)}\widetilde{\psi}_{\max}.$$
Hence, we obtain
\begin{align*}
\| \widetilde{\mathcal{C}}{(\nu, \ell, \sigma,N_\lin)}- {\mathcal{C}}{(\nu, \ell, \sigma)} \|_X &\geq |\widetilde{\psi}_{\max}^*(\widetilde{\mathcal{C}}{(\nu, \ell, \sigma,N_\lin)}- {\mathcal{C}}{(\nu, \ell, \sigma)})\widetilde{\psi}_{\max}| \\ &=  \widetilde{\psi}_{\max}^*{\mathcal{C}}{(\nu, \ell, \sigma)}\widetilde{\psi}_{\max}
- \widetilde{\psi}_{\max}^*\widetilde{\mathcal{C}}{(\nu, \ell, \sigma,N_\lin)}\widetilde\psi_{\max} \\ &=  \widetilde\psi_{\max}^*{\mathcal{C}}{(\nu, \ell, \sigma)}\widetilde{\psi}_{\max}
- \widetilde{\lambda}_{\max} \widetilde{\psi}_{\max}^*\widetilde{\psi}_{\max} \\
&\geq |\widetilde{\lambda}_{\max}|.
\end{align*}
This gives the bound
\begin{align*}
\| \widetilde{\mathcal{C}}_{0}{(\nu, \ell, \sigma,N_\lin)}- \widetilde{\mathcal{C}}{(\nu, \ell, \sigma,N_\lin)}\|_X \leq \| \widetilde{\mathcal{C}}{(\nu, \ell, \sigma,N_\lin)}- {\mathcal{C}}{(\nu, \ell, \sigma)} \|_X .
\end{align*}
Finally, using the triangle inequality, we arrive at
\begin{align*}
\| &\widetilde{\mathcal{C}}_0{(\nu, \ell, \sigma,N_\lin)}- {\mathcal{C}}{(\nu, \ell, \sigma)} \|_X \\ &\leq \| \widetilde{\mathcal{C}}_{0}{(\nu, \ell, \sigma,N_\lin)}- \widetilde{\mathcal{C}}{(\nu, \ell, \sigma,N_\lin)}\|_X +  \| \widetilde{\mathcal{C}}{(\nu, \ell, \sigma,N_\lin)}-\mathcal{C}{(\nu, \ell, \sigma)}\|_X \\
&\leq 2\| \widetilde{\mathcal{C}}{(\nu, \ell, \sigma,N_\lin)}-\mathcal{C}{(\nu, \ell, \sigma)}\|_X.
\end{align*}
\end{proof}
\end{appendix}

\section*{Acknowledgements}
The authors would like to thank Barbara Wohlmuth for pointing them to the work \cite{Horger2016} and for many helpful comments. 
JL and EU would like to thank the Isaac Newton Institute for Mathematical Sciences for support and hospitality during the programme \textit{Uncertainty quantification for complex systems: theory and methodologies} when work on this paper was undertaken.
This work was supported by the DFG through the International Graduate School of Science and Engineering at the Technical University of Munich within the project 10.02 BAYES, and by EPSRC Grant Number EP/K032208/1.

\bibliographystyle{plain} 
\bibliography{library_new}

\begin{thebibliography}{10}

\bibitem{abramowitzstegun1964}
Milton Abramowitz and Irene~A. Stegun.
\newblock {\em Handbook of Mathematical Functions with Formulas, Graphs, and
  Mathematical Tables}.
\newblock Dover, New York, ninth dover printing, tenth gpo printing edition,
  1964.

\bibitem{Adler1990}
Robert Adler.
\newblock {\em {An Introduction to Continuity, Extrema, and Related Topics for
  General Gaussian Processes}}.
\newblock IMS, 1990.

\bibitem{Benner2018}
P.~Benner, Y.~Qiu, and M.~Stoll.
\newblock {Low-Rank Eigenvector Compression of Posterior Covariance Matrices
  for Linear Gaussian Inverse Problems}.
\newblock {\em SIAM/ASA Journal on Uncertainty Quantification}, 6(2):965--989,
  2018.

\bibitem{Beskos2015}
Alexandros Beskos, Ajay Jasra, Ege~A. Muzaffer, and Andrew~M. Stuart.
\newblock {Sequential Monte Carlo methods for Bayesian elliptic inverse
  problems}.
\newblock {\em Stat. Comput.}, 25(4):727--737, 2015.

\bibitem{Betzetal:2014}
Wolfgang Betz, Iason Papaioannou, and Daniel Straub.
\newblock {Numerical methods for the discretization of random fields by means
  of the Karhunen-Lo\`eve expansion.}
\newblock {\em Comput. Methods Appl. Mech. Eng.}, 271:109--129, 2014.

\bibitem{BolinKirchner:2017}
David Bolin and Kristin Kirchner.
\newblock {The rational SPDE approach for Gaussian random fields with general
  smoothness}.
\newblock {\em Preprint}, arXiv:1711.04333v2, 2017.

\bibitem{BKK}
David Bolin, Kristin Kirchner, and Mih{\'a}ly Kov{\'a}cs.
\newblock {Weak convergence of Galerkin approximations for fractional elliptic
  stochastic PDEs with spatial white noise}.
\newblock {\em BIT Numerical Mathematics}, 2018.
\newblock Published online 06 August 2018.

\bibitem{Calvetti1994}
Daniela Calvetti, Lothar Reichel, and Danny Sorensen.
\newblock {An Implicitly Restarted Lanczos Method for Large Symmetric
  Eigenvalue Problems}.
\newblock {\em Electron. Trans. Numer. Anal}, 2:1--21, 1994.

\bibitem{ChanWood:1997}
Grace Chan and Andrew T.~A. Wood.
\newblock {An algorithm for simulating stationary Gaussian random fields}.
\newblock {\em {J. R. Stat. Soc., Ser. C}}, 46(1):171--181, 1997.

\bibitem{Charrier2012}
Julia Charrier.
\newblock {Strong and Weak Error Estimates for Elliptic Partial Differential
  Equations with Random Coefficients}.
\newblock {\em SIAM J. Numer. Anal.}, 50(1):216--246, 2012.

\bibitem{ChenStein:2017}
Jie Chen and Michael~L. Stein.
\newblock {Linear-cost covariance functions for Gaussian random fields}.
\newblock {\em Preprint}, arXiv:1711.05895, 2017.

\bibitem{Chen2015}
Peng Chen and Christoph Schwab.
\newblock {Sparse-grid, reduced-basis Bayesian inversion}.
\newblock {\em Computer Methods in Applied Mechanics and Engineering}, 297:84
  -- 115, 2015.

\bibitem{Chernov2016}
Alexey Chernov, Hakon Hoel, Kody J.~H. Law, Fabio Nobile, and Raul Tempone.
\newblock {Multilevel Ensemble Kalman Filtering for spatially extended models}.
\newblock {\em Preprint}, arXiv:1608.08558, 2016.

\bibitem{Rizzi2018}
A.~Contreras, P.~Mycek, O.~Le~Maître, F.~Rizzi, B.~Debusschere, and O.~Knio.
\newblock {Parallel Domain Decomposition Strategies for Stochastic Elliptic
  Equations. Part A: Local Karhunen--Loève Representations}.
\newblock {\em SIAM Journal on Scientific Computing}, 40(4):C520--C546, 2018.

\bibitem{Cotter2013}
S.~L. Cotter, G.~O. Roberts, A.~M. Stuart, and D.~White.
\newblock {MCMC Methods for Functions: Modifying Old Algorithms to Make Them
  Faster}.
\newblock {\em Statist. Sci.}, 28(3):424--446, 2013.

\bibitem{Dashti2011}
Masoumeh Dashti and Andrew~M. Stuart.
\newblock {Uncertainty Quantification and Weak Approximation of an Elliptic
  Inverse Problem}.
\newblock {\em SIAM J. Numer. Anal.}, 49(6):2524--2542, 2011.

\bibitem{Dashti2017}
Masoumeh Dashti and Andrew~M. Stuart.
\newblock {The Bayesian Approach to Inverse Problems}.
\newblock In Roger Ghanem, David Higdon, and Houman Owhadi, editors, {\em
  Handbook of Uncertainty Quantification}, pages 311--428. Springer, 2017.

\bibitem{DElia2013}
Marta D'Elia and Max Gunzburger.
\newblock {Coarse-Grid Sampling Interpolatory Methods for Approximating
  Gaussian Random Fields}.
\newblock {\em SIAM/ASA J. Uncertain. Quantif.}, 1(1):270--296, 2013.

\bibitem{Dietrich1997}
C.~R. Dietrich and G.~N. Newsam.
\newblock {Fast and Exact Simulation of stationary Gaussian Processes through
  Circulant Embedding of the Covariance Matrix}.
\newblock {\em SIAM J. Sci. Comput.}, 18(4):1088--1107, 1997.

\bibitem{Douglas1974}
J.~{Douglas Jr.}, T.~Dupont, and M.~F. Wheeler.
\newblock {A Galerkin procedure for approximating the flux on the boundary for
  elliptic and parabolic boundary value problems}.
\newblock {\em Rev. Fran{\c{c}}aise Automat. Informat. Recherche
  Op{\'{e}}rationnelle S{\'{e}}r. Rouge}, 8(R2):47--59, 1974.

\bibitem{Drohmann2015}
M.~Drohmann and K.~Carlberg.
\newblock {The ROMES Method for Statistical Modeling of Reduced-Order-Model
  Error}.
\newblock {\em SIAM/ASA Journal on Uncertainty Quantification}, 3(1):116--145,
  2015.

\bibitem{Dunlop2017}
Matthew~M. Dunlop, Marco~A. Iglesias, and Andrew~M. Stuart.
\newblock {Hierarchical Bayesian level set inversion}.
\newblock {\em Statistics and Computing}, 27(6):1555--1584, 2017.

\bibitem{Ernst2014}
Oliver~G. Ernst and Bj\"{o}rn Sprungk.
\newblock {Stochastic Collocation for Elliptic PDEs with Random Data: The
  Lognormal Case}.
\newblock In Jochen Garcke and Dirk Pfl{\"{u}}ger, editors, {\em Sparse Grids
  and Applications - Munich 2012}, pages 29--53, Cham, 2014. Springer
  International Publishing.

\bibitem{Everitt1981}
B.~S. Everitt and D.~J. Hand.
\newblock {\em {Finite Mixture Distributions}}.
\newblock Springer Netherlands, Dordrecht, 1981.

\bibitem{Farcas2018}
Ionu\c{t}-Gabriel Farca\c{s}, Benjamin Uekermann, Tobias Neckel, and
  Hans-Joachim Bungartz.
\newblock {Nonintrusive Uncertainty Analysis of Fluid-structure Interaction
  with Spatially Adaptive Sparse Grids and Polynomial Chaos Expansion}.
\newblock {\em SIAM Journal on Scientific Computing}, 40(2):B457--B482, 2018.

\bibitem{Feischl2018}
Michael Feischl, Frances~Y. Kuo, and Ian~H. Sloan.
\newblock {Fast random field generation with H-matrices}.
\newblock {\em Numerische Mathematik}, 140(3):639--676, 2018.

\bibitem{Fumagalli2016}
{Fumagalli, Ivan}, {Manzoni, Andrea}, {Parolini, Nicola}, and {Verani, Marco}.
\newblock Reduced basis approximation and a posteriori error estimates for
  parametrized elliptic eigenvalue problems.
\newblock {\em ESAIM: M2AN}, 50(6):1857--1885, 2016.

\bibitem{Ghanem2017}
R~Ghanem, D~Higdon, and H~Owhadi, editors.
\newblock {\em {Handbook of Uncertainty Quantification}}.
\newblock Springer, 2017.

\bibitem{Graham2018}
Ivan~G. Graham, Frances~Y. Kuo, Dirk Nuyens, Robert Scheichl, and Ian~H. Sloan.
\newblock Analysis of circulant embedding methods for sampling stationary
  random fields.
\newblock {\em SIAM Journal on Numerical Analysis}, 56(3):1871--1895, 2018.

\bibitem{GuRuhe2000}
M.~Gu, A.~Ruhe, R.~Lehoucq, D.~Sorensen, R.~Freund, G.~Sleijpen, H.~van~der
  Vorst, Z.~Bai, and R~Li.
\newblock {\em {Hermitian Eigenvalue Problems}}, chapter~4, pages 45--107.
\newblock Society for Industrial and Applied Mathematics, 2000.

\bibitem{Hadamard1902}
J.~Hadamard.
\newblock {Sur les probl{\`{e}}mes aux d{\'{e}}riv{\'{e}}s partielles et leur
  signification physique}.
\newblock {\em Princeton University Bulletin}, 13:49--52, 1902.

\bibitem{Harbrechtetal:2015}
Helmut Harbrecht, Michael Peters, and Markus Siebenmorgen.
\newblock {Efficient approximation of random fields for numerical
  applications.}
\newblock {\em {Numer. Linear Algebra Appl.}}, 22(4):596--617, 2015.

\bibitem{Haskard2007}
K~A Haskard.
\newblock {\em {An anisotropic Mat\'ern spatial covariance model: REML
  estimation and properties}}.
\newblock PhD thesis, University of Adelaide, 2007.

\bibitem{Book2}
J~S Hesthaven, G~Rozza, and B~Stamm.
\newblock {\em Certified reduced basis methods for parametrized partial
  differential equations}.
\newblock SpringerBriefs in Mathematics. Springer, Cham; BCAM Basque Center for
  Applied Mathematics, Bilbao, 2016.

\bibitem{Horger2016}
Thomas Horger, Barbara Wohlmuth, and Thomas Dickopf.
\newblock {Simultaneous Reduced Basis Approximation of Parameterized Elliptic
  Eigenvalue Problems}.
\newblock {\em ESAIM: M2AN}, 51(2):443--465, 2017.

\bibitem{Horger2017}
Thomas Horger, Barbara Wohlmuth, and Linus Wunderlich.
\newblock {Reduced Basis Isogeometric Mortar Approximations for Eigenvalue
  Problems in Vibroacoustics}.
\newblock In Peter Benner, Mario Ohlberger, Anthony Patera, Gianluigi Rozza,
  and Karsten Urban, editors, {\em Model Reduction of Parametrized Systems},
  pages 91--106. Springer International Publishing, Cham, 2017.

\bibitem{JiangOu:2017}
L.~Jiang and N.~Ou.
\newblock Multiscale model reduction method for {B}ayesian inverse problems of
  subsurface flow.
\newblock {\em J. Comput. Appl. Math.}, 319:188--209, 2017.

\bibitem{Karhunen1947}
K~Karhunen.
\newblock {{\"{U}}ber lineare Methoden in der Wahrscheinlichkeitsrechnung}.
\newblock {\em Ann. Acad. Sci. Fennicae. Ser. A. I. Math.-Phys.}, 37:1--79,
  1947.

\bibitem{Khoromskij2009}
Boris~N. Khoromskij, Alexander Litvinenko, and Hermann~G. Matthies.
\newblock {Application of hierarchical matrices for computing the
  Karhunen-Lo{\`{e}}ve expansion}.
\newblock {\em Computing}, 84(1):49--67, 2009.

\bibitem{Lieberman2010}
C.~Lieberman, K.~Willcox, and O.~Ghattas.
\newblock Parameter and state model reduction for large-scale statistical
  inverse problems.
\newblock {\em SIAM Journal on Scientific Computing}, 32(5):2523--2542, 2010.

\bibitem{Lifshits1995}
M.~A. Lifshits.
\newblock {\em {Gaussian Random Functions}}.
\newblock Springer, 1995.

\bibitem{Loeve1978}
M.~Loeve.
\newblock {\em {Probability Theory II}}.
\newblock Springer, New York, 1978.

\bibitem{Machiels2000}
Luc Machiels, Yvon Maday, Ivan~B. Oliveira, Anthony~T. Patera, and Dimitrios~V.
  Rovas.
\newblock Output bounds for reduced-basis approximations of symmetric positive
  definite eigenvalue problems.
\newblock {\em Comptes Rendus de l'Académie des Sciences - Series I -
  Mathematics}, 331(2):153 -- 158, 2000.

\bibitem{Manzoni2016}
A.~Manzoni, S.~Pagani, and T.~Lassila.
\newblock {Accurate Solution of Bayesian Inverse Uncertainty Quantification
  Problems Combining Reduced Basis Methods and Reduction Error Models}.
\newblock {\em SIAM/ASA Journal on Uncertainty Quantification}, 4(1):380--412,
  2016.

\bibitem{Mattis2018}
Steven~A. Mattis and Barbara Wohlmuth.
\newblock Goal-oriented adaptive surrogate construction for stochastic
  inversion.
\newblock {\em Computer Methods in Applied Mechanics and Engineering},
  339:36--60, 2018.

\bibitem{Minden2017}
V.~Minden, A.~Damle, K.~Ho, and L.~Ying.
\newblock {Fast Spatial Gaussian Process Maximum Likelihood Estimation via
  Skeletonization Factorizations}.
\newblock {\em Multiscale Modeling \& Simulation}, 15(4):1584--1611, 2017.

\bibitem{Noor1980}
A.~K. Noor and J.~M. Peters.
\newblock {Reduced Basis Technique for Nonlinear Analysis of Structures}.
\newblock {\em AIAA Journal}, 18(4):455--462, 1980.

\bibitem{Osbornetal:2017a}
Sarah Osborn, Panayot~S. Vassilevski, and Umberto Villa.
\newblock {A multilevel, hierarchical sampling technique for spatially
  correlated random fields}.
\newblock {\em {SIAM J. Sci. Comput.}}, 39(5):543--562, 2017.

\bibitem{Osb2018}
Sarah Osborn, Patrick Zulian, Thomas Benson, Umberto Villa, Rolf Krause, and
  Panayot~S Vassilevski.
\newblock {Scalable hierarchical PDE sampler for generating spatially
  correlated random fields using non-matching meshes}.
\newblock {\em Numerical Linear Algebra with Applications}, 25:e2146, 2018.

\bibitem{Peherstorfer2016}
B~Peherstorfer, K~Willcox, and M~Gunzburger.
\newblock {Optimal Model Management for Multifidelity Monte Carlo Estimation}.
\newblock {\em SIAM J. Sci. Comput.}, 38(5):A3163--A3194, 2016.

\bibitem{Peherstorfer2018}
B.~Peherstorfer, K.~Willcox, and M.~Gunzburger.
\newblock Survey of multifidelity methods in uncertainty propagation,
  inference, and optimization.
\newblock {\em SIAM Review}, 60(3):550--591, 2018.

\bibitem{PraneshGhosh:2015}
Srikara Pranesh and Debraj Ghosh.
\newblock Faster computation of the {K}arhunen-{L}o\`eve expansion using its
  domain independence property.
\newblock {\em Comput. Methods Appl. Mech. Engrg.}, 285:125--145, 2015.

\bibitem{Book1}
A~Quarteroni, A~Manzoni, and F~Negri.
\newblock {\em Reduced basis methods for partial differential equations},
  volume~92 of {\em Unitext}.
\newblock Springer, Cham, 2016.

\bibitem{Robert2007}
Christian~P. Robert.
\newblock {\em {The Bayesian Choice}}.
\newblock Springer, 2nd edition, 2007.

\bibitem{Robert2004}
Christian~P. Robert and George Casella.
\newblock {\em {Monte Carlo Statistical Methods}}.
\newblock Springer, 2004.

\bibitem{Roininen2016}
L~Roininen, M~Girolami, S~Lasanen, and M~Markkanen.
\newblock {Hyperpriors for Mat{\'{e}}rn fields with applications in Bayesian
  inversion}.
\newblock {\em Preprint}, arXiv:1612.02989, 2016.

\bibitem{Roininenetal:2018}
L~Roininen, S~Lasanen, M~Orisp\"a\"a, and S~S\"arkk\"a.
\newblock {Sparse approximations of Fractional Mat\'ern fields}.
\newblock {\em Scandinavian Journal of Statistics}, 45:194--216, 2018.

\bibitem{Rubio2018}
Paul-Baptiste Rubio, Fran\c{c}ois Louf, and Ludovic Chamoin.
\newblock {Fast model updating coupling Bayesian inference and PGD model
  reduction}.
\newblock {\em Computational Mechanics}, 62:1485, 2018.

\bibitem{Saibabaetal:2016}
Arvind~K. {Saibaba}, Jonghyun {Lee}, and Peter~K. {Kitanidis}.
\newblock {Randomized algorithms for generalized Hermitian eigenvalue problems
  with application to computing Karhunen-Lo\`eve expansion.}
\newblock {\em {Numer. Linear Algebra Appl.}}, 23(2):314--339, 2016.

\bibitem{SchwabTodor:2006}
Christoph Schwab and Radu~Alexandru Todor.
\newblock {Karhunen-Lo{\`{e}}ve approximation of random fields by generalized
  fast multipole methods}.
\newblock {\em Journal of Computational Physics}, 217(1):100--122, 2006.

\bibitem{Sirkovic2016}
Petar Sirkovi\'c.
\newblock {\em {Low-rank methods for parameter-dependent eigenvalue problems
  and matrix equations}}.
\newblock PhD thesis, {\'Ecole Polytechnique F\'ed\'erale de Lausanne}, 2016.

\bibitem{Smith2014}
Ralph~C. Smith.
\newblock {\em {Uncertainty Quantification: Theory, Implementation, and
  Applications}}.
\newblock Society for Industrial and Applied Mathematics, 2014.

\bibitem{Sraj2016}
Ihab Sraj, Olivier~P. {Le Ma{\^{i}}tre}, Omar~M. Knio, and Ibrahim Hoteit.
\newblock {Coordinate Transformation and Polynomial Chaos for the Bayesian
  Inference of a Gaussian Process with Parametrized Prior Covariance Function}.
\newblock {\em Computer Methods in Applied Mechanics and Engineering},
  298:205--228, 2016.

\bibitem{Stuart2018}
A.M. Stuart and A.L. Teckentrup.
\newblock {Posterior consistency for Gaussian process approximations of
  Bayesian posterior distributions}.
\newblock {\em Mathematics of Computation}, 87:721--753, 2018.

\bibitem{Stuart2010}
Andrew~M. Stuart.
\newblock {Inverse problems: A Bayesian perspective}.
\newblock In {\em Acta Numerica}, volume~19, pages 451--559. Cambridge
  University Press, 2010.

\bibitem{TagadeChoi:2014}
Piyush~M. Tagade and Han-Lim Choi.
\newblock A generalized polynomial chaos-based method for efficient {B}ayesian
  calibration of uncertain computational models.
\newblock {\em Inverse Probl. Sci. Eng.}, 22(4):602--624, 2014.

\bibitem{Vallaghe2015}
Sylvain Vallagh{\'e}, Phuong Huynh, David~J. Knezevic, Loi Nguyen, and
  Anthony~T. Patera.
\newblock Component-based reduced basis for parametrized symmetric
  eigenproblems.
\newblock {\em Advanced Modeling and Simulation in Engineering Sciences},
  2(1):7, 2015.

\bibitem{VanDerwerken2013}
D.~N. VanDerwerken and S.~C. Schmidler.
\newblock {Parallel Markov Chain Monte Carlo}.
\newblock {\em Preprint}, arXiv:1312.7479, 2013.

\bibitem{Vapnik1971}
V.~N. Vapnik and A.~Ya. Chervonenkis.
\newblock {On the Uniform Convergence of Relative Frequencies of Events to
  Their Probabilities}.
\newblock {\em Theory of Probability {\&} Its Applications}, 16(2):264--280,
  1971.

\bibitem{Wikle2017}
Christopher~K. Wikle.
\newblock {Hierarchical Models for Uncertainty Quantification}.
\newblock In Roger Ghanem, David Higdon, and Houman Owhadi, editors, {\em
  Handbook of Uncertainty Quantification}, pages 193--218. Springer, 2017.

\bibitem{ZhengDai:2017}
Zhibao Zheng and Hongzhe Dai.
\newblock Simulation of multi-dimensional random fields by {K}arhunen-{L}o\`eve
  expansion.
\newblock {\em Comput. Methods Appl. Mech. Engrg.}, 324:221--247, 2017.

\end{thebibliography}





\end{document}